\documentclass[letterpaper, 11pt,  reqno]{amsart}
\usepackage[margin=1.2in,marginparwidth=1.5cm, marginparsep=0.5cm]{geometry}

\usepackage{amsmath,amssymb,amscd,amsthm,amsxtra, esint, xcolor}
\usepackage[implicit=true]{hyperref}

\usepackage[shortlabels]{enumitem}

\allowdisplaybreaks[2]

\usepackage{color}

\definecolor{gr}{rgb}   {0.,   0.69,   0.23 }
\definecolor{bl}{rgb}   {0.,   0.5,   1. }
\definecolor{mg}{rgb}   {0.85,  0.,    0.85}
\definecolor{yl}{rgb}   {0.8,  0.7,   0.}
\definecolor{or}{rgb}  {0.7,0.2,0.2}

\newtheorem{theorem}{Theorem} [section]

\newtheorem{lemma}[theorem]{Lemma}
\newtheorem{proposition}[theorem]{Proposition}
\newtheorem{remark}[theorem]{Remark}

\newtheorem{corollary}[theorem]{Corollary}

\newtheorem{conjecture}[theorem]{Conjecture}

\DeclareMathOperator*{\intt}{\int}

\DeclareMathOperator*{\supp}{supp}
\DeclareMathOperator{\med}{med}

\DeclareMathOperator{\MAX}{MAX}

\newcommand{\I}{\hspace{0.5mm}\text{I}\hspace{0.5mm}}

\newcommand{\II}{\text{I \hspace{-2.8mm} I} }
\newcommand{\III}{\text{I \hspace{-2.9mm} I \hspace{-2.9mm} I}}

\newcommand{\IV}{\text{I \hspace{-2.8mm} V} }

\newcommand{\noi}{\noindent}
\newcommand{\Z}{\mathbb{Z}}
\newcommand{\R}{\mathbb{R}}
\newcommand{\C}{\mathbb{C}}
\newcommand{\T}{\mathbb{T}}

\newcommand{\hi}{\textup{hi}}
\newcommand{\HI}{\textup{HI}}

\let\Im=\undefined\DeclareMathOperator*{\Im}{Im}

\let\P= \undefined
\newcommand{\P}{\mathbf{P}}

\newcommand{\Q}{\mathbf{Q}}

\renewcommand{\L}{\mathcal{L}}

\newcommand{\RR}{\mathcal{R}}

\newcommand{\J}{\mathcal{J}}

\newcommand{\GG}{\mathcal{G}}

\newcommand{\F}{\mathcal{F}}

\newcommand{\be}{\beta}
\newcommand{\dl}{\delta}

\newcommand{\eps}{\varepsilon}
\newcommand{\kk}{\kappa}
\newcommand{\g}{\gamma}
\newcommand{\G}{\Gamma}

\newcommand{\Ld}{\Lambda}
\newcommand{\s}{\sigma}

\newcommand{\ft}{\widehat}
\newcommand{\Ft}{{\mathcal{F}}}
\newcommand{\wt}{\widetilde}
\newcommand{\cj}{\overline}
\newcommand{\dx}{\partial_x}

\newcommand{\dt}{\partial_t}

\newcommand{\embeds}{\hookrightarrow}

\newcommand{\ta}{\theta}

\newcommand{\uu}{\mathbf{u}}
\newcommand{\vv}{\mathbf{v}}
\newcommand{\ww}{\mathbf{w}}
\newcommand{\UU}{\mathbf{U}}
\newcommand{\VV}{\mathbf{V}}
\newcommand{\WW}{\mathbf{W}}

\renewcommand{\l}{\ell}

\newcommand{\les}{\lesssim}
\newcommand{\ges}{\gtrsim}

\newcommand{\jb}[1]
{\langle #1 \rangle}

\renewcommand{\b}{\beta}
\newcommand{\ind}{\mathbf 1}

\renewcommand{\S}{\mathcal{S}}

\newcommand{\M}{\mathcal{M}}

\newcommand{\N}{\mathbb{N}}
\newcommand{\NN}{\mathcal{N}}

\newtheorem*{ackno}{Acknowledgements}

\renewcommand{\H}{\mathcal{H}}

\def\sgn{\textup{sgn}}

\newcommand{\Pbhi}{\mathbf{P}_{\textup{hi}} }
\newcommand{\Pblo}{\mathbf{P}_{\textup{lo}} }
\newcommand{\Pbhip}{\mathbf{P}_{\textup{+,hi}} }
\newcommand{\PbHIp}{\mathbf{P}_{\textup{+,HI}} }
\newcommand{\PbHI}{\mathbf{P}_{\textup{HI}}}
\newcommand{\PbLO}{\mathbf{P}_{\textup{LO}}}

\newcommand{\Id}{\textup{Id}}

\newcommand{\lax}{\mathcal{L}}
\newcommand{\peter}{\mathcal{P}}
\newcommand{\op}{\textup{op}}
\newcommand{\Pih}{\Pi_{+,h}}
\newcommand{\TT}{\mathcal{T}}

\numberwithin{equation}{section}
\numberwithin{theorem}{section}

\makeatletter
\@namedef{subjclassname@2020}{\textup{2020} Mathematics Subject Classification}
\makeatother

\begin{document}
\baselineskip = 14pt

\title[Well-posedness for the periodic INLS]{Well-posedness for the periodic Intermediate nonlinear Schr\"{o}dinger equation}

\author[A.~Chapouto, J.~Forlano, T.~Laurens]
{Andreia Chapouto, Justin Forlano, Thierry Laurens}

\address{Andreia Chapouto,
CNRS, Laboratoire de math\'ematiques de Versailles, UVSQ, Universit\'e Paris-Saclay, CNRS, 45 avenue des \'Etats-Unis, 78035 Versailles Cedex, France, 
and 
School of Mathematics, Monash University, Clayton, VIC 3800, Australia}

\email{andreia.chapouto@monash.edu}

\address{Justin Forlano,
School of Mathematics, Monash University, Clayton, VIC 3800, Australia}

\email{justin.forlano@monash.edu}

\address{Thierry Laurens,
Department of Mathematics, University of Wisconsin--Madison, WI, 53706, USA}
\email{laurens@math.wisc.edu}



\begin{abstract}
We study the well-posedness for the intermediate nonlinear Schr\"{o}dinger equation (INLS) with periodic boundary conditions. Using a gauge transform, we obtain large data local well-posedness in $H^{s}(\T)$ for any $s\geq \frac 12$. We extend this result to global well-posedness under a small $L^2$-norm constraint by exploiting the complete integrability of the continuum Calogero-Moser equation (CCM). 
We also establish additional results such as the unconditional well-posedness in the energy space and the convergence of solutions to INLS to those of CCM in the infinite-depth limit.
\end{abstract}

\maketitle

\tableofcontents

\section{Introduction}

\subsection{The intermediate nonlinear Schr\"odinger equation}

We consider the Cauchy problem for the periodic intermediate nonlinear Schr\"{o}dinger equation~(INLS):
\begin{equation}
\left\{
\begin{aligned}
  & \dt u+i\dx^2 u = \be u (1+i\mathcal{T}_{h})\dx(|u|^2)+i\g |u|^2 u,\\
   & u|_{t=0}=u_0,
\end{aligned}
 \right. \label{INLS}
\end{equation}
where $u:\R\times \T\to \C$, $\T = \R/(2\pi \Z)$, $0<h<\infty$, and $\g,\be\in \R$.
Here, the operator $\mathcal T_h$ is a Fourier multiplier operator with symbol
\begin{align}
\ft{\mathcal{T}_h f}(\xi) = -i\coth(h \xi) \ft f(\xi), \quad \xi \in \Z\setminus\{0\},
\notag
\end{align}
and we define $\ft{\mathcal{T}_h f}(0)=0$.\footnote{We use the notation $\ft{f}$ for the Fourier transform of $f$; see \eqref{Fourier}. }
The sign of $\be$ in \eqref{INLS} determines the defocusing ($\be>0$) or focusing ($\be<0$) variety of this equation.

The equation \eqref{INLS} was proposed in \cite{Pel1, Pel2} as a model for the evolution of quasi-harmonic internal waves in a two-fluid layer system, where the bottom fluid is of a finite depth $h>0$, and with higher order effects accounted for by $\g\neq 0$. 
Additionally, \eqref{INLS} with $\g=0$ and on $\R$ has a rich structure: it is completely integrable \cite{Pel3, CFL}, possesses multi-soliton solutions \cite{Pel1, Tutiya}, and is closely related to two other equations of great interest, as we now discuss. 

Motivated by the observation that $\mathcal{T}_{h}$ formally converges to the Hilbert transform $\H$ as $h\to \infty$, i.e., that $-i\coth(h\xi) \to -i\sgn(\xi)$ as $h\to\infty$ for fixed $\xi\in \Z\setminus\{0\}$,
we rewrite \eqref{INLS} as
\begin{align}
\dt u+i\dx^2 u = \be u (1+i\H)\dx(|u|^2)-i\be u \GG_{h}(|u|^2)+i\g |u|^2 u , \label{INLS20}
\end{align}
where 
\begin{align}
\GG_{h} : = (\H - \mathcal{T}_{h})\dx. \label{Lh}
\end{align}
Now, \eqref{INLS20} admits further rewriting, for which we need to define some frequency projectors. Given a set $A\subseteq \Z$, we write $\ind_{A}$ for the characteristic function of the set $A$. We define $\P_{\pm}$ as the Fourier multiplier with symbol $\ind_{\{ \pm \xi>0\}}$, so that on $\T$, we have
\begin{align}
\P_{+}+\P_{0}+\P_{-}=\text{Id},
\notag
\end{align}
where $\P_{0}$ is the projection onto the zeroth Fourier mode:
\begin{align*}
\P_{0}f =\frac{1}{2\pi} \int_{\T} f(x) dx.
\end{align*}
We define $\P_{\neq 0} : = \text{Id}-\P_{0}$.
With these notations, the Hilbert transform is given by $\H= -i \P_{+}+i\P_{-}$ and thus $(1+i\H)=2\P_{+}+\P_{0}$.
Thus, \eqref{INLS2} becomes:
\begin{align}
\dt u+i\dx^2 u = 2\be u \P_{+}\dx(|u|^2)-i\be u \GG_{h}(|u|^2)+i\g |u|^2 u ,\label{INLS2}
\end{align}
 since $\P_{0}\dx =0$ on $\T$. 
 Now, formally taking the \textit{infinite depth limit} $h\to\infty$ of \eqref{INLS2}, we find 
 \begin{align}
\dt u+i\dx^2 u = 2\be u \P_{+}\dx(|u|^2)+i\g |u|^2 u .\label{CCM1}
\end{align}
It is then natural to define $\GG_{\infty}:=0$ so that we may speak of \eqref{INLS2} for all $0<h\leq \infty$.
In the special case $\g=0$, we obtain the continuum Calogero-Moser equation (CCM)
\begin{align}
\dt u+i\dx^2 u = 2\be u \P_{+}\dx(|u|^2),  \label{CCM}
\end{align}
which is  also known as the Calogero-Sutherland derivative nonlinear Schr\"{o}dinger equation.
  On $\R$, the (focusing) CCM arises as a suitable thermodynamic limit of the $N$-body Calogero-Sutherland-Moser system \cite{Abanov}.

 From the well-posedness viewpoint, scaling considerations identify $L^2$ as a critical space for \eqref{CCM}. This suggests that, within the scope of $L^2$-based Sobolev space $H^s$, it is necessary to have $s\geq 0$ in order for \eqref{CCM} to be well-posed, and we expect the same to hold true for \eqref{INLS2}. 
Some evidence for this is that the solution map for \eqref{INLS2}, if it exists, cannot be $C^3$ at the origin in $H^{s}(\R)$ for any $s<0$ \cite{PdM}.
The investigations of well-posedness of \eqref{CCM} have been accelerated by the observation that it preserves the Hardy space $L^{2}_{+}$, which is defined~as 
 \begin{align}
    L^2_+  
    &=
    \big\{
    f \in L^2 : \ \supp \ft f \subset [0,\infty)
    \big\} .
    \notag
\end{align}
In particular, global well-posedness for large (resp. small) data in the defocusing (resp. focusing) setting has recently been established in the Hardy space $L^2_{+}$ on $\T$ \cite{Rana1} and on $\R$ \cite{KLV2}. This was revisited in \cite{KMarV} on both geometries using the method of commuting flows; see \cite{PdM, PMP, BdMS, GL} for prior works.  
The results \cite{Rana1, KLV2} both exploit the complete integrability of \eqref{CCM} through an explicit formula for solutions in the Hardy space, similar to that for the Benjamin-Ono equation (BO) discovered by Gérard~\cite{Gerard}. 
The explicit formulae have proven very useful beyond well-posedness, such as in the study of travelling waves \cite{Rana3}, zero dispersion limits \cite{Rana2}, scattering with non-vanishing background \cite{Chen}, soliton resolution for BO \cite{GGM}, numerical schemes \cite{ABL}, and blow-up solutions \cite{ChenLenz}.
Unfortunately, it is not known if such an explicit formula holds for \eqref{CCM} outside of the Hardy space, and if at all for \eqref{INLS}. 
Thus, in order to study \eqref{INLS}, we pursue an approach which does not rely on complete integrability.

For CCM outside of the Hardy space and 
for \eqref{INLS}, less is understood regarding the low-regularity well-posedness. On $\R$, using dispersive properties of the linear part of \eqref{INLS} and a gauge transform, de Moura-Pilod \cite{PMP} proved local well-posedness for \eqref{INLS} in $H^{s}(\R)$ for~$s>\frac 12$. See also \cite{demoura1, PdM, BdMS}.
Very recently, in \cite{CFL},  
we lowered this regularity condition to $s>\frac 14$. 
In contrast, key techniques employed in the line setting do not extend to the periodic setting, as they are based on smoothing effects (e.g.\ local smoothing and bilinear Strichartz estimates), which are absent on $\T$. 
One can still apply the standard energy method on $\T$, to cover $H^{s}(\T)$ for $s>\frac{3}{2}$; see for instance \cite[Section 5]{PMP} which can be adapted to the circle case. However, going beyond this result requires a different approach, which is the main goal of this work.

Before we move onto discussing our results, we want to point out that for the long-time behaviour of energy class solutions for CCM on $\R$, even outside of the Hardy space, there have been remarkable results regarding soliton resolution \cite{KKK2}, and blow-up solutions and classification of blow-up rates \cite{HoganKowalski, KKK1, JeongKim, JKKK}. In the periodic setting, traveling waves have been studied in \cite{Rana3,ChenPel,Matsuno1,Matsuno7,Matsuno23}.
Another approach to understanding the long-time qualitative behaviour in the periodic case would be to study \eqref{INLS} in statistical equilibrium; namely, to construct strong solutions on the support of the associated invariant Gibbs measure and apply the Poincar\'e recurrence theorem. 
For \eqref{INLS}, the samples of the Gibbs measure would be supported in $H^{s}(\T)\setminus H^{\frac 12}(\T)$ for any $s<\frac 12$ almost surely, which is just beyond the results we can prove here and we leave this interesting question to a future work.

\subsection{Main results}

Our first main result is on the local well-posedness of \eqref{INLS} and its unconditional uniqueness.

\begin{theorem}\label{THM:LWP}
Let $s\geq \frac 12$, $0<h\leq \infty$, and $\be,\g\in\R$. Then, \text{INLS}~\eqref{INLS} is locally well-posed in $H^{s}(\T)$. More precisely, for any $u_0 \in H^{s}(\T)$, there exist $T_{h}=T(\|u_0\|_{H^s},h)>0$ and a unique solution 
\begin{align}
u \in  Z^{s}_{T_{h}}: =C([-T_h,T_h];H^{s}(\T))\cap L^4_{T_h} W^{s,4}_{x} \cap X^{s-1,1}_{T_h} 
\label{uclass}
\end{align}
to \eqref{INLS} in $[-T_h,T_h]$ with the property that
\begin{align}
\wt{v} (t,x) := \Pbhip[ e^{i\be F[u]}u](t,x-2\be \P_{0}(|u_0|^2)t) \in X^{s,\frac 12+}_{T_h} \quad \text{and} \quad w:= \P_{-,\textup{hi}}u \in X^{s,\frac 12+}_{T_h}, \label{gaugeclass}
\end{align} 
where $F$ is the primitive of $\P_{\neq 0}(|u|^2)$; see \eqref{F}, and the spaces $X^{s,b}_{T}$ are defined in Section~\ref{SEC:Xsb}.
Here $T_h>0$ can be chosen uniformly in $1\le h \le \infty$. 
Furthermore, the solution map $u_0 \mapsto u$ is continuous from $H^{s}(\T)$ to $Z^{s}_{T_h}$
and the solution $u$ is unique in the class:\\ 
\noi
\textup{(i)} $u\in L^{\infty}((0,T_h);H^{\frac 12}(\T))\cap L^{4}((0,T_h); W^{\frac 12,4}_{x}(\T))$ with $\wt{v},w\in X^{\frac 12,\frac 12+}_{T_h}$; 
\\
\textup{(ii)} $u\in L^{\infty}((0,T_h);H^{s}(\T)) \cap L^{4}((0,T_h);W^{s,4}_x (\T))$ if $s>\frac{3}{5}$;
\\
\textup{(iii)} $u\in L^{\infty}((0,T_h);H^{s}(\T))$ if $s>\frac{17}{20}$. 
\end{theorem}

Theorem~\ref{THM:LWP} establishes the local well-posedness of \eqref{INLS2} in $H^{s}(\T)$ for any $s\geq \frac 12$, which is more than a whole derivative of improvement compared to the previous best result of $s>\frac 32$~\cite{PMP}. 
In proving Theorem~\ref{THM:LWP}, we employ Fourier analytic arguments using a frequency-localized gauge transform which weakens the problematic $\text{low}\times\text{high}\to \text{high}$ interactions in the nonlinear term $u \P_{+}\dx(|u|^2)$.  The gauge transform is a modification to the periodic setting of the one used on the line \cite{PMP, CFL}, which itself goes back to the gauge transform introduced by Tao~\cite{TAO04} for BO. 
The overall structure of the proof follows that of Molinet-Pilod~\cite{MP} in their study of BO in $L^2$, involving a bootstrap argument and the Fourier restriction norm spaces of Bourgain \cite{BO93}. These spaces are well-adapted to the periodic setting allowing us to gain additional smoothing through multilinear dispersion, which is the key for the main trilinear estimates in Proposition~\ref{PROP:tri}. 
However, we encounter some new fundamental difficulties, absent in both \cite{MP} and in the setting of the line $\R$ from our previous work \cite{CFL}. We postpone this discussion to Subsection~\ref{SEC:intro-method}.

We point out that Theorem~\ref{THM:LWP}~(iii) states  that uniqueness holds \textit{unconditionally} in $L^{\infty}_{T_h}H^s(\T)$ for any $s>\frac{17}{20}$; namely, 
without intersecting with any auxiliary function space.
The concept of unconditional uniqueness goes back to Kato~\cite{Kato} and it allows to compare $L^{\infty}_{T}H^{s}(\T)$-solutions irrespective of how they were constructed. In comparison, the uniqueness statements in Theorem~\ref{THM:LWP}~(i)-(ii), when $\frac 12 \leq s \leq \frac{17}{20}$, involve more restrictive classes. The unconditional uniqueness issue has been studied by many authors in the past decades with various methods developed. See for instance \cite{Zhou, GuoKO, CH, K19} for a sampling of results.

We do not expect the range $s>\frac{17}{20}$ to be sharp for unconditional uniqueness. Indeed, the ideal range occurs whenever $u\in L^{\infty}_{T}H^{s}(\T)$ is a distributional solution to \eqref{INLS}, then the nonlinearity is well-defined as a distribution in space-time. For \eqref{INLS}, or more specifically, the rewritten version \eqref{INLS2}, we need to make sense of $u\P_{+}\dx(|u|^2)$ and $|u|^2 u$. The former is clearly more restrictive and requires $s>\frac 12$, which is a consequence of the Moser estimate:
\begin{align}
\| f g\|_{H^{s}} \leq \|f\|_{L^{\infty}}\|g\|_{H^{s}} + \|f\|_{H^s} \|g\|_{L^{\infty}} \quad \text{for} \quad s\geq 0.
\label{moser}
\end{align}
See Remark \ref{RMK:UUDNLS} below on the endpoint $s=\frac 12$. 
Remarkably, the situation is vastly improved for the CCM equation \eqref{CCM} in the Hardy space. Indeed, we have the estimate 
\begin{align}
\| (\P_+ f_1) \P_{+}\dx ( \cj{f_2} f_3)\|_{H^{-3}} \les  \|f_1\|_{L^2} \|f_2\|_{L^2} \|f_3\|_{L^2}, 
\label{UU+}
\end{align}
which suggests that unconditional uniqueness of solutions to \eqref{CCM} in $L^\infty_T L^{2}_{+}(\T)$ may hold. This would be especially interesting as it corresponds to the scaling critical regularity. 

\begin{conjecture}\label{CONJ:UU}
The defocusing CCM equation \eqref{CCM} is unconditionally globally well-posed in $L^{2}_{+}(\T)$ and $L^{2}_{+}(\R)$.
\end{conjecture}

The unconditional uniqueness result in Theorem~\ref{THM:LWP} also applies when restricted to CCM in the Hardy space and presents the first progress towards Conjecture~\ref{CONJ:UU}.

Our interest in our range for unconditional uniqueness in Theorem~\ref{THM:LWP} is that it covers the energy space at $s=1$. Indeed, the equation \eqref{INLS} has the conservation law:
\begin{align}
E(u) & = \int_{\T}
\big\{ 
|\dx u|^2 - \be |u|^2 \text{Im}(u\dx \cj{u}) + \tfrac{\be}{2} |u|^2 \mathcal{T}_h \dx (|u|^2)  
+ \tfrac\g2 |u|^4 + \tfrac{\be^2}{3} |u|^6
\Big\} dx
.
 \label{Energy} 
\end{align}
Although the second term in \eqref{Energy} is not sign-definite, we can still use $E(u)$ to obtain a priori bounds in $H^1(\T)$ for solutions to \eqref{INLS}, which combined with Theorem~\ref{THM:LWP} lead to the following result.

\begin{corollary}\label{COR:gwp}
Let $\g \in \R$, $r>0$, and $B_{r}(0)=\{ f\in L^2(\T) \, :\, \|f\|_{L^2(\T)} \leq r\}$.
Then, the defocusing \eqref{INLS} is unconditionally globally well-posed in $H^{1}(\T)$, and the focusing \eqref{INLS} is unconditionally globally well-posed in $H^{1}(\T)\cap B_{r}(0)$
for some $r>0$ sufficiently small.
\end{corollary}

We point out the bounds we obtain to prove Corollary~\ref{COR:gwp} do not depend upon the sign of $\g$ (and only on the sign of $\be$). This seems to be the first time that this has been observed in the literature \cite{demoura1, BdMS}, where $\g\geq 0$ was previously imposed. 
The small $L^2$-restriction in the focusing case has very recently been shown to be necessary for the focusing CCM \cite{ChenLenz}, as otherwise, blowup can occur. 

The next natural question we consider is the globalization of solutions \textit{below} the energy space. It appears again that there is a dichotomy in the behaviour between even CCM \eqref{CCM} in and outside of the Hardy space. Let us consider the defocusing case $(\be >0)$. Using the complete integrability of CCM in $L^{2}_{+}(\T)$, Baddredine~\cite{Rana1}  obtained a priori bounds for solutions in $H^{s}(\T)$ for any $0\leq s\leq 1$. Outside of the Hardy space, however, it appears that the conservation laws lose their efficacy. For example, the momentum
\begin{align}
P(u) = \int_{\T} \text{Im}(\cj u \dx {u}) + \tfrac{\be}{2}|u|^4 dx
\label{momentum-intro}
\end{align} 

\noi
is conserved for \eqref{CCM} (and also for solutions to \eqref{INLS} with $\g\in \R$) but alas the first term in \eqref{momentum-intro} is not coercive. However, in the Hardy space, it can be written as
\begin{align*}
 \int_{\T} \text{Im}(\cj u \dx {u}) dx = \int_{\T} | |\dx|^{\frac 12}u|^2 dx
\end{align*}
which is now coercive and, consequently,  
$P(u)$ gives control over the $\dot{H}^{\frac 12}(\T)$ norm of $u$.
Despite the apparent lack of coercivity, we showed in \cite{CFL} that INLS \eqref{INLS} on $\R$ with $\g=0$ has a Lax pair \emph{outside} of the Hardy space which can be used to obtain a priori bounds in $H^{s}(\R)$ for any $s\geq \frac 14$, under a small mass constraint (even in the defocusing case). 

Returning to our current setting on $\T$, the straightforward approach of mimicking the line case, namely, taking the same pair of operators $( \lax_{u;h}, \peter_{u;h})$, modulo trivial redefinition of $\TT_h$, and showing that they form a Lax pair on $\T$ \emph{fails}. 
In fact, $( \lax_{u;h}, \peter_{u;h})$ fail to be a Lax pair due to the presence of zeroth Fourier modes on $\T$; see Proposition~\ref{PROP:Laxpair}.
This failure stems from the Cotlar identity \eqref{Tilb22} for $\mathcal{T}_h$ on $\T$ being slightly different to that on $\R$, requiring mean-zero assumptions on the functions. 
Note that the mean $\P_{0}u$ is not preserved for solutions to \eqref{INLS}, so we cannot restrict to mean zero functions to avoid this issue.

Nonetheless, we are able to obtain low-regularity a priori bounds for solutions to \eqref{INLS} with $\g=0$ by using \textit{only} the complete integrability of the CCM equation \eqref{CCM}. 
This allows us to extend Corollary~\ref{COR:gwp} as follows. 

 \begin{theorem}\label{THM:GWP}
 Let $\frac 12\leq s<1$, $0<h \leq \infty$, $\be\in\R$, and $\g=0$. Then, there exists $r>0$ such that INLS~\eqref{INLS} is globally well-posed in $H^{s}(\T)\cap B_{r}(0)$.
\end{theorem}

Similar to our previous work on $\R$ \cite{CFL}, we need to impose a small $L^{2}$-norm assumption even in the defocusing case to run a bootstrap argument; see Theorem~\ref{THM:apriori}. This seems to go back to the $L^2$-criticality of \eqref{INLS}. We currently do not know if this can be removed in the defocusing case.
Additionally, we have to make the stronger regularity assumption $s\geq \frac 12$ (compared to $s>\frac14$ on $\R$), which we now explain.
Our argument is inspired by the strategy for obtaining a priori bounds for perturbations of completely integrable equations in \cite{CFLOP-2}, where the authors obtained a priori bounds for dispersive perturbations of the BO equation. See also \cite{GassotL, IL}. The main idea there is to define some quantity which is conserved for the unperturbed equation and is equivalent to a fractional Sobolev norm $H^{s}(\T)$, and then use a Gronwall argument to handle the errors introduced due to the lack of a precise conservation. 
In \cite{CFLOP-2}, the error terms are of the form $\mathcal{G}_{h} u$ and are thus \emph{smooth} and \emph{linear} in $u$. 
In our case of \eqref{INLS2} with $\g=0$, the error term is $-i\be u \mathcal{G}_{h}(|u|^2)$, which is neither smooth nor linear! In order to keep the multiplicity of $u$-terms in check for the Gronwall argument, we still rely on the smoothing of $\mathcal{G}_h$ to send more terms to the conserved $L^2$ space. This also explains the restriction to the $\g=0$ setting: a lack of smoothing means that we cannot control the multiplicity any longer for the cubic term $i\g |u|^2 u$.

The conserved quantity we consider is motivated by observations and arguments in \cite{KLV2}, ultimately going back to the fact that there is a Lax pair $(\lax_{u;\infty}, \peter_{u;\infty})$ (see Proposition~\ref{PROP:Laxpair}) for \eqref{CCM} such that \eqref{CCM} is equivalently written as $\dt u  =\mathcal{P}_{u;\infty}u$.
This structure combined with Loewner's theorem makes the quantities $\jb{u, (\L_{u;\infty} +1)^{s} u}$ particularly convenient for establishing low-regularity a priori bounds in $H^s(\T)\cap L^2_+(\T)$. 
Outside of the Hardy space, $\L_{u;\infty}$ is no longer positive definite, so we replace this by $\jb{u, (\L^2_{u;\infty} +1)^{\frac{s}{2}} u}$. The ensuing Gronwall argument becomes much more involved than in \cite{CFLOP-2}.  Ultimately, these reasons are the source of the regularity restriction $s\geq \frac 12$.

Lastly, our final main result concerns the infinite depth limit $h\to \infty$ of solutions to \eqref{INLS} to those of \eqref{CCM1}. For simplicity, we state the result when $\g=0$ and globally-in-time. The local-in-time convergence as $h\to \infty$ for any $\g\in \R$ is a consequence of the estimates we proved in the course of establishing Theorem~\ref{THM:LWP}. See Section~\ref{SEC:conv}.

  \begin{theorem}\label{THM:convergence}
Fix $s\geq \frac 12$, $\be\in \R$, $\g=0$, and $r>0$ as in Theorem~\ref{THM:GWP}. Let $u_0\in H^{s}(\T)\cap B_{r}(0)$ and $\{u_{0,h}\}_{1 \le h <\infty} \subset H^s(\T) \cap B_r(0)$ be a net such that $u_{0,h} \to u_0$ in $H^s(\T)$ as $h\to\infty$. Also, let $u_{\infty}$ and $u_h$ denote the global solutions to \eqref{CCM} and \eqref{INLS}, respectively, with $u_{\infty}\vert_{t=0}=u_0$ and $u_{h} \vert_{t=0} = u_{0,h}$ constructed in Theorem~\ref{THM:GWP}. Then, $u_{h}$ converges to $u_{\infty}$ as $h\to \infty$ in $C(\R;H^{s}(\T))$.\footnote{We endow $C(\R;H^{s}(\T))$ with the compact-open topology in time.}
In particular, when $s>\frac{17}{20}$, the convergence occurs unconditionally: the uniqueness of both the CCM \eqref{CCM} and INLS \eqref{INLS} solutions holds in the entire class $C(\R;H^{s}(\T))$.
\end{theorem}

The notion of unconditional convergence property stated in Theorem~\ref{THM:convergence} was first introduced in \cite{FLZ}.\footnote{In the context of the intermediate long wave equation converging to BO in the deep-water limit. See \cite{Gli, CLOP}.} The significance of this is that the convergence happens irrespective of how the approximations $\{u_{h}\}_{1\leq h<\infty}$ or the limit $u_{\infty}$ were constructed.

\subsection{On the proof of Theorem~\ref{THM:LWP}}\label{SEC:intro-method}

The proof of local well-posedness in Theorem~\ref{THM:LWP} follows the strategy in \cite{Moli1, Moli2, MP, CFL}. 
However, we encounter some new difficulties, absent in both \cite{MP} and from our previous work in the line setting \cite{CFL}. Moreover, these difficulties are entirely due to working \textit{outside} of the Hardy space and do not appear for CCM in the Hardy space. Combined with our previous work \cite{CFL}, this again highlights the difference in behaviour of CCM in and outside of the Hardy space.

Firstly, given a solution $u$ to \eqref{INLS},
we work with the gauged variables $(\wt v, w)$ defined in \eqref{gaugeclass}, instead of the pair $(v,w)$ where $v(t,x) :=\Pbhip[ e^{i\be F[u]}u](t,x)$, used on the line \cite{CFL}. 
The pair $(\wt v, w)$ is more suitable for the periodic study, as we clarify in the following. 
Examining the equations satisfied by $(v,w)$  (see \eqref{veq}-\eqref{weq}), 
 we note that there are additional terms on the right-hand side when compared to the Euclidean setting \cite[(3.14)-(3.15)]{CFL}, due to $F$ being the mean-zero primitive of $|u|^2$; see \eqref{F}. 
 Some of these are harmless, but 
 the last linear derivative term in \eqref{veq} for $v$, $ 2 \be \P_{0}(|u|^2) \dx v$, creates an obstruction for the well-posedness, as its linearity (in $v$) means that we cannot absorb the derivative loss. However,  the~mass
\begin{align}
M(u) : = 2 \pi \P_{0}(|u|^2)  = \int_\T |u|^2 dx 
\label{mass}
\end{align}
is conserved for \eqref{INLS}. Consequently, it is possible to remove this bad term through the use of the gauge transform
\begin{align}\label{TT}
\wt{v}(t,x): =\J_u [v] (t,x) : = v  \big( t,  x - \tfrac\be\pi M(u) t \big),
\end{align}
which explains the presence of the term $\wt{v}$ in \eqref{gaugeclass}, as a replacement for $v$. As the gauged equations \eqref{veq2b} and \eqref{weq} for $(\wt v, w)$ are not closed, we are forced to work with the variable quartet $(u, \wt u, \wt v, w)$, where $\wt u := \J_u[u]$. Although it may seem more natural to work with $(\wt u, \wt v, \wt w)$, this is also not suitable for \eqref{INLS} on the torus, as we explain below.

In the simpler setting of \eqref{INLS} with $\g=0$ and only positive frequencies, namely CCM \eqref{CCM} in the Hardy space, we have $w \equiv 0$.
 It is then possible to apply $\J_u$ at a ``global level": namely, given a smooth solution $u$ to \eqref{CCM}, it suffices to run the local well-posedness argument for~$(\wt{u},\wt{v})$. Here, our arguments yield that the solution map $u_0 \mapsto \wt{u}$ is locally Lipschitz continuous in $Z^{s}_{T}$; see \eqref{uclass}. 
It is then the inversion of $\J_u$ that causes the solution map $u_0 \mapsto u$ to be merely continuous.

In the more general case of INLS \eqref{INLS}, or \eqref{CCM} outside of the Hardy space, 
$w$  no longer vanishes. 
One could consider using the gauge $\J_u$ at a ``global level'' and work with the trio $(\wt u , \wt v, \wt w)$ instead. Unfortunately, while the change of variable $v \mapsto \wt v$ removes the term $ 2 \be \P_{0}(|u|^2) \dx v$ from the $v$-equation, the change of variable $w \mapsto \wt w$ \emph{introduces} an analogous dangerous term in the $\wt w$-equation!
Hence, the variable \emph{quartet} $(u,\wt{u}, \wt{v}, w)$ is the appropriate choice for our analysis.

The variable $\wt u$ plays only an auxiliary role in showing the relevant a priori bounds, as our final estimates are only in terms of $(u, \wt v, w)$; see Subsection~\ref{SEC:apriori}. Also, some estimates for $u$, such as Lemma~\ref{LEM:uinfo}, rely on the recovery formula \eqref{PbHIu} which involves $(v,w)$, not $(\wt v, w)$, but this is not problematic as the the $L^{\infty}_{T}H^{s}_x$- and $L^{4}_{T}W^{s,4}_x$-norms remain invariant under $\J_u$. However, the inability to completely remove $\wt u$ from the analysis creates a new difficulty in estimating the \emph{difference} of solutions.

When considering the difference of solutions, we now need to invert the gauge $\J_u$ at the ``local level", namely, within our construction of low-regularity solutions, and not afterwards. Thus, our difference estimates in Lemma~\ref{LEM:diffests} become more involved.
Indeed, given two solutions $u_1,u_2$ to \eqref{INLS}, we have
\begin{align*}
\wt v_1 -\wt v_2  = \J_{u_1}[v_1] - \J_{u_2}[v_2] =    \J_{u_1}[v_1-v_2] +  \J_{u_1}[v_2] - \J_{u_2}[v_2] , 
\end{align*}
where the two final terms only differ in the spatial translation parameters depending on the potentially different masses; recall \eqref{TT}. 
In order to see a 
genuine difference $u_1-u_2$ to run the bootstrap argument later on (Section~\ref{SEC:diffs}), we would need to lose derivatives to obtain local Lipschitz continuity of the map $u\mapsto \J_{u}[v]$.\footnote{Recall that on the spatial Fourier side, $\mathcal{F}_{x}\{\J_{u}[v]\}(t,\xi) = e^{i\xi \tfrac{\be}{\pi}M(u)t} \ft v(t,\xi)$.}
This is untenable. 
Instead, we proceed via frequency truncation, writing 
\begin{align*}
 \J_{u_1}[v_2] - \J_{u_2}[v_2] = \Pi_{\leq \Ld} \big(  \J_{u_1}[v_2] - \J_{u_2}[v_2] \big) + \Pi_{>\Ld} \big( \J_{u_1}[v_2] -  \J_{u_2}[v_2] \big),
\end{align*}
where $\Pi_{\leq \Ld}$ is a Fourier projector to frequencies $\{|\xi|\leq \Ld\}$, $\Ld\in \N$, and $\Pi_{>\Ld}= \Id - \Pi_{\leq \Ld}$.
For the low-frequency part, we obtain a difference estimate due to the Lipschitz continuity of $\J_u$, at the cost of 
losing powers of $t\Ld$.  
For the high-frequency part, 
we dispense with the need for a genuine difference by proving an equicontinuity-type result for $\Pi_{>\Ld}\wt{v}$ in the relevant Fourier restriction norms $X^{s,\frac 12+}_{T}$; see Proposition~\ref{PROP:tight}. 

\smallskip

To show this equicontinuity result (Proposition~\ref{PROP:tight}), one would hope to exhibit some nonlinear smoothing effect in order to gain negative powers of $\Ld$. However, this is not possible for the first and last contributions in \eqref{veq2b}. 
Remarkably,  due to the sign structure of the main trilinear term $\Pi_{>\Ld}\P_{+}[\wt{v}\, \P_{-}\dx(|\wt{u}|^2)]$,  the exact same operator $\Pi_{>\Ld}$ can be added to the $\wt{v}$ factor for free, and the main trilinear estimate \eqref{vX1} then applies to give a small power of $T$ and a factor of $\| \Pi_{>\Ld}\wt{v}\|_{X^{s,\frac 12+}_{T}}$, which can be ``hidden'' on the right-hand side of the estimate. The same applies to the last term in \eqref{veq2b}, as it is linear in $\wt v$.  This circumvents any need for nonlinear smoothing for these terms.

For the remaining cubic terms in \eqref{veq2b}, we either observe \emph{nonlinear smoothing} away from their resonances or recover factors of both $\Pi_{>\Ld} \wt v$ and $T$; see Lemma~\ref{LEM:tightv}. Although, nonlinear smoothing effects for the periodic cubic Schr\"odinger equation are well-known \cite[Section 4.1]{ErdoganTzirakis}, they do not apply here, since the regularity assumptions for $(u, \wt u)$ are unconventional. 
Indeed, from \eqref{Xsbu}, we only know that $u, \wt u \in X^{s-\frac 12-,\frac 12}_{T}$ rather than $u\in X^{s,\frac 12}_{T}$, consituting a loss of $\frac 12$ of a derivative. 
Interestingly, our setting is more akin to that encountered in the area of quasi-invariance of Gaussian measures under the flow of Hamiltonian PDE \cite{Tz}, where one needs energy estimates for a $H^{s}$-functional but with input functions only of regularity $s-\frac 12-$; see for instance the case of one dimensional periodic cubic NLS in \cite{FS}.

Combining the arguments above, for sufficiently small $T>0$, we obtain differences estimates depending only on the $H^\frac12$-norms of the initial data $u_1(0), u_2(0)$: 
\begin{align}
\| \wt v_1 -\wt v_2 \|_{X^{\frac 12,\frac 12+}_{T}} +\| w_1-w_2\|_{X^{\frac 12,\frac 12+}_{T}}\leq C(T, \Ld) \|u_1 (0)-u_2(0)\|_{H^{\frac 12}} +\Ld^{-\ta}, 
\notag
\end{align}
together with similar estimates for the relevant norms of $(u_j, \wt u_j)$, $j=1,2$. By first taking the limit as $u_1 (0) \to u_2(0)$, followed by the limit $\Ld\to \infty$, we can construct solutions and obtain the continuity of the solution map. 
Note that we only recover local Lipschitz continuity on the closed subspace of $H^{s}(\T)$ comprising functions of a fixed mass.\footnote{Similar to the case of the periodic derivative nonlinear Schrodinger equation in \cite{Herr}.} This is in contrast to the line setting, where the solution map is locally Lipschitz continuous for any $s>\frac 14$ \cite{PMP, CFL} and, in fact, smooth when $s>\frac 12$ and the initial data is sufficiently small \cite{PMP}. This latter result relies heavily on dispersive smoothing effects on $\R$ and is thus not applicable on~$\T$.

Lastly, since \eqref{INLS} is $L^2$-critical, it is not enough to restrict to the case of small data, via a scaling argument. We thus need to gain factors of $T$ to ensure smallness in the bootstrap estimates. 
In the estimates for $(u_1-u_2)$ and $(\wt u_1 - \wt u_2)$ in $L^{\infty}_{T} H^s(\T)$, there is no obvious gain of $T$. As in \cite{GLM}, we decompose into high and low frequencies depending on a parameter $M\gg 1$. For the high-frequency contribution, we use the recovery formula \eqref{recovery}. However, for the ensuing difference term $\P_{\ges M}[(e^{-i\be F[u_1]}-e^{-i\be F[u_2]})v_1]$, the worst contribution is 
\begin{align}
\P_{\ll M}\big[e^{-i\be F[u_1]}-e^{-i\be F[u_2]}\big] \P_{\ges M} v_1 \label{Msmallness}
\end{align}
since, naively, we only obtain smallness from a factor  $\|\P_{\ges M}u_1(0)\|_{H^{\frac 12}}$ (after inserting the Duhamel formula for $v_1$). Now, if $u_1(0)\in H^s$ for $s>\frac 12$, then 
by Bernstein's inequality, we may choose $M$ large enough depending only on $ \|u_1(0)\|_{H^s}$ to make this contribution small (see \eqref{Mdata}), and then choose $T\sim M^{-a}$ to make other terms small. Unfortunately, when $s=\frac 12$, the dependency of $M$, and hence $T$, on $u_1(0)$ is no longer explicit. While we can still close the local well-posedness argument, it becomes unclear how to globalise solutions when the local time $T$ depends on the profile of the data and not just its norm.
The fix we devise is to keep the projection $\P_{\ll M}$ 
on the first term in \eqref{Msmallness} and exploit the fact that if this term was evaluated at time $t=0$, no smallness would be required to handle it. By the fundamental theorem of calculus, it turns out that the remaining part can be estimated in $L^{\infty}_{T}H^{\frac 12}(\T)$ with a gain of a power of $T$; see Lemma~\ref{LEM:expdiffM}. Upon further inspection, this issue at the endpoint $s=\frac 12$ also appears in \cite{GLM} for the modified BO equation in the energy space. We expect that our approach here may also be effective in \cite{GLM}.

 We conclude this introduction with a number of remarks.

 \begin{remark}\rm
As \eqref{INLS} is critical in $L^2(\T)$, the question of local well-posedness in $H^{s}(\T)$ for $0\leq s <\frac 12$ remains open. An inspection of our trilinear estimates in Proposition~\ref{PROP:tri} shows that the worst case is the nearly-resonant one $|\xi|\sim |\xi_3|$. In the periodic setting, it is crucial to obtain smoothing through multilinear dispersive effects, captured by the phase function
\begin{align*}
\Phi (\xi,\xi_1,\xi_2,\xi_3) = \xi^2 - \xi_1^2 +\xi_2^2 -\xi_3^2= 2(\xi-\xi_3)(\xi_3-\xi_2),
\end{align*}
where the second equality holds under the convolution condition $\xi=\xi_1-\xi_2+\xi_3$. In particular, 
the precise interaction $\xi=\xi_3$ is fatal since $\Phi= 0$. On the line, in \cite{CFL}, we were able to overcome these nearly-resonant interactions by gaining additional smoothing from bilinear Strichartz estimates, which unfortunately, do not hold true on $\T$. It is also not clear to us if these resonant contributions in the nonlinear term $\P_{+}[ \wt v \, \P_{-}\dx(|\wt v|^2)]$ could be gauged away as they do not seem to correspond to conservation laws for \eqref{INLS}.

For \eqref{INLS} with $\g=0$ on $\T$, another potential approach would be to explore the completely integrable structure developed in \cite{Pel3}.  In fact, it may suffice to leverage these features only for $h=\infty$.
Recently, the third author with L. Gassot \cite{GL} proved well-posedness for dispersive perturbations of BO  below $L^{2}(\T)$. The first step in this direction would be to construct Birkhoff coordinates for \eqref{CCM} on $\T$, analogous to the BO setting, which are not yet known, even when restricted to the Hardy space.
  \end{remark}

\begin{remark}\rm \label{RMK:UUDNLS}

Recently, Kishimoto \cite{KishimotoDNLS} proved unconditional uniqueness of the derivative nonlinear Schr\"odinger equation (DNLS)
\begin{align}
\dt u + i\dx^2 u = 2\dx(|u|^2 u) \label{DNLS}
\end{align}
 in $L^{\infty}_{T}H^{\frac 12} (\T)$.
 In view of the nonlinearity in \eqref{DNLS}, we see that unconditional uniqueness makes sense whenever $u(t)\in L^3(\T)$, which is implied by $u(t) \in H^{\frac 16}(\T)$. However, passing to the gauged DNLS variable $v: = e^{-i\dx^{-1}\P_{\neq 0}(|u|^2)} u$, which satisfies 
 \begin{align}\notag
\dt v + i\dx^2 v = v^{2} \dx \cj{v} + \text{lower order terms} 
,
\end{align}
  the roughest nonlinear term is $v^2 \dx \cj{v}$. Using \eqref{moser} and Sobolev embedding, we see that if $u\in L^{\infty}_{T}H^{s} (\T)$, $s>\frac 12$, then this nonlinear term makes sense as a spatial distribution for each fixed time. When $s=\frac 12$, Kishimoto used refined Strichartz estimates to show that if  $u\in L^{\infty}_{T}H^{\frac 12} (\T)$ solves \eqref{DNLS} in the sense of distributions, then $u\in L^{4}_{T}L^{\infty}(\T)$ and hence $v\in L^{4}_{T}L^{\infty}(\T)$. By then integrating in time in the estimate \eqref{moser}, one can make sense of~$v^2 \dx \cj{v}$. 

To apply this approach to \eqref{INLS}, we need to prove refined Strichartz estimates. A key step here is to establish the estimate
\begin{align}
\| \P_{N}[ u \P_{+}\dx(|u|^2)] \|_{H^{-\frac 12-\eps}_x} \leq C_{N} \|u\|_{H^{\frac 12}_x}^{3} \label{UUest}
\end{align}
for some $\eps>0$,  where the constant $C_{N}$ may depend on $N\in \N$.  Unfortunately, the estimate \eqref{UUest} is false.\footnote{For example, one can take $u$ with $\ft{u}(\xi) = \jb{\xi}^{-1} (\log \jb{\xi})^{-\frac 23}$.}
The corresponding estimate for \eqref{DNLS} is 
\begin{align}
\| \P_{N}\dx( |u|^2 u) \|_{H^{-\frac 12-\eps}_x} \leq C_{N} \|u\|_{H^{\frac 12}_x}^{3} , \label{UUest2}
\end{align}
which does hold true. The main difference between \eqref{UUest} and \eqref{UUest2} is that the nonlinearity for \eqref{INLS} cannot be written in divergence form as in \eqref{DNLS}. Moreover, \eqref{UUest} differs from \eqref{UU+} due to the lack of an extra $\P_{+}$ projection on the left-hand side.
\end{remark}

\begin{remark}\rm 
In the recent paper \cite{KMarV}, the authors investigate the Hamiltonian structure of CCM \eqref{CCM} in the Hardy space on both $\R$ and $\T$. They propose that, based on this structure, the correct model to study on $L^{2}_{+}(\T)$ is that of the gauged variable $\wt u$, namely \eqref{ueq2} (when $h=\infty$ and $\g=0$). We point out that our Theorem~\ref{THM:LWP} also applies to \eqref{ueq2}.
At a technical level, we also note that simply starting with \eqref{ueq2} instead of \eqref{INLS} does not solve the difficulties that we encountered in the proof of Theorem~\ref{THM:LWP}, since it is not possible to control the corresponding variable $\wt w=\P_{-,\text{hi}} \wt u$. One would need to consider $w = \mathcal{J}_{u}^{-1}[\wt w]$ which leads to essentially the same quartet of variables as discussed earlier.
\end{remark}

\subsection{Organization of the paper}

The remaining of the paper is organised as follows.
In Section~\ref{SEC:Prelim}, we introduce relevant notation and useful product-type estimates. 
Section~\ref{SEC:Gauge} introduces the gauged variables $v, w$ in \eqref{gauge}, establishes the regularity and difference estimates for $(u,\wt{u})$.  
The crucial trilinear estimates to handle the nonlinearities for $(\wt{v},w)$ are proven in Section~\ref{SEC:trilin}. 
In Section~\ref{SEC:LWP} we present the proof of the local well-posedness part of Theorem~\ref{THM:LWP}. We then prove Theorem~\ref{THM:convergence}, the infinite depth limit $h\to\infty$, in Section~\ref{SEC:conv}. In Section~\ref{SEC:UU}, we complete the proof of Theorem~\ref{THM:LWP} by establishing the unconditional uniqueness result in items (ii) and (iii).
Lastly, in Section~\ref{SEC:GWP}, we consider the global well-posedness results, proving Corollary~\ref{COR:gwp} and Theorem~\ref{THM:GWP}.

\section{Preliminaries}\label{SEC:Prelim}

\subsection{Notation}

We use $A\les B$ to denote $A\leq C B$ for some constant $C>0$, $A\sim B$ when $A \les B$ and $B \les A$, and $A \ll B$ when $A \le \eps B$ for some small $0<\eps<1$. 
Our convention for the Fourier transform on $\M = \T$ or $\M=\R$ is as follows:
\begin{align}
(\F f) (\xi) = \ft f(\xi) = \frac{1}{\sqrt{2\pi}} \int_{\M} f(x) e^{-ix\xi} dx \quad \text{and} \quad f(x) &= \frac{1}{\sqrt{2\pi}} \int_{\ft \M} \ft f(\xi) e^{ix\xi}d\xi
, \label{Fourier}
\end{align}
for $\xi \in \ft{\M}$, where $\ft\M = \Z$ when $\M=\T$ and $\ft\M=\R$ for $\M= \R$. We will write $\F_x$, $\F_t$, and $\F_{t,x}$ to denote the Fourier transform with respect to the spatial variable, the time variable, or both, respectively, unless it is clear from context.  
We then see that
\begin{align}
\P_{0}f =  \frac{1}{\sqrt{2\pi} } \ft f(0) = \frac{1}{2\pi } \int_{0}^{2\pi } f(x)dx, \label{P0}
\end{align}
and $\P_{\neq 0}= \Id -\P_{0}$.
We use the following convention for the inner product on $L^2(\T)$:
\begin{align*}
    \jb{f,g} = \int_{\T} \cj{f(x)}g(x)dx.
\end{align*}

Next, let $\eta:\R\to [0,1]$ be a smooth function  with support in $[-2,2]$ and equal to $1$ on $[-1,1]$. We define $\eta_{N}(\xi)=\eta(\frac{\xi}{N})$ and $\psi_N(\xi)=\eta(\frac{\xi}{N})-\eta(\frac{2\xi}{N})$. 
For $N\in 2^{\Z}$, we denote 
 the Littlewood-Paley projectors by
 \begin{align*}
  \F \,(\P_{\leq N} f)(\xi)  = \eta(\tfrac\xi{N})\ft f(\xi), \quad  \F( \P_N f) = \F (\P_{\leq N} f ) -\F  (\P_{\leq \frac{N}{2}} f) =\eta_N \ft f,
 \end{align*}
 and $\P_{>N}:= \Id-\P_{\leq N}$. We then set
\begin{align}
\begin{split}
 \F\,( \P_{+} f)(\xi) = \ind_{\xi> 0} \ft f(\xi), &\quad   \F\, ( \P_{-} f)(\xi)  = \ind_{\xi< 0} \ft f(\xi),\\
  \Pblo = \P_{\le 1}, \quad \PbLO =\P_{\le 10}, \quad  \Pbhi &=\Id -\Pblo, \quad \Pbhip= \P_{+} \Pbhi, \quad \PbHI = \P_{>10}.
  \end{split} 
  \notag
\end{align}
We also define two types of frequency projections for functions on $\R\times \T$ with Fourier variables  $(\tau,\xi)\in \R\times \Z$. Given $K\in 2^{\N}$, we set 
\begin{align}
\begin{split}
\mathcal{F}_{t,x}\{ \Q_{\ll K}u\}(\tau,\xi) &= \eta_{10^{-10}K}(\tau - \xi^2)\ft u(\tau,\xi),\\
 \mathcal{F}_{t,x}\{ \Q_{\ges K}u\}(\tau,\xi) &=(1- \eta_{10^{-10}K})(\tau - \xi^2)\ft u(\tau,\xi).
 \end{split} \label{Qproj}
\end{align}
We will use $\Pi_{>\Ld}$ to denote the sharp frequency projection with multiplier $\ind_{\{ |\xi|>\Ld\}}$ for $\Ld\in \N$. Similarly, we define $\Pi_{\leq \Ld}$ with multiplier $\ind_{\{ |\xi|\leq \Ld\}}$. 

For $1< p<\infty$, we define the dyadic $L^p$-spaces via the norm
\begin{align}
\| f\|_{ \wt{L^p_{t,x}}} = \bigg( \sum_{N\in 2^{\N}} \| \P_{N}f\|_{L^p_{t,x}}^{2} \bigg)^{\frac 12}.
\notag
\end{align}
By the Littlewood-Payley square function theorem and Minkowski's inequality, it holds that 
\begin{align*}
\|f\|_{L^{p}_{t,x}} \les \|f\|_{\wt{L^{p}_{t,x}}}
\end{align*}
for any $2\leq p<\infty$. We similarly define the spaces $\wt{L^p_{x}}$. Given $T>0$ and a space $X$, we also write $L^p_{T} X = L^p([0,T]; X)$.

\subsection{Product estimates}
In this subsection, we collect some auxiliary product estimates. For a proof of the following lemma, see \cite{CW,GO,BL} on $\R$ and 
\cite{IK2,BOZ} for $\T$.

\begin{lemma}[Fractional Leibniz rule]
\label{LEM:leib}
Let $s>0$ and $1<p_j,q_j,r\leq \infty$
such that $\frac1r=\frac{1}{p_j}+\frac{1}{q_j}$, $j=1,2$. 
Then, we have
\begin{align*}
  \| J^s(fg)\|_{L^r_x} \les \|J^s f\|_{L^{p_1}_x} \|g\|_{L^{q_1}_x}  + \|f\|_{L^{p_2}_x} \| J^s  g\|_{L^{q_2}_x}.
\end{align*}
\end{lemma}
\noi

We now give a sequence of product estimates. The complicated nature of these estimates is due to the lack of the Sobolev embedding of $H^{s}(\T)$ into $L^{\infty}(\T)$ when $s=\frac 12$.

\begin{lemma}\label{LEM:dxs1}
Let $f, g: \T \to \C$.
\\
{\rm(i)}
Given $s\geq \frac 12$, 
we have
\begin{align}
\| f \dx g\|_{H^{s-1}_x} \les \big(  \|f\|_{H^{\frac 14}_x} + \| J_x^{\frac 12} f\|_{\wt{L^{\frac 43+}_x}} \big) \|J_x^{s} g\|_{\wt{L^4_x}} + \|f\|_{H^{s-\frac 14}_x}\|J_x^{\frac 12} g\|_{\wt{L^4_x}}.
 \label{prodest}
\end{align}
In particular, for $u_1,u_2,u_3 : \T \to \C$, we have
\begin{align}
 \|  \P_{\pm}(u_1 u_2) \dx u_3\|_{H^{s-1}_x}  
& \les \|u_1\|_{H^{\frac 12}} \|u_2\|_{H^{\frac 12}_x} \|J^{s}_x u_3\|_{\wt{L^4_x}} \notag
 \\
 &
 \quad +  ( \|u_1\|_{H^{\frac 12}_x} \|u_2\|_{H^{s}_x} +  \|u_1\|_{H^{s}_x} \|u_2\|_{H^{\frac 12}_x} ) \| J_x^{\frac 12} u_3\|_{\wt{L^4_x}}. 
\label{prodest1}
\end{align}

\noi{\rm{(ii)}} 
It holds that 
\begin{align}
\| f\dx g\|_{H^{-1}_x} \les \|f\|_{H^{\frac 12}_x} \|g\|_{H^{\frac 12}_x}.
\label{uniqueL2}
\end{align}
Moreover, given $s>\frac 12$, it holds that
\begin{align}
\| f \dx g\|_{H^{s-1}_x} \les \|f\|_{H^{\frac 12+}_x} \|g\|_{H^{s}_x}+\|f\|_{H^{s}_x} \|g\|_{H^{\frac 12+}_x}. \label{prodestalg}
\end{align}
\end{lemma}

\begin{proof}
We first prove (i).
By duality and dyadic decomposition, it is enough to consider the quantity
\begin{align}
\sum_{N,N_1, N_2} N^{s-1}N_2 \int_{\T} \P_{N}h \cdot \P_{N_1} f \cdot \P_{N_2}g  dx =:\sum_{N,N_1, N_2} I_{N,N_1,N_2}\label{prodest2}
\end{align}
where $h\in L^2(\T)$ with $\|h\|_{L^2}\le 1$.
We split the summation in \eqref{prodest2} into three pieces. If $N_2 \sim N$, then by Cauchy-Schwarz and Bernstein inequalities, 
\begin{align*}
\bigg| \sum_{N\sim N_2} \sum_{N_1 } I_{N,N_1,N_2} \bigg| 
= \bigg|  \sum_{N\sim N_2} N^{s}_2 \int_{\T} \P_{N}h \cdot f \cdot \P_{N_2}g dx \bigg| 
&  
\les \| f\|_{L^4} \sum_{N\sim N_2} N_2^s  \|\P_N h\|_{L^2} \| \P_{N_2} g\|_{L^4}  
\\
& 
\les \| f\|_{H^{\frac 14}} \|J^s g\|_{\wt{L^4}} . 
\end{align*}
If $N_2 \ll N$, then we have $N \sim N_1$ and hence
\begin{align*}
\sum_{N\sim N_1} \sum_{N_2 \ll N}| I_{N,N_1,N_2}| 
& \les \sum_{N\sim N_1} \sum_{N_2 \ll N } \Big(\frac{N_2}{N} \Big)^{\frac 12} N_2^{\frac 12}\|\P_{N_2}g\|_{L^4} N^{-\frac14} \|\P_{N}h\|_{L^4}  N_1^{s-\frac 14} \|\P_{N_1}f\|_{L^2} \\
& \les \| J^{\frac 12} g\|_{\wt{L^4}} \|h\|_{L^2} \| f\|_{H^{s-\frac 14}}.
\end{align*}
Lastly, when $N_2 \gg N$ so that $N_2 \sim N_1$, we have 
\begin{align*}
\sum_{N_2\sim N_1} \sum_{N \ll N_2 }| I_{N,N_1,N_2}|  & \les \sum_{N_2\sim N_1} \sum_{N \ll N_2 } N_2 N^{s-1}  \| \P_{N}h\|_{L^{\infty-}} \| \P_{N_2}g\|_{L^4}  \|\P_{N_1}f\|_{L^{\frac 43+}} \\
& \les  \sum_{N_2\sim N_1}   \| J^{s} \P_{N_2}g\|_{L^4}  \| J^{\frac 12} \P_{N_1}f\|_{L^{\frac 43+}} \sum_{N \ll N_2} \Big( \frac{N}{N_2}\Big)^{s-\frac 12-} \|\P_{N}h\|_{L^{2}} \\
& \les \| J^{s} g\|_{\wt{L^4}} \|J^{\frac 12} f\|_{\wt{L^{\frac 43 +}}} \|h\|_{L^2}.
\end{align*}
This completes the proof of \eqref{prodest}. 

For \eqref{prodest1}, we apply \eqref{prodest} with $f=\P_{\pm}(u_1 u_2)$ and $g=u_3$. By the fractional Leibniz rule and Sobolev embedding we have
\begin{align*} 
\| u_1 u_2\|_{H^{\frac 14}} &\les \| J^{\frac 14} u_1\|_{L^{2+}} \|u_2\|_{L^{\infty-}} + \|  u_1\|_{L^{\infty-}} \|J^{\frac 14} u_2\|_{L^{2+}}   \les \|u_1\|_{H^{\frac 12}} \|u_2\|_{H^{\frac 12}} ,\\
\| u_1 u_2\|_{H^{s-\frac 14}} & \les \| J^{s-}u_1\|_{L^{2+}} \|u_2\|_{L^{\infty-}}  + \|u_1\|_{L^{\infty-}} \|J^{s-}u_2\|_{L^{2+}}  \\
& \les  \|u_1\|_{H^{\frac 12}} \|u_2\|_{H^{s}} +  \|u_1\|_{H^{s}} \|u_2\|_{H^{\frac 12}}.
\end{align*}
Next, by H\"{o}lder and Bernstein inequalities, we have 
\begin{align*}
\sum_{N} N  \| \P_{\ll N}u_1 \cdot \wt{\P}_{N} u_2\|_{L^{\frac 43+}}^{2} &\les  \bigg(\sum_{N}N \|\wt{\P}_{N}u_2\|_{L^{2}}^2 \bigg)  \| u_1\|_{H^{\frac 14+}}^2 \les \|u_2\|_{H^{\frac 12}}^2 \|u_1\|_{H^{\frac 12}_x}^2, \\
 \sum_{N} N \|  \P_{\ges N} u_1 \cdot \P_{\ges N} u_2]\|_{L^{\frac 43 +}}^2  & \les \sum_{N} N^{-\frac 12+} \|u_1\|_{H^{\frac 12}}^2 \|u_2\|_{H^{\frac 12}}^2 \les  \|u_1\|_{H^{\frac 12}}^2 \|u_2\|_{H^{\frac 12}}^2, 
\end{align*}
where $\wt\P_N$ denotes a wider projection. 
Then, combining these inequalities with the boundedness of $\P_{\pm}$ on $L^p$ for $1<p<\infty$ and
\begin{align*}
\| \P_{N} [ u_1 u_2]\|_{L^{\frac 43 +}}  \leq \|  \P_{\ges N} u_1 \cdot \P_{\ges N} u_2]\|_{L^{\frac 43 +}}  + \| \P_{\ll N}u_1 \cdot \wt{\P}_{N} u_2\|_{L^{\frac 43+}}+ \| \wt{\P}_{N}u_1 \cdot \P_{\ll N} u_2\|_{L^{\frac 43+}},
\end{align*}
we obtain 
\begin{align*}
\| J^{\frac 12} \P_{\pm}(u_1 u_2)\|_{\wt{L^{\frac 43+}}} \les   \|u_1\|_{H^{\frac 12}} \|u_2\|_{H^{\frac 12}}.
\end{align*}
This completes the proof of \eqref{prodest1}.

To show (ii), note that \eqref{uniqueL2} is implied by the following bound: 
\begin{align}
\bigg\|\sum_{\xi=\xi_1+\xi_2} \frac{ \jb{\xi_2}^{\frac 12}}{ \jb{\xi} \jb{\xi_1}^{\frac 12}} u(\xi_1) v(\xi_2) d\xi_1 \bigg\|_{\l^2_{\xi}} \les \|u\|_{\l^2} \|v\|_{\l^2}.
\label{Hminus1a}
\end{align}
If $|\xi_2|\gg |\xi_1|$, then by Young's and Cauchy-Schwarz inequality, we have
\begin{align*}
\text{LHS \eqref{Hminus1a}}
& \les \big\| \big( \jb{\cdot}^{-1}u) \ast v\big\|_{\l^2} \les \| \jb{\cdot}^{-1}u\|_{\l^1} \| v\|_{\l^2} \les \| u\|_{\l^2}\| v\|_{\l^2}.
\end{align*}
If instead $|\xi_2|\les |\xi_1|$, then 
\begin{align*}
\text{LHS \eqref{Hminus1a}}
\les \| \jb{\xi}^{-1}( u \ast v)\|_{\l^2_\xi} \les \|u \ast v\|_{\l^{\infty}_\xi} \les \|u\|_{\l^2}\|v\|_{\l^2}.
\end{align*} 
The proof of \eqref{prodestalg} is similar to that of \eqref{uniqueL2} and thus we omit it.
This completes the proof. \qedhere 
\end{proof}

We now state some estimates concerning products with exponential factors $e^{iF}$ for $F$ real-valued.
As $e^{iF}\in L^{\infty}(\T)$, it follows that all Littlewood-Payley type frequency projectors acting on expressions of the form $e^{iF}$ are well-defined. In particular, the function $\Pblo e^{iF}$ is supported on frequencies $|
\xi|\les 1$. This is in stark contrast to the case on the real-line, where 
$e^{iF}\in L^{\infty}(\R)\setminus L^2(\R)$ and more care is needed. See for example \cite{MP, MP-2023, CFL}.

\begin{lemma}\label{LEM:expdyad}
Let $2\leq p<\infty$ and assume $F_1, F_2$ are two real-valued functions such that 
$\dx F_j  =|f_j|^2$  with $f_j\in H^{\frac 12}(\T)$, for $j=1,2$. Then, it holds that 
\begin{align}
&\|\P_{\ges N}e^{i F_1}\|_{L^{\infty}_{x}}  \les N^{-1+} \|f_1\|_{H^{\frac 12}_x}^2, \label{LinftyE11} \\
&\|\P_{N} e^{i F_1}\|_{L^{p}_x}  \les N^{-2+}\|f_1\|_{H^{\frac 12}_x}^{4} +N^{-1+}\| \P_{\ges N} f_1\|_{H^{\frac 12-\frac 1p}_x} \|f_1\|_{H^{\frac 12}_x}, \label{LpE}\\
 &\| \P_{\ges N}(e^{i F_1}-e^{i F_2})\|_{L^{\infty}_{x}} \notag  \\
 & 
 \les  N^{-1+}
 (\|f_1\|_{H^{\frac 12}_x}+\|f_2\|_{H^{\frac 12}_x}) 
 \big\{ \|f_1-f_2\|_{H^{\frac 12}_x} + 
 (\|f_1\|_{H^{\frac 12}_x}+\|f_2\|_{H^{\frac 12}_x})
 \|e^{i F_1}-e^{i F_2}\|_{L^{\infty}_{x}} \big\}, \label{L4est2} \\
 & \| \P_{ N}(e^{i F_1}-e^{i F_2})\|_{L^{p}_{x}}  \notag \\
 & \les N^{-2+}(\|f_1\|_{H^{\frac 12}_x}+\|f_2\|_{H^{\frac 12}_x})^{3} \big\{ \|f_1-f_2\|_{H^{\frac 12}_x} 
 + 
 (\|f_1\|_{H^{\frac 12}_x}+\|f_2\|_{H^{\frac 12}_x})
 \|e^{i F_1}-e^{i  F_2}\|_{L^{\infty}_{x}} \big\}  \notag \\
 &\quad +N^{-1+} \big\{  \|\P_{\ges N}(f_1-f_2)\|_{H^{\frac 12-\frac 1p}_x} \| f_1+f_2\|_{H^{\frac 12}_x}  +\|\P_{\ges N}(f_1+f_2)\|_{H^{\frac 12-\frac 1p}_x} \| f_1-f_2\|_{H^{\frac 12}_x}   \big\} \notag \\
 &\quad +N^{-1+}  \| \P_{\ges N}f_2\|_{H^{\frac 12-\frac 1p}_x}\|f_2\|_{H^{\frac 12}_x}\|e^{i  F_1}-e^{i  F_2}\|_{L^{\infty}_{x}}. \label{umod2b}
\end{align}
\end{lemma}

\begin{proof}
By Cauchy-Schwarz, we have 
\begin{align}
\| \P_{\les N} f\|_{L^{\infty}_{x}}  \les (\log N)^{\frac 12} \|f\|_{H^{\frac 12}_x}.
\label{Linfty0+}
\end{align}
Then, we get
\begin{align}
\begin{split}
&\| \P_{\ges N}(fg )\|_{L^{\infty}_{x}} \\
& \les N^{\frac 12} \|\P_{\ges N}(fg)\|_{L^2_x}  \\
&\les N^{\frac 12} \big\{ \| \P_{\ges N}f \cdot \P_{\ll N}g\|_{L^2_x} +  \| \P_{\ges N}g \cdot \P_{\ll N}f\|_{L^2_x}+  \| \P_{\ges N}f\cdot \P_{\ges N}g\|_{L^2_x}\big\} \\
& \les N^{\frac 12} \big\{ \|\P_{\ges N}f\|_{L^2_x} \|\P_{\ll N}g\|_{L^{\infty}_{x}} +\|\P_{\ges N}g\|_{L^2_x} \|\P_{\ll N}f\|_{L^{\infty}_{x}}+ \|\P_{\ges N} f\|_{L^4_x}\|\P_{\ges N}g\|_{L^4_x}\big\}  \\
& \les N^{0+} \|f\|_{H^{\frac 12}_x}\|g\|_{H^{\frac 12}_x}.
\end{split} 
\label{L4est1}
\end{align}
It then follows from \eqref{L4est1} that 
\begin{align}
\begin{split}
\| \P_{\ges N}e^{i  F_1}\|_{L^{\infty}_{x}}  
 &\les N^{-1}   \| \P_{\ges N}( \P_{\neq 0}(|f_1|^2) e^{i  F_1}) \|_{L^{\infty}_{x}}\\
 & \les N^{-1} \big\{  \| \P_{\ges N}[ \P_{\neq 0}(|f_1|^2)]\|_{L^{\infty}_{x}} + \| \P_{\ll N}\P_{\neq 0}(|f_1|^2) \cdot \P_{\ges N}e^{i  F_1}\|_{L^{\infty}_{x}} \big\} \\
 &\les N^{-1+}  \|f_1\|_{H^{\frac 12}_x}^2  +N^{-1+} \||f_1|^2 \|_{L^{\infty-}_{x}}  \les N^{-1+}\|f_1\|_{H^{\frac 12}_x}^{2}.
 \end{split} \label{LinftyE1}
\end{align}
This proves \eqref{LinftyE11}. 
Therefore, proceeding as in \eqref{LinftyE1} with \eqref{L4est1}, we have
\begin{align*}
 &\| \P_{\ges N}[e^{i  F_1}-e^{i  F_2}]\|_{L^{\infty}_{x}}  \\
 &\les N^{-1} \big\{  \| \P_{\ges N} [\P_{\neq 0}(|f_1|^2- |f_2|^2) e^{i  F_1}] \|_{L^{\infty}_{x}}+ \| \P_{\ges N}[ \P_{\neq 0}(|f_2|^2) (e^{i  F_1}-e^{i  F_2})] \|_{L^{\infty}_{x}}
 \big\}
 \\
 & \les N^{-1} \big\{  \| \P_{\ges N}[ \P_{\neq 0}(|f_1|^2-|f_2|^2)]\|_{L^{\infty}_{x}} + \| \P_{\ll N}\P_{\neq 0}(|f_1|^2-|f_2|^2) \cdot \P_{\ges N}e^{i  F_1}\|_{L^{\infty}_{x}} + \  \\
 & \hphantom{X}+\| \P_{\ges N}[ \P_{\neq 0}(|f_2|^2)]\|_{L^{\infty}_{x}} \|e^{i  F_1}-e^{i  F_2}\|_{L^{\infty}_x} +\| \P_{\ll N}\P_{\neq 0}(|f_2|^2) \cdot \P_{\ges N}(e^{i  F_1}-e^{i  F_2})\|_{L^{\infty}_{x}}   \big\} \\
 &\les N^{-1+} \|f_1 +f_2\|_{H^{\frac 12}_x} \|f_1-f_2\|_{H^{\frac 12}_x}  +N^{-1+} \||f_1|^2 - |f_2|^2 \|_{L^{\infty-}_{x}} +N^{-1+}\|f_2\|_{H^{\frac 12}_x}^2 \| e^{i  F_1}-e^{i  F_2}\|_{L^{\infty}_{x}}  \\
 & \les N^{-1+}
 (\|f_1\|_{H^{\frac 12}_x}+\|f_2\|_{H^{\frac 12}_x}) 
 \big\{ \|f_1-f_2\|_{H^{\frac 12}_x} + (\|f_1\|_{H^{\frac 12}_x}+\|f_2\|_{H^{\frac 12}_x}) 
 \|e^{i  F_1}-e^{i  F_2}\|_{L^{\infty}_{x}} \big\} , 
\end{align*}
proving \eqref{L4est2}.
 Let $2\leq p<\infty$. By Bernstein's inequality and \eqref{Linfty0+}, we have
 \begin{align}
\begin{split}
\| \wt{\P}_{N}( f g)\|_{L^p_x}& \les N^{0+}\| \wt{\P}_{N}( f g)\|_{L^{p-}_x} \\
&\les N^{0+} \big\{  \| \P_{\ges N} f \cdot g\|_{L^{p-}_{x}} + \|\P_{\ll N}f \cdot \wt{\P}_{N} g\|_{L^{p-}_x}    \big\} \\
& \les N^{0+} \big\{ \|\P_{\ges N} f\|_{L^{p}_x} \|g\|_{L^{\infty-}_x} + \|f\|_{H^{\frac 12}_x} \| \wt{\P}_{N} g\|_{L^{p}_x}\big\} \\
&\les N^{0+} \big\{   \|\P_{\ges N}f\|_{H^{\frac 12-\frac 1p}_x}  \|g\|_{H^{\frac 12}_x} +\|f\|_{H^{\frac 12}_{x}} \|\P_{\ges N}g\|_{H^{\frac 12-\frac 1p}_x}   \big\}.
\end{split} 
\label{fgLp}
\end{align}

 Then, from \eqref{L4est1}, \eqref{LinftyE1}, \eqref{fgLp}, and \eqref{L4est2}, we have
\begin{align}
&\| \P_{N}(e^{i  F_1}-e^{i  F_2}) \|_{L^p_x} \notag \\
& \les N^{-1}\big\{ \|\P_{\ll N}(|f_1|^2-|f_2|^2)\|_{L^{p}_x} \| \wt{\P}_{N}e^{i  F_1}\|_{L^{\infty}_x} + \| \wt{\P}_{N}(|f_1|^2-|f_2|^2)\|_{L^p_x}  \notag \\ 
& \hphantom{XX}+ \|\P_{\gg N}(|f_1|^2-|f_2|^2) \P_{\gg N}e^{i  F_1}\|_{L^p_x} + \|\P_{\ll N}(|f_2|^2)\|_{L^{\infty}_x} \|\wt{\P}_{N}(e^{i  F_1}-e^{i  F_2})\|_{L^{\infty}_x} \notag \\
& \hphantom{XX} +\|\wt{\P}_{N}(|f_2|^2)\|_{L^p_x}\| e^{i  F_1}-e^{i  F_2}\|_{L^{\infty}_{x}}  + \|\P_{\gg N}(|f_2|^2) \P_{\gg N}(e^{i  F_1}-e^{i  F_2})\|_{L^p_x}  \big\} \notag \\
& \les N^{-2+}(\|f_1\|_{H^{\frac 12}_x}+\|f_2\|_{H^{\frac 12}_x})^{3} \|f_1-f_2\|_{H^{\frac 12}_x} \notag  \\
&\hphantom{XX} +N^{-1+} \big\{  \|\P_{\ges N}(f_1-f_2)\|_{H^{\frac 12-\frac 1p}_x} \| f_1+f_2\|_{H^{\frac 12}_x}  +\|\P_{\ges N}(f_1+f_2)\|_{H^{\frac 12-\frac 1p}_x} \| f_1-f_2\|_{H^{\frac 12}_x}   \big\} \notag \\
& \hphantom{XX} +N^{-2+}(\|f_1\|_{H^{\frac 12}_x}+\|f_2\|_{H^{\frac 12}_x})^{3} 
\big\{ \|f_1-f_2\|_{H^{\frac 12}_x} 
+
(\|f_1\|_{H^{\frac 12}_x}+\|f_2\|_{H^{\frac 12}_x})
\|e^{i  F_1}-e^{i  F_2}\|_{L^{\infty}_{x}} \big\} \notag \\
& \hphantom{XX} +N^{-1+}  \| \P_{\ges N}f_2\|_{H^{\frac 12-\frac 1p}_x}\|f_2\|_{H^{\frac 12}_x}\|e^{i  F_1}-e^{i  F_2}\|_{L^{\infty}_{x}}
\notag
\end{align}
where the slight loss in the powers of $N$ can be chosen arbitrarily small. This proves \eqref{umod2b} and similar arguments establish \eqref{LpE}.
\end{proof}

\begin{lemma}\label{LEM:L2prod}
Let $0\leq s< 2$ and assume that $F$ is a real-valued function such that $\dx F= |f|^2$ for $f\in H^{\frac12}(\T)$. Then,  we have
\begin{align}
\| J^{s}\Pbhi( e^{i F} g)\|_{L^{2}_x} \les \|f\|_{H^{\frac 12}_x}^{2}\|g\|_{H^s_x} +\|f\|_{H^{\frac 12}_x}(\|f\|_{H^{\frac 12}_x}^3 + \|f\|_{H^{\max(\frac 12,s-\frac 34)}_x}) \|g\|_{H^{\frac 12}_x}
.
 \label{eFg}
\end{align}
Moreover, if $F_j$ are real-valued functions such that $\dx F_j =|f_j|^2$, $j=1,2$, $R:=\|f_1\|_{H^{\frac 12}_x}+\|f_2\|_{H^{\frac 12}_x}$, and $\s=\max(\frac 12, s-\frac 34)$, then it holds that
\begin{align}
&\| J^{s}\Pbhi( [e^{i F_1}-e^{iF_2}] g)\|_{L^{2}_x}  \notag \\
&\les R(1+R^2)\{ \|f_1-f_2\|_{H^{\frac 12}_x} +R \|e^{iF_1}-e^{iF_2}\|_{L^{\infty}_x}\big\} \|g\|_{H^{s}_x} \label{eFgdiff} \\
&  \hphantom{X} +\{R \|f_1-f_2\|_{H^{\s}_x} + (\|f_1\|_{H^{\s}_x}+\|f_2\|_{H^{\s}_x})( \|f_1-f_2\|_{H^{\frac 12}_x} + R \|e^{iF_1}-e^{iF_2}\|_{L^{\infty}_x}) \}
\|g\|_{H^{\frac12}_x} . \notag
\end{align}
\end{lemma}

\begin{proof}
We begin with \eqref{eFg}. By dyadic decompositions, we have 
\begin{align*}
\| J^{s}_x \Pbhi (e^{iF} g)\|_{L^2_x}^{2}  \les \sum_{N} N^{2s} \bigg( \sum_{N_1, N_2} \|\P_{N_1}e^{iF} \cdot \P_{N_2}g\|_{L^2_x} \bigg)^2.
\end{align*}
We split the above sum into two pieces. When $N_1 \ll N_2$, we use \eqref{LinftyE11} to bound this contribution by 
\begin{align*}
\sum_{N} N^{2s} \bigg( \sum_{N_1\ll N_2 \sim N} N_1^{-1+} \|f\|_{H^{\frac 12}_x}^{2} \|\P_{N_2}g\|_{L^2_x}\bigg)^2 \les \|f\|_{H^{\frac 12}_x}^{4} \|g\|_{H^{s}_x}^{2}.
\end{align*}
For the contribution from $N_1 \ges N_2$, we use \eqref{Linfty0+} and \eqref{LpE} to obtain
\begin{align*} 
\sum_{N}& N^{2s} \bigg( \sum_{N_1 \ges N_2} \|\P_{N_1}e^{iF}\|_{L^2_x}   \|\P_{N_2}g\|_{L^{\infty}_x}\bigg)^2  \\
& \les \sum_{N}N^{2s} \bigg( \sum_{N_1 }  \big\{ N_1^{-2+}\|f\|_{H^{\frac 12}_x}^{4} +N_1^{-1+} \|\P_{\ges N_1}f\|_{L^2_x} \|f\|_{H^{\frac 12}_x} \big\}  \sum_{N_2 \les N_1 } N_{2}^{0+} \|g\|_{H^{\frac 12}_x}\bigg)^2  \\
& \les \|f\|_{H^{\frac 12}_x}^{8} \|g\|_{H^{\frac 12}_x}^2 + \|f\|_{H^{\frac 12}_x}^2 \|g\|_{H^{\frac 12}_x}^{2} \sum_{N}N^{2s} \bigg(  \sum_{N_1 \ges N} N_1^{-1+} \|\P_{\ges N_1}f\|_{L^2_x} \bigg)^{2} \\
& \les  \|f\|_{H^{\frac 12}_x}^{8} \|g\|_{H^{\frac 12}_x}^2 + \|f\|_{H^{\frac 12}_x}^{2} \|f\|^2_{H^{\max(\frac 12, s-\frac 34)}_x} \|g\|_{H^{\frac 12}_x}^2, 
\end{align*}
completing the proof of \eqref{eFg}. The difference estimate in \eqref{eFgdiff} follows analogously. \qedhere
\end{proof}

\section{The gauge transform and properties of solutions}\label{SEC:Gauge}
\subsection{Fourier restriction norm spaces}\label{SEC:Xsb}

For $s,b\in \R$, we consider the Fourier restriction norm spaces $X^{s, b} = X^{s,b}(\R\times \T)$ which are the completion of $\S(\R\times \T)$ under the norm:
\noi
\begin{align}
\| u\|_{X^{s, b}(\R\times \T)} &= \big\| \jb{\tau-\xi^2}^{b}\jb{\xi}^{s} \ft u(\tau, \xi)\big\|_{L^2_\xi(\Z) L^{2}_{\tau}(\R)}. 
\notag
\end{align}
We will use the notation $\s =\s(\tau,\xi): = \tau-\xi^2$.
Given a time interval $I \subset \R$, we define localised in time versions of these spaces as follows: if $u: I \times \T \to \C$, then 
\begin{align}
\|u\|_{X^{s,b}_{I}}: =\inf\{ \| z \|_{X^{s,b}} \, : \, z :\R\times \T \to \C, \,\, z \vert_{I\times \T} = u\}. 
\notag
\end{align}
 When $I=[0,T]$ for some $T>0$, we denote the spaces $X^{s,b}_{I}=X^{s,b}_{T}$.
For the time localised spaces, the following embedding holds:
\begin{align}
X_{T}^{s,b} \subset C([0,T];H^{s}(\T)) \label{YsCTHs}
\end{align}
for any $b>\frac 12$.
We recall the following linear estimates related to the $X^{s,b}$-spaces \cite{BO93, TAObook}.

\begin{lemma}\label{LEM:linXsb}
Let $s, b \in \R$ , $\eta$ denote a smooth time cutoff, and $S(t) = e^{i t\dx^2}$.

\noi
\textup{(i)} Let $0<\dl<\frac12$. Then,  
\begin{align}
\| \eta(t) S(t)f\|_{X^{s,b}}  \les \|f\|_{H^{s}} \quad \text{and} \quad 
\bigg\| \eta(t) \int_{0}^{t} S(t-t') g(t')dt' \bigg\|_{X^{s,\frac 12+\dl}}  \les \|g\|_{X^{s, -\frac 12 +\dl}}
.
\notag
\end{align}
\noi
\textup{(ii)} For any $T>0$ and $-\frac 12 < b'\leq b<\frac 12$, it holds that 
\begin{align*}
\| u\|_{X^{s,b'}_{T}} \les T^{b-b'}\|u\|_{X^{s,b}_{T}}.
\end{align*}
\textup{(iii)} It holds that 
\begin{align}
\|u\|_{L^{4}_{t,x}(\R\times \T)} \les \|u\|_{X^{0,\frac 38}} \quad \text{and} \quad
 \|u\|_{L^{4}_{T,x}}  \les T^{\frac 18} \|u\|_{X^{0,\frac 12}_{T}}.  \label{L4}
\end{align}
\end{lemma}

Finally, we have the following useful $L^p$-boundedness result related to the modulation operators $\Q_{\ll K}$. See \cite[Lemma 4.6]{GLM}.

\begin{lemma}\label{LEM:Qbd}
Let $1<p<\infty$,  $N,K\in 2^{\N}$ satisfying $K\gg N^2$, and $\Q_{\ll K}$ as in \eqref{Qproj}. Then, there exists $C_p>0$ such that
\begin{align}
\| \P_{N} \Q_{\ll K}f\|_{L^{p}_{t,x}} \leq C_{p} \| \P_{N}f\|_{L^p_{t,x}}. \label{QKLp}
\end{align}
\end{lemma}

\subsection{Gauge transform} \label{SUBSEC:Gauge}

We construct a gauge transform as in \cite{PMP}. 
For a function $f:\T \to \C$ such that $\P_{0}f=0$, we define its periodic mean zero primitive:
\begin{align*}
(\dx^{-1}f)(x) : = \mathcal{F}^{-1}[ (i\xi)^{-1}\mathcal{F}f](x)  = \frac{1}{2\pi} \int_{0}^{2\pi} \bigg( \int_{z}^{x}f(y)dy\bigg)dz. 
\end{align*}
Note that $\dx^{-1} \dx = \dx \dx^{-1} \P_{\ne0} =  \P_{\ne0}$.
We then set 
\begin{align}
F=F[u] = \dx^{-1}\P_{\neq 0} (|u|^2) \label{F}
\end{align}
and thus $\dx F = \P_{\neq 0} (|u|^2).$
We then define the gauged variables 
\begin{align}
v := \P_{+,\text{hi}}[ e^{i\be F[u]} u]
  \quad
\text{and}
\quad  w:= \P_{-, \text{hi}} u . 
\label{gauge}
\end{align}
The choice of adding a $\Pbhi$ operator to the terms in \eqref{gauge} is convenient as we cannot assume a mean zero condition for solutions to \eqref{INLS}.
The following lemma determines the dynamics for $(v,w)$ in \eqref{gauge}. 
We need to introduce further gauged variables $\wt{u}, \wt{v}, \wt{w}$, to remove a problematic linear term in the equation for $v$ defined in \eqref{gauge}.

\begin{lemma}\label{LEM:Xsbu}
Let $u \in C([0,T];H^{\infty}(\T))$ be a solution of \eqref{INLS2} and $v,w$ be the gauged variables defined in \eqref{gauge}.

\noi
{\rm(i)} The gauged variables $v$ and $w$ defined in \eqref{gauge} satisfy the following equations:
\begin{align}
\dt v + i\dx^2 v & = -2\be\Pbhip[ v \P_{-}\dx(|u|^2)] 
+i\g \Pbhip[ e^{i\be F} |u|^2 u] -i\be \Pbhip[ e^{i\be F} u \GG_{h}(|u|^2)]
   \notag \\
&  \hphantom{X}
+ i \tfrac\be\pi\big[  \tfrac{\be}{4\pi} M(u)^2 - P(u) \big] v + \tfrac\be\pi M(u) \dx v, 
\label{veq}\\
\dt w+i\dx^2 w&  =  2\be \P_{-,\textup{hi}}[ w \P_{+}\dx(|u|^2)]-i\be \P_{-, \hi}[ u \GG_{h}(|u|^2)] + i\g \P_{-, \hi}[ |u|^2 u]
=:
\mathcal{N}_{w}(u)
,
\label{weq}
\end{align}
where $\GG_h$ is as in \eqref{Lh}, and $M(u)$ and $P(u)$ denote the mass and momentum, as in \eqref{mass} and~\eqref{momentum-intro}.

\noi 
{\rm(ii)}
We define new gauged variables
\begin{align}
\label{gauge2}
\begin{split}
\wt v & : = \J_u [v], 
\quad 
\wt w : = \J_u[w ], 
\\
\wt u & : = \J_u [u], 
\quad 
\wt{F}   : = \J_u[F] = F [ \wt u ]  = \dx^{-1} \P_{\ne0}(|\wt u|^2)
, 
\end{split}
\end{align}
where $\J_u$ is as in \eqref{TT}.
Then,  $\wt u,\wt v$ satisfy:
\begin{align}
\dt \wt u + i \dx^2 \wt u 
& =
2 \be \wt u \P_+ \dx (|\wt u|^2) - i \be \wt u \GG_h (|\wt u |^2) + i \g |\wt u|^2 \wt u - \tfrac\be\pi M(\wt u) \dx \wt u
, 
\label{ueq2}
\\
\dt \wt{v} +i\dx^2 \wt{v} 
& = -2\be\Pbhip[ \wt{v} \P_{-}\dx(|\wt{u}|^2)]  
+i\g \Pbhip[ e^{i\be \wt F} |\wt u|^2 \wt u] 
\nonumber
\\
&  
\qquad 
-i\be \Pbhip[ e^{i\be \wt F} \wt u \GG_{h}(|\wt u|^2)]+ i \tfrac\be\pi\big[  \tfrac{\be}{4\pi} M(\wt u)^2 - P(\wt u) \big] \wt v
=: \NN_{\wt v}(\wt u )
.
\label{veq2b}
\end{align}

\end{lemma}

\begin{proof}
Part (ii) follows from \eqref{veq}-\eqref{weq}, the definition of the gauged variables \eqref{gauge2} and the conservation of mass $M(u)$ for solutions $u$ to \eqref{INLS2}. Thus, it suffices to show Part (i), which follows from a modification of the proof in \cite[Lemma 3.4]{PMP}, with more terms appearing here due to the zeroth frequency.
We only detail the changes needed in the proof of \cite[Lemma 3.5]{CFL} in the periodic setting.

Using \eqref{INLS2}, the fact that $\be, \g\in\R$, and that $\GG_h$ preserves real-valuedness, we~have
\begin{align}
\dt F & = \P_{\neq 0}[ i (u\dx \cj{u}-\cj{u}\dx u )+\be |u|^4]
= 
i (u\dx \cj{u}-\cj{u}\dx u )+\be |u|^4
-
\tfrac1\pi P(u) \label{dtF}
\end{align}
where $P(u)$ denotes the momentum in \eqref{momentum-intro}, 
and hence
\begin{align}
\dt F+ i\dx^2 F - \be (\dx F)^2 &  =
2i u \dx\cj{u} 
+ \tfrac\be\pi M(u) |u|^2 - \tfrac{\be}{4\pi^2} M(u)^2 - \tfrac1\pi P(u). 
\notag
\end{align}
Combining the above with \cite[(3.16)]{CFL} and the subsequent simplifications, gives
\begin{align*}
\dt v +i\dx^2 v &  
=
-2 \be \P_{+, \hi} [ v \P_{-}\dx(|u|^2)] + i \g \P_{+, \hi} [ e^{i\be F} |u|^2 u ] - i \be \P_{+, \hi} [ e^{i\be F} u \GG_h (|u|^2)]
\\
&
\quad 
+ i \tfrac\be\pi\big[  \tfrac{\be}{4\pi} M(u)^2 - P(u) \big] v + \tfrac\be\pi M(u) \dx v, 
\end{align*}
from which \eqref{veq} follows. The assertion for $\wt{v}$ in \eqref{veq2b} follows from the conservation of mass and momentum for $u$, $M(u(t)) = M(u_0)$ and $P(u(t))=  P(u_0)$, and direct computation.

To obtain \eqref{weq}, we simply apply $\P_{-, \text{hi}}$ to both sides of \eqref{INLS2} together with $\P_{-,\text{hi}}[ u \P_+ \dx(|u|^2)]  = \P_{-,\text{hi}}[ w \P_{+}\dx(|u|^2)]$, 
completing the proof.
\end{proof}

\begin{remark}
\rm 
The transformation $\J_u$ in \eqref{TT} and the mapping $v \mapsto \wt v$ explicitly exploit the conservation of mass $M(u)$ as in \eqref{mass}, to remove the problematic $\dx v$-term in \eqref{veq}. 
Note that \eqref{veq2b} still contains a linear $\wt v$-term, which could also be removed by considering $\wt v \mapsto e^{- i \frac\be\pi  [ \frac{\be}{4\pi} M(\wt u)^2 - P(\wt u) ] t} \wt v$, taking advantage of the conservation of both mass $M(\wt u)$ and momentum $P(\wt u)$ in \eqref{mass} and \eqref{momentum-intro}. 
However, we choose to not incorporate this additional phase rotation and work with $\wt v$ instead, as this linear term is harmless for the well-posedness analysis. 
Moreover, 
this allows us to simplify the unconditional uniqueness argument in Section~\ref{SEC:UU}, as we only need to justify conservation of mass $M(u)$ (see \eqref{MuN}, for example) for $u \in L^\infty_T H^s_x \cap L^4_T W^{s, 4}_x$, but not the conservation of momentum $P(u)$.

\end{remark}

\subsection{Properties of the gauged variables}

We have the following recovery result for the solution $u$ to \eqref{INLS} depending on the gauged variables $v,w$ in \eqref{gauge}:
\begin{align}
 u&= \Pbhip  u +  w+\Pblo  u  \notag \\
&  =\Pbhip [ e^{-i\be  F}  v] +  \Pbhip [ e^{-i\be  F} \P_{+,\text{lo}}( e^{i\be  F} u)]+    \Pbhip (e^{-i\be  F})\P_{0}(e^{i\be  F} u) \notag \\
&\hphantom{X}+ \Pbhip[ \P_{+}(e^{-i\be  F}) \P_{-}(e^{i\be  F} u)] +  w+\Pblo  u. \label{recovery}
\end{align}
Applying $\P_{\text{HI}}$ to both sides, we obtain
\begin{align}
\begin{split}
\PbHI   u & = \PbHIp[ e^{-i\be  F}  v] +\PbHIp[ \Pbhi(e^{-i\be  F}) \P_{+,\text{lo}}( e^{i\be   F} u)]\\
& \hphantom{X}+    \PbHIp(e^{-i\be  F})\P_{0}(e^{i\be  F}  u )  + \PbHIp[ \P_{+}(e^{-i\be  F}) \P_{-}(e^{i\be  F}  u)]  + \PbHI  w. 
\end{split} \label{PbHIu}
\end{align}
Note that in the second term, the projectors $\PbHI$ and $\Pblo$ allow us to place for free an extra projector $\Pbhi$ onto the first factor $e^{-i\be  F}$.
Moreover, the recovery formulas in \eqref{recovery} and \eqref{PbHIu} also hold for the gauged variable $\wt u$ via applying the translation operator $\TT_u$ in \eqref{TT}, replacing every instance of $u, v, w, F$ with $\wt u, \wt v, \wt w, \wt F$ as in \eqref{gauge2}, respectively.

In the following, we obtain estimates for $u$ (resp. $\wt u$) in terms of the gauged variables $v$ and $w$ in \eqref{gauge} (resp. $\wt v$ and $\wt w$ in \eqref{gauge2}). To get estimates with small constants, we need to refine the estimate for the $L^{\infty}_{T}H^s_x$-norm in \eqref{uHsbd}, and we employ an idea from \cite[Proposition 3.2]{GLM}.

\begin{lemma}
\label{LEM:uinfo}
Let $\frac 12 \leq s<2$, $0\leq T\leq 1$, $u$ be a solution to \eqref{INLS2} on $[0,T]$, 
$v,w$ be the gauged variables in \eqref{gauge},
and
$\wt u, \wt v, \wt w$ be as in~\eqref{gauge2}.

\noi{\rm(i)} The following estimate holds:
\begin{align}
\begin{split}
 \sup_{0 \le \ta \le 1} \|  u\|_{X^{ s-\ta, \ta}_{T}} 
+
\sup_{0 \le \ta \le 1} \|  \wt u\|_{X^{ s-\ta, \ta}_{T}} 
& \les 
T^{\frac 14}(1+ \|\GG_{h}\|_{\op})  
\| u\|_{L^{\infty}_{T}H^{\frac 12}_x}^2 
\|J^s_x  u\|_{\wt{L^{4}_{T,x}}} 
\\
& 
\quad 
+ T^{\frac 14} \|J_x^{\frac 12} u\|_{L^{4}_{T,x}} \| u\|_{L^{\infty}_{T}H^{\frac 12}_x} \|  u\|_{L^{\infty}_{T}H^{s}_x}
+
 \| u\|_{L^{\infty}_{T}H^{s}_{x}}
.
\end{split}
\label{Xsbu}
\end{align}

\smallskip
\noi{\rm(ii)}
Given $\ta>0$ and $M\in\N$ sufficiently large, we define the quantities
\begin{align}
G_1(\uu, \vv,\ww)
& :=
T^{\frac 18}
( 1+\| \uu\|_{L^{\infty}_{T}H^{\frac 12}_{x}}^2 + 
\| \uu\|_{L^{\infty}_{T}H^{\frac 12}_{x}}
\| \uu\|_{L^\infty_T H^{\max(\frac12,s-\frac 34+)}_{x}})
\| \vv\|_{X^{s,\frac 12}_{T}}   
\nonumber
\\
& 
\quad 
+ T^{\frac 14} (1+\|\uu\|_{L^{\infty}_{T}H^{ \frac12}_{x}}^5 )
\|\uu\|_{L^{\infty}_{T}H^{ \frac12}_{x}}^2
\|\uu\|_{L^{\infty}_{T}H^{ \max(\frac12, s-\frac{3}{4}+)}_{x}} + T^{\frac 18} \| \ww\|_{X^{s, \frac 12}_{T}}  
,
\notag
\\
G_2(\uu,\vv,\ww)
&
:=
\| \uu(0) \|_{H^{s}} 
 + \| \vv \|_{X^{s,\frac 12+}_T}
+\| \ww \|_{X^{s,\frac 12+}_T} 
\nonumber
\\
&\qquad+
T^{\frac 34}M^3 ( 1 + \| \GG_h\|_{\op}) \| \uu\|_{L^\infty_T H^{ \frac12}_{x}}^2( \|\uu\|_{L^\infty_T H^{ \frac12}_x}+  \| J_x^{\frac 12} \uu \|_{\wt{L^4_{T,x}}})
\nonumber
\\
&
\qquad+ 
 M^{-\ta}( 1 +  \| \uu\|_{L^\infty_T H^{ \frac12}_x}^4 ) (1 +  \|\vv\|_{X^{s,\frac 12+}_{T}} )  \| \uu\|_{L^\infty_T H^{ \max(\frac12, s-\frac34+)}_x}.
\notag
\end{align}
Then, the following estimates hold
\begin{align}
\| J_x^{s}\PbHI u\|_{\wt{L^{4}_{T,x}}} 
 &
\les 
G_1(u, \wt v, w)
, 
&
\| J_x^{s}\PbHI \wt u\|_{\wt{L^{4}_{T,x}}} 
& \les   
G_1(\wt u, \wt v, w)
, 
\label{Jsu}
\\
\|  u\|_{L^{\infty}_{T}H^{s}_x} 
&
\les 
G_2( u, \wt v, w) 
, 
&
\| \wt u \|_{L^\infty_T H^s_x}
&
\les 
G_2 (\wt u , \wt v, w)
.
\label{uHsbd}
\end{align}

\end{lemma}

\begin{proof}

We first prove \eqref{Xsbu} in (i) by arguing as in \cite[Proposition 3.2]{MP}.
The main point is that we have
\begin{align}\label{uXsb1a}
\sup_{0\le \ta \le 1} 
\| u\|_{X_{T}^{ s-\ta, \ta}} \les \| \dt u +i\dx^2 u\|_{L^{2}_{T}H^{ s-1}_{x}} + \| u\|_{L^{\infty}_{T}H^{ s}_{x}}.
\end{align}
To estimate the first contribution above, we consider the equation for $ u$ \eqref{INLS2}. Let $\mathcal{Q}_{h} \in \{ \text{Id}, \mathcal{G}_{h}\}$.
When $s\leq 1$, Sobolev embedding and the fractional Leibniz rule give
\begin{align}
\label{uXsb1aa}
\begin{split}
\| \mathcal{Q}_{h}( | u|^2)  u\|_{L^{2}_{T}H^{s-1}_x}
\leq \| \mathcal{Q}_{h}(| u|^2)  u\|_{L^{2}_{T,x}} 
& \les T^\frac14 \|  \mathcal{Q}_{h}( | u|^2)\|_{L^{\infty}_{T}L^{4}_x} \| u\|_{L^{4}_{T,x}} \\
& \les T^{\frac 14} \|\mathcal{Q}_{h}\|_{\text{op}}  \|J_x^{\frac 14} ( | u|^2)\|_{L^{\infty}_{T}L^{2}_x} \|  u\|_{L^{4}_{T,x}} \\
& \les  T^{\frac 14} \|\mathcal{Q}_{h}\|_{\text{op}} \| u\|_{L^{\infty}_{T}H^{\frac 38}_x}^2 \| u\|_{L^{4}_{T,x}}
.
\end{split}
\end{align}
When $s\in (1,2)$, we have 
\begin{align}
\label{uXsb1ab}
\begin{split}
\| \mathcal{Q}_{h} (| u|^2) u\|_{L^{2}_{T}H^{s-1}_x} 
& 
\les \| \mathcal{Q}_{h}(| u|^2) u\|_{L^{2}_{T, x}}
+
\| \mathcal{Q}_{h}(| u|^2) \dx  u\|_{L^{2}_{T, x}} + \| \mathcal{Q}_{h}\dx(| u|^2)  u\|_{L^{2}_{T}L^{2-}_{x}} \\
& 
\les T^{\frac 14} \| \mathcal{Q}_{h}(| u|^2) \|_{L^{\infty}_{T}L^{4}_{x}} \| J_x  u\|_{L^{4}_{T, x}} 
+  \| \mathcal{Q}_{h} \dx(| u|^2)\|_{L^{2}_{T, x}} \| u\|_{L^{\infty}_{T}L^{\infty-}_{x}}  \\
& \les  T^{\frac 14} \| \mathcal{Q}_{h}\|_{\text{op}} \|u\|_{L^{\infty}_{T}H^{\frac 12}_x}^2 \| J_x^{s}u\|_{L^{4}_{T,x}}
.
\end{split}
\end{align}
By \eqref{prodest1}, Minkowski's and H\"{o}lder's inequality, we have 
\begin{align*}
\| 2\be u \P_{+}\dx(| u|^2)\|_{L^{2}_{T}H^{s-1}_x}\les T^{\frac 14} \| u\|_{L^{\infty}_{T}H^{\frac 12}_{x}}^2 \|J^{s}_x  u\|_{\wt{L^{4}_{T,x}}} + T^{\frac 14} \|J^{\frac 12}_x  u\|_{L^{4}_{T,x}} \| u\|_{L^{\infty}_{T}H^{\frac 12}_x} \| u\|_{L^{\infty}_{T}H^{s}_x}.
\end{align*}
Combining the estimates above, we obtain \eqref{uXsb1a}
for $u$. 
The bound for $\wt u$ follows from \eqref{uXsb1a} and using the equation \eqref{ueq2} for the first term. Comparing \eqref{INLS} and \eqref{ueq2}, there is only one additional term $-\frac\be\pi M(\wt u) \dx \wt u$, whose $L^2_T H^{s-1}_x$-norm is controlled by the first term in \eqref{Xsbu}. 
This completes the proof of \eqref{Xsbu}.

We now prove the first estimate in \eqref{Jsu} in (ii). 
First, we prove some auxiliary estimates on $v$, $w$, and $|u|^2$. By \eqref{L4} with
\eqref{gauge2}, we have
\begin{align}
\| J_x^{s} v\|_{\wt{L^{4}_{T,x}}}=\| J_x^{s} \wt{v}\|_{\wt{L^{4}_{T,x}}} \les  T^{\frac 18} \|\wt{v}\|_{X^{s, \frac12}_{T}} \quad \text{and}\quad 
\| J_x^{s}  w\|_{\wt{L^{4}_{T, x}}}=\| J_x^{s} \wt{w}\|_{\wt{L^{4}_{T,x}}} \les  T^{\frac 18} \|w\|_{X^{s, \frac12}_{T}}.
 \label{LpLqvw}
\end{align}

By the recovery formula in \eqref{PbHIu} and the triangle inequality, we~have
\begin{align}
\begin{split}
\| \P_{N}\PbHI u\|_{L^{4}_{T,x} }
& 
\les 
\| \P_{N}\PbHIp[ e^{-i\be F} v]\|_{L^{4}_{T,x}}
+\|\P_{N} \PbHIp[ \Pbhi(e^{-i\be F}) \P_{+,\text{lo}}( e^{i\be F} u)]\|_{L^{4}_{T,x}}
 \\
& \hphantom{XX} 
+
\|\P_{0}(e^{i\be F}u)\|_{L^\infty_T} \| \P_{N}\PbHIp (e^{-i\be F})\|_{L^{4}_{T,x}} 
 \\
&\hphantom{XX} 
+\|\P_{N}\PbHIp [ \P_{+}(e^{-i\be F}) \P_{-}(e^{i\be F}u)] \|_{L^{4}_{T,x}} + \|\P_{N}\PbHI w\|_{L^{4}_{T,x}}
\\
&
=: \I_N+\II_N+\III_N+\IV_N+ \|\P_{N}\PbHI w\|_{L^{4}_{T,x}}
\end{split} \label{Jsu1}
\end{align}
and we need to estimate each of these terms. 
For the first contribution in \eqref{Jsu1}, we have
\begin{align}
\I_N \sim \|\P_{N}[ e^{-i\be F} v]\|_{L^{4}_{T,x}} \les \|\P_{\ll N}(e^{-i\be F}) \wt{\P}_{N} v\|_{L^{4}_{T,x}} + \|\P_{N}\big[\P_{\ges N}(e^{-i\be F}) v \big]\|_{L^{4}_{T,x}}.
\label{INa}
\end{align}

\noi 
Then, summing the first term in \eqref{INa}, we get
\begin{align*}
\sum_{N} N^{2s} \|\P_{\ll N}(e^{-i\be F}) \wt{\P}_{N} v\|_{L^{4}_{T,x}}^{2} \les  \| J^{s}_x v\|^2_{\wt{L^{4}_{T,x}}}
.
\end{align*}
For the second term in \eqref{INa}, by
\eqref{LinftyE11}, we have 
\begin{align*}
 \|\P_N [\P_{\ges N}(e^{-i\be F}) v] \|_{L^{4}_{T,x}} \les \| \P_{\ges N}(e^{-i\be F})\|_{L^{\infty}_{T, x}} \| v\|_{L^4_{T,x}} \les N^{-1+} \|u\|_{L^{\infty}_{T}H^{\frac 12}_x}^2  \| J_x^{\frac 12}v\|_{L^4_{T,x}}
 ,
\end{align*}
so that if $s<1$, then
\begin{align*}
\sum_{N} N^{2s}  \| \P_N [\P_{\ges N}(e^{-i\be F}) v]\|_{L^{4}_{T,x}}^2 \les  \|u\|_{L^{\infty}_{T}H^{\frac 12}_x}^4  \| J_x^{\frac 12}v\|_{L^4_{T,x}}^2.
\end{align*}
When $1\leq s<2$, using \eqref{LinftyE11}, H\"older's inequality, and Sobolev embedding, we obtain
\begin{align*}
&
\|\P_{N}\big[\P_{\ges N}(e^{-i\be F}) v \big]\|_{L^{4}_{T,x}}
\\
& \les N^{-1} \|\dx \P_{N}\big[\P_{\ges N}(e^{-i\be F}) v \big]\|_{L^{4}_{T,x}} \\
& \les N^{-1} \Big( \| \P_{\gg N}(e^{-i\be F})\|_{L^{\infty}_{T, x}} \| \P_{\ne0} (|u|^2) v\|_{L^{4}_{T, x}} + \|\P_{\ges N}u\|_{L^{\infty}_{T}L^{4+}_x} \| u v\|_{L^{4}_{T}L^{\infty-}_{x}}  \\
& \hphantom{XXXXXX}+ \| \P_{\ges N}(e^{-i\be F}) \|_{L^{\infty}_{T, x}} \| \dx v\|_{L^{4}_{T,x}} \Big) \\
& \les N^{-2+} (1+\|u\|_{L^{\infty}_{T}H^{\frac 12}_x}^{2})\|u\|_{L^{\infty}_{T}H^{\frac 12}_x}^{2} \|J_x v\|_{L^{4}_{T,x}} \\
& \hphantom{XXX} +N^{-\frac 34 -\max(\frac 12,s-\frac 34+)} \|u\|_{L^{\infty}_{T}H^{\frac 12}_x} \|u\|_{L^{\infty}_{T}H^{\max(\frac 12, s-\frac 34+)}_x} \|J^{\frac 12}_x v\|_{L^{4}_{T,x}}.
\end{align*}
Therefore, combining the estimates above with and summing in $N$, gives
\begin{align}
\bigg(\sum_{N} N^{2s} \I_N^2 \bigg)^{\frac 12}  \les T^{\frac 18} \big(
1+\|u\|_{L^{\infty}_{T}H^{\frac 12}_x}^{2} 
+ \| u\|_{L^\infty_T H^{\frac12}_x }
\|u\|_{L^\infty_T H^{ \max(\frac12, s-\frac34+)}_x}
\big)
\| J^s_x v \|_{\wt{L^4_{T,x}}}.
\label{Jsu2Ia}
\end{align}

Considering the third contribution in \eqref{Jsu1},
we estimate the first factor 
by \eqref{P0} and Cauchy-Schwarz:
\begin{align}
\label{P0e}
\|\P_{0}(e^{i\be F}u) \|_{L^\infty_T} & \les \| u\|_{ L^\infty_T L^2_x}.
\end{align}
For the second factor, 
if $s < 1$, using \eqref{LinftyE11},  we have
\begin{align}
\bigg(\sum_{N} N^{2s} \| \P_{N}\P_{+,\HI} (e^{-i\be F})\|_{L^{4}_{T,x}}^2 \bigg)^{\frac 12}
\les
T^{\frac14} \|u\|_{L^\infty_T H^{ \frac12}_x}^2,
\label{Jsu211b}
\end{align}
while for $1 \leq  s \le 2$, using \eqref{LpE} with $p=4$, we obtain
\begin{align}
\begin{split}
\bigg(\sum_{N}N^{2s}\| \P_{N}\P_{+,\HI} (e^{-i\be F})\|_{L^{4}_{T,x}}^2 \bigg)^{\frac 12}
&
\les  \bigg(\sum_{N}N^{2s-4+} \bigg)^{\frac 12} T^{\frac 14}  \|u\|_{L^{\infty}_{T}H^{\frac 12}_x}^4  \\ 
& \hphantom{X} +T^{\frac 14} \bigg(\sum_{N}N^{2s-2+} \| J_{x}^{\frac 14}\P_{\ges N}u\|_{L^{\infty}_{T}L^{2}_x}^2 \bigg)^{\frac 12} \| u\|_{L^{\infty}_{T}H^{\frac 12}_x} \\
& \les T^{\frac 14}\|u\|_{L^{\infty}_{T}H^{\frac 12}_x}\big[   \|u\|_{L^{\infty}_{T}H^{\frac 12}_x}^3+ \|  u\|_{L^{\infty}_{T}H^{s-\frac 34+}_x}\big].
\end{split}
\label{Jsu211c}
\end{align}
Thus, combining \eqref{P0e}, \eqref{Jsu211b}, and \eqref{Jsu211c}, we conclude that
\begin{align}
\label{Jsu2III}
\bigg( \sum_{N} N^{2s} \III_{N}^2 \bigg)^{\frac 12} 
&
\les
T^{\frac14} 
\|u\|_{L^{\infty}_{T}H^{\frac 12}_x}^2 \big[   \|u\|_{L^{\infty}_{T}H^{\frac 12}_x}^3 + \|  u\|_{L^{\infty}_{T}H^{\max(\frac 12, s-\frac 34+)}_x}\big]. 
\end{align}

For the second term in \eqref{Jsu1}, we notice that by impossibility of frequency interactions, we must have    
\begin{align*}
\II_{N}
&=\|\P_{N} \PbHIp[ \wt{\P}_{N} \Pbhi(e^{-i\be F}) \P_{+,\text{lo}}( e^{i\be F} u)]\|_{L^{4}_{T,x}},
\end{align*}
and thus by \eqref{Jsu211b}, \eqref{Jsu211c}, and Bernstein's inequality, we have 
\begin{align}
\begin{split}
\bigg( \sum_{N} N^{2s} \II_{N}^2 \bigg)^{\frac 12} 
& \les  \|  \P_{+,\text{lo}}( e^{i\be F} u)\|_{L^{\infty}_{T,x}} \bigg( \sum_{N}N^{2s} \| \wt{\P}_{N} \Pbhi(e^{-i\be F})\|_{L^{4}_{T,x}}^{2} \bigg)^{\frac 12} \\
& \les T^{\frac14} 
\|u\|_{L^{\infty}_{T}H^{\frac 12}_x}^2 \big[   \|u\|_{L^{\infty}_{T}H^{\frac 12}_x}^3+ \|  u\|_{L^{\infty}_{T}H^{\max(\frac 12, s-\frac 34+)}_x}\big].
\end{split} \label{Jsu2II}
\end{align}

Next, we consider $\IV_{N}$ in \eqref{Jsu1}. Using Bernstein's inequality and \eqref{LinftyE11}, we have 
\begin{align}
\| \P_{M}(e^{i\be F}u)\|_{L^{\infty}_{x}} & \leq \| \P_{M}[\P_{\ll M} e^{i\be F}\cdot  \wt{\P}_{M}u]\|_{L^{\infty}_{x}} + \|\P_{M} [ \P_{\ges M}e^{i\be F} \cdot u]\|_{L^{\infty}_{x}} 
\nonumber
\\
& \les  \|\wt{\P}_{M} u\|_{L^{\infty}_{x}}   + M^{\frac 12} \|\P_{\ges M}e^{i\be F}\|_{L^{\infty}_x} \|u\|_{L^2_x} 
\nonumber
\\
& \les \|u\|_{H^{\frac 12}_x}+ M^{-\frac 12+} \|u\|_{H^{\frac 12}_x}^{3}.
\label{IVNa}
\end{align}
Therefore, if $s<1$, we have from \eqref{LinftyE11} and \eqref{IVNa} that
\begin{align}
\IV_{N} & \les \sum_{N_1 \ges N\vee N_2} \| \P_{N_1}e^{-i\be F} \|_{L^{4}_{T,x}} \| \P_{N_2} (e^{i\be F}u)\|_{L^{\infty}_{T,x}} \notag \\
& \les N^{-1+} T^{\frac 14} \|u\|_{L^{\infty}_{T}H^{\frac 12}_x}^{3}(1+\|u\|_{L^{\infty}_{T}H^{\frac 12}_x}^{2}),
 \notag
\end{align}
while if $1\leq s <2$, we instead use \eqref{LpE} which implies 
\begin{align*}
\IV_{N} & \les T^{\frac 14} N^{-s-} \|u\|_{L^{\infty}_{T}H^{\frac 12}_x}^{2}(1+\|u\|_{L^{\infty}_{T}H^{\frac 12}_x}^2) ( \| u\|_{L^\infty_T H^\frac12_x}^3+ \|u\|_{L^{\infty}_{T}H^{s-\frac 34+}_x}).
\end{align*}
Combining the two estimates above, we get
\begin{align}
\bigg( \sum_{N} N^{2s} \IV_{N}^{2} \bigg)^{\frac 12} 
\les 
T^{\frac14}
( 1 + \| u \|_{L^\infty_T H^{ \frac12}_x}^5 ) \| u \|_{ L^\infty_T H^{ \frac12}_x }^2
 \| u\|_{L^\infty_T H^{ \max(\frac 12, s - \frac34+) }_x} 
. 
\label{Jsu2IV}
\end{align}

\noi
Therefore, the first estimate in \eqref{Jsu} follows once we combine \eqref{Jsu1}, \eqref{Jsu2Ia}, 
\eqref{Jsu2II}, 
\eqref{Jsu2III},
\eqref{Jsu2IV}, 
together
with 
\eqref{LpLqvw} for the $v$ term in \eqref{Jsu2Ia} and \eqref{L4} for the last term in \eqref{Jsu1}.
The second estimate in \eqref{Jsu} follows from the approach above, since the recovery formula \eqref{PbHIu} extends to $\wt u$ by replacing all $v,w$ with $\wt v, \wt w$, where we apply \eqref{L4} to the corresponding $\wt v$ term in \eqref{Jsu2Ia} and to the $\wt w$ term in \eqref{Jsu1} as follows
$$\| J^s_x \P_\HI \wt w\|_{\wt{L^4}_{T,x}} = \| J^s_x \P_\HI w \|_{\wt{L^4_{T,x}}} \les \| w\|_{X^{s, \frac38}_T},  $$
completing the proof of \eqref{Jsu}.

We now move on to the first estimate in \eqref{uHsbd}. We first control the high-frequency portion of $u$ using \eqref{PbHIu}. If we followed the same strategy as in \eqref{Jsu1}, we would no longer obtain extra factors of $T$ from the terms $\II-\IV$, so we need to refine the estimates for those terms. Moreover, our estimate for $\I$ includes contributions from the norms of $u$ which causes issues in the bootstrap argument if we do not assume small data. We need to ensure there is a small factor at the front of such terms.  Fix $M\in \N$ large, and decompose $ u = \P_{\leq M}u+\P_{>M}u$.
For the contribution from $\P_{\leq M}u$, we use the Duhamel formula for $u$ and H\"{o}lder's inequality in time to get
\begin{align}
\begin{split}
\|\P_{\leq M} u \|_{L^{\infty}_{T}H^{s}_x}
\les& \| u_0\|_{H^{s}} +  T^{\frac 34}\| \P_{\leq M} [u \P_{+}\dx(|u|^2)]\|_{L^{4}_{T}H^{s}_x} \\
& + T\big\{\| \P_{\leq M}( u \GG_{h}(|u|^2))\|_{L^{\infty}_{T}H^{s}_x}  
 +  \| \P_{\leq M}( |u|^2 u)\|_{L^{\infty}_{T}H^{s}_x} \big\}.
 \end{split}
\label{JsuPLO}
\end{align}

\noi
The third and fourth terms in \eqref{JsuPLO} are easily bounded by
\begin{align}
\begin{split}
TM^{s+1} (1+\|\GG_{h}\|_{\text{op}})  \|u\|_{ L^{\infty}_{T}H^{\frac 12}_x}^{3}
,
\end{split} 
\notag
\end{align}

\noi
which can be proved easily using similar ideas to Lemma~\ref{LEM:dxs1}~(i).
Then, using the projection $\P_{\leq M}$ followed by \eqref{prodest1}, Sobolev embedding, and Minkowski's inequality, we have
\begin{align*}
\| \P_{\leq M}[ u \P_{+}\dx(|u|^2)]\|_{L^{4}_{T}H^{s}_x}& \les
\| \P_{\le M } \dx [ u \P_+(|u|^2) ] \|_{L^4_T H^s_x} 
+ 
\| \P_{\le M } [ \dx u \cdot \P_+(|u|^2)] \|_{L^4_T H^s_x}
\\
&
\les 
M^{1+s} 
\big\{ 
\| u \P_+(|u|^2) \|_{L^4_T L^2_x} 
+ 
\| \dx u \cdot \P_+(|u|^2) \|_{L^4_T H^{-1}_x}
\big\}
\\
& \les M^{3}\|u\|_{L^{\infty}_{T}H^{\frac 12}_x}^2(T^{\frac 14}\|u\|_{L^{\infty}_{T}H^{\frac 12}_x} + \|J_x^{\frac 12}u\|_{L^4_{T}\wt{L^4_x}} )\\
&\les M^{3}\|u\|_{L^{\infty}_{T}H^{\frac 12}_x}^2 (\|u\|_{L^{\infty}_{T}H^{\frac 12}_x}+\|J_x^{\frac 12}u\|_{\wt{L^4_{T,x}}} ),
\end{align*}
using $s\leq 2$.
Combining these estimates, we have
\begin{align}\label{uHsPM}
\| \P_{\leq M} u \|_{L^\infty_T H^s_x} 
& 
\les 
\|u_0\|_{H^s} + T^{\frac 34} M^{3}(1 + \|\GG_{h}\|_{\op})\|u\|_{L^{\infty}_{T}H^{\frac 12}_x}^2 (\|u\|_{L^{\infty}_{T}H^{\frac 12}_x}+\|J_x^{\frac 12}u\|_{\wt{L^4_{T,x}}} ).
\end{align}

Now we consider the contribution from $\P_{>M}u$ for which we use \eqref{PbHIu}, leading to the terms in \eqref{Jsu1} with this additional projection $\P_{>M}$ and in $L^{\infty}_{T}H_x^s$. 
For the contribution of the first term in \eqref{PbHIu}, we have
\begin{align}
\begin{split} 
\| \P_{>M}J_x^s &\PbHIp[ e^{-i\be F}v]\|_{L^{\infty}_{T}L^2_x}  \\
& \les \| \P_{>M}J_x^s[\P_{\ll M}( e^{-i\be F}) \P_{\ges M }v]\|_{L^{\infty}_{T}L^2_x}  + \| \P_{>M}J_x^s[\P_{\ges M}( e^{-i\be F})v]\|_{L^{\infty}_{T}L^2_x}. 
\end{split} \label{rec1term}
\end{align}

\noi
For the first term in \eqref{rec1term}, by the fractional Leibniz rule, and \eqref{LpE}, we have
\begin{align}
\begin{split}
&
\| \P_{>M} 
J_x^s[\P_{\ll M}( e^{-i\be F}) \P_{\ges M }v]\|_{L^{\infty}_{T}L^2_x} 
\\
& \les \|J_x^{s} \P_{\ll M}(e^{-i\be F})\|_{L^{\infty}_{T}L^{4}_{x}} \|\P_{\ges M}v\|_{L^{\infty}_{T}L^4_x} 
+ \|\P_{\ll M}(e^{-i\be F})\|_{L^{\infty}_{T,x}} \|J_x^{s}\P_{\ges M}v\|_{L^{\infty}_{T}L^2_x} \\
& \les M^{-\frac 14}\|u\|_{L^{\infty}_{T}H^{\frac 12}_x}(1+\|u\|_{L^{\infty}_{T}H^{\frac 12}_x}^3+  \|u\|_{L^{\infty}_{T}H^{\max(\frac 12, s-\frac 34+)}_x}  )\|v\|_{L^\infty_T H^{\frac12}_x} + \|  v \|_{L^\infty_T H^s_x}, 
\end{split}
\label{rec1term1}
\end{align}
for $s<2$.
Next, in a similar way, for the second term in \eqref{rec1term}, we get
\begin{align}
\begin{split}
\|& \P_{>M}J_x^s[\P_{\ges M}( e^{-i\be F})v]\|_{L^{\infty}_{T}L^2_x}  \\
& \les \| J_x^{s}\P_{\ges M}e^{-i\be F}\|_{L^{\infty}_{T}L^{2+}_{x}} \|v\|_{L^{\infty}_{T}L^{\infty-}_{x}} + \|\P_{\ges M}e^{-i\be F}\|_{L^{\infty}_{T,x}}\|J^{s}_x v\|_{L^{\infty}_{T}L^2_x} \\
& \les M^{-\ta} (1+\|u\|_{L^{\infty}_{T}H^{\frac 12}_x}^3)
\|u\|_{L^{\infty}_{T}H^{\max(\frac 12,s-\frac 34+)}_x} \| {v}\|_{L^\infty_T H^s_x}  ,
\end{split} \label{rec1high}
\end{align}
for some $\ta>0$, 
where we used \eqref{LpE}.

For the contribution from the second term in \eqref{PbHIu}, from the $\Pblo$ and for $M$ sufficiently large, we note that 
\begin{align*}
\|\P_{>M}&J_x^s \PbHIp[ \Pbhi(e^{-i\be F}) \P_{+,\text{lo}}( e^{i\be F} u)]\|_{L^{\infty}_{T}L^2_x} \\
& \les \|J_x^s \PbHIp[ \P_{\ges M}(e^{-i\be F}) \P_{+,\text{lo}}( e^{i\be F} u)]\|_{L^{\infty}_{T}L^2_x}.
\end{align*}
Then, we use the fractional Leibniz rule and similar estimates as in \eqref{rec1high} to bound this by
\begin{align}
\| J_x^{s}\P_{\ges M}(&e^{-i\be F})\|_{L^{\infty}_{T}L^2_x}\|  \P_{+,\text{lo}}( e^{i\be F} u)\|_{L^{\infty}_{T,x}}+\|\P_{\ges M}(e^{-i\be F})\|_{L^{\infty}_{T}L^2_x} \| J_x^{s}\P_{+,\text{lo}}( e^{i\be F} u)\|_{L^{\infty}_{T,x}}
\nonumber 
\\
& \les M^{-\ta} (1+\|u\|^4_{L^{\infty}_{T}H^{\frac 12}_x})
\|u\|_{L^{\infty}_{T}H^{\max(\frac 12,s-\frac 34+)}_x},
\label{rec2term}
\end{align}
for some $\ta>0$.
In the course of this estimate, we have also essentially bounded the contribution coming from the similar third term in \eqref{PbHIu}.
For the fourth term in \eqref{PbHIu}, we note that frequency signs imply 
\begin{align*}
\P_{>M}\P_{+, \HI}[ \P_{+}(e^{-i\be F})  \P_{-}(e^{i\be F}u) ]= \P_{>M}\P_{+,\HI}[ \P_{+}\P_{>M}(e^{-i\be F})  \P_{-}(e^{i\be F}u) ],  
\end{align*}
for which we proceed as the first term in \eqref{rec1high} with Sobolev inequality to get
\begin{align}
     &
 \| J^s_x \P_{>M}\P_{+, \HI}[ \P_{+}(e^{-i\be F})  \P_{-}(e^{i\be F}u) ] \|_{L^\infty_T L^2_x}
\nonumber
\\
    &\quad 
    \les \| J^s_x \P_+ \P_{>M}  e^{-i\be F} \|_{L^\infty_T L^{2+}_x} \| \P_-(e^{i\be F} u )\|_{L^\infty_T L^{\infty-}_x}
\nonumber
    \\
    &
\quad 
    \les M^{-\ta } (1+ \|u\|^4_{L^\infty_T H^{\frac12}_x}) \|u\|_{L^\infty_T H^{\max(\frac12, s-\frac34+)}_x}. 
\label{rec3term}
\end{align}
Then, the first estimate in \eqref{uHsbd} follows from combining 
\eqref{uHsPM}, \eqref{PbHIu}, \eqref{rec1term}-\eqref{rec1high}, \eqref{rec2term}, \eqref{rec3term}, together with applying \eqref{YsCTHs} to the last term in \eqref{PbHIu} and to the $v$ terms in in \eqref{rec1term1}-\eqref{rec1high}, after replacing $\| v\|_{L^\infty_T H^s}$ with $\|\wt v\|_{L^\infty_T H^s_x}$. 

For the second estimate in \eqref{uHsbd}, we handle the low frequency part $\P_{\le M} \wt u$ by using the Duhamel formulation for \eqref{ueq2}, where we estimate  the last contribution in \eqref{ueq2} by
\begin{align*}
T \| M(\wt u ) \P_{\le M} \dx \wt u\|_{L^\infty_T H^s_x} \les T M^2 \| \wt u\|^3_{L^\infty_T H^\frac12_x}
.
\end{align*}
The high frequency part is handled as for $\P_{>M}u$, only changing the estimate for the last term in \eqref{PbHIu}, where we first replace the $\wt w$-norm by $w$ and then use \eqref{YsCTHs}.  
 This completes the proof of \eqref{uHsbd}. \qedhere

\end{proof}

We now extend \eqref{Xsbu}, \eqref{Jsu}, and \eqref{uHsbd}
to difference estimates.

\begin{lemma}[Difference estimates] \label{LEM:diffests}
Let $\frac 12 \leq s<2$, $0\leq T\leq 1$, and $u_j$ be smooth solutions to \eqref{INLS2} on $[0,T]$ with initial data $u_j(0)=u_{0,j}$, $j=1,2$.  Let $F_j=F_j [u_j]$ be as defined in \eqref{F}, set $v_j= v_j [u_j]$ and $w_j= w_{j}[u_j]$ as in \eqref{gauge}, $U: =u_1-u_2$ and $W:=w_1-w_2$. 
Also, given $\wt{u}_j : = \J_{u_j}[u_j]$ and $\wt{v}_j : = \J_{u_j}[v_j]$ as in \eqref{gauge2}, $j=1,2$, set
$\wt{V}:= \wt{v}_1 - \wt{v}_2 : = \J_{u_1}[v_1]- \J_{u_2}[v_2]$. 
Lastly, let $R_{s_0}:=R_{s_0}(u_1,u_2): = \|u_1\|_{L^{\infty}_{T}H^{s_0}_{x}}+\|u_2\|_{L^{\infty}_{T}H^{s_0}_{x}}$ and $\s : = \max(\tfrac 12 , s-\tfrac34+).$

\noi{\rm(i)} The following estimates hold:
\begin{align}
\sup_{0 \le \ta \le 1}\| U\|_{X^{ \frac 12-\ta, \ta}_{T}} 
&
\les 
\wt G_0(u_1,u_2,U)
\quad 
\text{and}
\quad 
\sup_{0 \le \ta \le 1 } \|\wt U \|_{X^{\frac12- \ta, \ta}_T}
\les 
\wt G_0(\wt u_1, \wt u_2, \wt U),
 \label{XsbU}
\end{align}
where $G_0$ is given by
\begin{align}
\wt G_0(\uu_1, \uu_2,  \UU)
&
:=
 \|\UU\|_{L^{\infty}_{T}H^{\frac 12}_x}+ T^\frac14(1+\|\GG_{h}\|_{\op}) R_{\frac 12}(\uu_1, \uu_2)
\nonumber
\\
&
\quad 
\times 
\Big\{R_{\frac 12}(\uu_1,\uu_2)
\| J^{\frac 12}_{x}\UU\|_{\wt{L^4_{T,x}}}
+  \max_{j=1,2}  \| J^{\frac 12}_{x}\uu_j\|_{\wt{L^4_{T,x}}} 
\|\UU\|_{L^{\infty}_{T}H^{\frac 12}_x}
\Big\}
.
\notag
\end{align}

\smallskip
\noi{\rm(ii)}
Given $\ta>0$, $M, \Ld\in \N$ sufficiently large, we define the quantities
\begin{align*}
&
\wt G_1(\uu_1, \uu_2, \vv_1, \mathbf{z}, \UU, \VV, \WW)
\\
&
:=
 T^{\frac 18} \big[1+Q\big(R_{\s}(\uu_1,\uu_2) \big)\big]
\Big\{ 
\| \VV\|_{X^{s,\frac 12+}_{T}} 
+\| \WW\|_{X^{s,\frac 12+}_{T}}  
+ 
\| \Pi_{>\Ld} \mathbf{z} \|_{X^{s, \frac12+}_T}
\\
& \hspace{5cm}
+
\big[ \|e^{i\be F[\uu_1]}-e^{i\be F[\uu_2]}\|_{L^{\infty}_{T,x}} + \|\UU\|_{L^{\infty}_{T} H^{\s}_x}\big]  
(
1+ \| J^s \vv_1\|_{\wt{L^4_{T,x}}} 
 )  
\Big\}
\\
&
\quad
+
T^\frac18
\Ld^{\frac54} \big| M\big(\uu_1(0)\big) - M\big(\uu_2(0)\big) \big|   \| \mathbf{z}\|_{L^\infty_T H^s_x}
, 
\\
&
\wt G_2(\uu_1, \uu_2, \vv_1,\mathbf{z}, \UU, \VV, \WW)
\\
&
:=
 \|\UU(0)\|_{H^s}
        +
        \| \VV\|_{X^{s, \frac12+}_T}
        +
        \| \WW\|_{X^{s, \frac12+}_T}
        \\
        &
        \quad 
        + T^\frac34 M^3(1 + \|\GG_h\|_{\op})
        \big(
        1 + R^2_\frac12(\uu_1,\uu_2) + \max_{j=1,2}\| J^\frac12_x \uu_j \|_{\wt{L^4_{T,x}}}
        \big)
        \\
        &\qquad \times 
        \Big\{
        \| \UU\|_{L^\infty_T H^\frac12_x} + \| J_x^\frac12 \UU \|_{\wt{L^4_{T,x}}}
        + \| e^{\pm i \be F[\uu_1]} - e^{\pm i \be F[\uu_2]} \|_{L^\infty_{T,x}}
        \Big\}
        \\
        &
        \quad 
        + M^{-\ta} Q\big(R_\s (\uu_1,\uu_2) \big) \Big\{ 
        \|\VV\|_{X^{s, \frac12+}_T}
        +
        \| \UU \|_{L^\infty_T H^s_x}
        +
        \| e^{i\be F[\uu_1]} - e^{i \be F[\uu_2]} \|_{L^\infty_{T,x}}
        \Big\}
        \\
        &
        \quad 
        + \| \P_{\ll M}[ e^{-i \be F[\uu_1]} - e^{-i \be F[\uu_2]}] \|_{L^\infty_{T,x}} \| \P_{\ges M } \vv_1 \|_{X^{s, \frac12+}_T}
        \\
        &
        \quad 
        + \Ld T \big| M(\uu_1(0)) - M(\uu_2(0)) \big| \| \mathbf{z} \|_{X^{s, \frac12+}_T} + \| \Pi_{>\Ld} \mathbf{z} \|_{X^{s, \frac12+}_T}, 
\end{align*}
where $Q$ is a polynomial of degree at least 1, $F$ is as in \eqref{F}, and $M(\uu)$ denotes the mass of $\uu$ as in \eqref{mass}.
 Then, the following  estimates hold
\begin{align}
\| J^{s}\PbHI U\|_{\wt{L^{4}_{T,x}}}  
&
\les  \wt G_1( u_1,u_2, \wt v_1, \wt v_1, U, \wt V, W )
, 
\label{JsU}
\\
\| J^{s}\PbHI \wt U\|_{\wt{L^{4}_{T,x}}}  
&\les 
\wt G_1(\wt u_1, \wt u_2, \wt v_1, w_1, \wt U, \wt V, W)
,
\label{JsU2}
\\
\| U\|_{L^{\infty}_{T}H^{s}_x} 
        &
        \les 
       \wt G_2( u_1,u_2, \wt v_1, \wt v_1, U, \wt V, W )
\label{uHsbddiff}
,
\\
\| \wt U\|_{L^{\infty}_{T}H^{s}_x} 
        &
        \les 
\wt G_2(\wt u_1, \wt u_2, \wt v_1, w_1, \wt U, \wt V, W)
\label{uHsbddiff2}
.
\end{align}

\end{lemma}
\begin{proof}
We follow the same strategy as for Lemma~\ref{LEM:uinfo}.
To estimate $\|U\|_{X^{\frac12-\ta, \ta}_T}$ in \eqref{XsbU}, we proceed as in the proof of \eqref{Xsbu}, 
which reduces to controlling
\begin{align*}
\| \dt U +i\dx^2 U\|_{L^{2}_{T}H^{ -\frac 12}_x} 
&
\les 
\| u_1 \dx \P_{+}( |u_1|^2)-u_2 \dx \P_{+}( |u_2|^2)\|_{L^{2}_{T}H^{ -\frac{1}{2}}_x}  \\
&
\quad 
  +  \| u_1 \mathcal{G}_{h}(|u_1|^2)-u_2 \mathcal{G}_{h}(|u_2|^2) \|_{L^{2}_{T}H^{  -\frac{1}{2}}_x}  + \| |u_1|^2 u_1 - |u_2|^2 u_2\|_{L^{2}_{T}H^{  -\frac{1}{2}}_x},
\end{align*}
and each of these can be estimated as in the proof of \eqref{Xsbu}. 
In the estimate for $\| \wt U\|_{X^{\frac12-\ta, \ta}_T}$, we proceed as above, and use the equation \eqref{ueq2} for $\wt u_1, \wt u_2$, which contributes with the additional~term
\begin{align*}
\|  \tfrac\be\pi M(\wt u_1) \dx \wt{u}_1 - \tfrac\be\pi M(\wt u_2) \dx \wt{u}_2    \|_{L^2_T H^{-\frac12}_x}
&
\le | M(\wt u_1) - M(\wt u_2) | \| \wt u_1 \|_{L^2_T H^{\frac12}_x} + | M(\wt u_2) | \| \wt U \|_{L^2_T H^\frac12_x}
\\
&
\les T^\frac14 R_\frac12^2 \| J^\frac12_x \wt U\|_{L^4_{T,x}}, 
\end{align*}
completing the proof of \eqref{XsbU} for $\wt U$.

\smallskip 
 We move onto \eqref{JsU}. Letting $F_j:= F[u_j]$, $j=1,2$, we see from \eqref{PbHIu} that 
 \begin{align}
\PbHI U  =&  \PbHIp[ (e^{-i\be F_1}-e^{-i\be F_2})v_1] 
+ \PbHIp[ e^{-i\be F_2} V] 
\label{rec1}  \\
& + \PbHI \big[  \Pbhi( e^{-i\be F_1}-e^{-i\be F_2})  \P_{+,\text{lo}} (e^{i\be F_1}u_1)  \big]
 \label{rec2} 
 \\
& + \PbHI \big[  \Pbhi( e^{-i\be F_2})  \P_{+,\text{lo}} ((e^{i\be F_1}-e^{i\be F_2})u_1)  \big] \label{rec3}\\
&+\PbHI \big[  \Pbhi( e^{-i\be F_2})  \P_{+,\text{lo}} (e^{i\be F_2} U)  \big]
\label{rec4} 
\\
&+ \PbHIp[ e^{-i\be F_1}-e^{-i\be F_2}] \P_0( e^{i\be F_1}u_1)+ \PbHIp[e^{-i\be F_2}] \P_0( (e^{i\be F_1}-e^{i\be F_2})u_1) \label{rec5} \\
& + \PbHIp[e^{-i\be F_2}] \P_0( e^{i\be F_2} U) 
\label{rec6} 
\\
&  +\PbHIp \big[  \P_{+}( e^{-i\be F_1}-e^{-i\be F_2})  \P_{-} (e^{i\be F_1}u_1)\big] \label{rec7} \\
& + \PbHIp \big[  \P_{+}(e^{-i\be F_2})  \P_{-} ((e^{i\be F_1}-e^{i\be F_2})u_1)\big] \label{rec8} \\
&+  \PbHIp \big[  \P_{+}( e^{-i\be F_2})  \P_{-} (e^{i\be F_2}U)\big] 
\label{rec9}
\\
& + \PbHI W. 
\label{rec10}
\end{align}
We then estimate each of these terms similarly to how we did for \eqref{Jsu} and with the additional terms in (first one in) \eqref{rec1}, \eqref{rec2},  \eqref{rec3}, \eqref{rec5}, \eqref{rec7}, and \eqref{rec8}, for which we use the difference estimates \eqref{L4est2} and \eqref{umod2b}. 
We omit the similar details and note the only change in handling the second term in \eqref{rec1}.

Proceeding as in the proof of \eqref{Jsu}, namely the estimate for $\I_N$ in \eqref{Jsu2Ia}, we have
\begin{align}
    \| J^s_x \PbHIp[ e^{-i\be F_2} V] \|_{\wt{L^4_{T,x}}}
    & 
    \les 
T^\frac18 \big(1 + R_{\frac12}(u_1,u_2) \big)
\big(1 + R_\s(u_1,u_2) \big)
\| J^s_x V \|_{\wt{L^4_{T,x}}}
. 
\notag
\end{align}
Note that we would like to replace $V$ by $\wt{V}$ above, but  \eqref{LpLqvw} does not apply directly to the differences $V$ and $\wt{V}$, since $u_1,u_2$ may have distinct mass $M_j:= M(u_j)$, $j=1,2$; see \eqref{gauge2}. Instead, we expand $V$ further and consider a high and low-frequency decomposition. First, we write
\begin{align}
V(t,x)
&
= (v_1 - v_2 )(t,x)
\nonumber
\\
&
=  \J_{u_1}[v_1](t, x + \tfrac\be\pi M_1 t) - \J_{u_2}[v_2](t, x + \tfrac\be\pi M_2 t)
\nonumber
\\
&
=  \J_{u_1}[v_1] (t, x + \tfrac\be\pi M_1 t) - \J_{u_1}[v_1](t, x + \tfrac\be\pi M_2 t)
+ (\J_{u_1}[v_1]  - \J_{u_2}[v_2] )  (t, x + \tfrac\be\pi M_2 t)
\nonumber
\\
&
= \wt{v}_1 (t, x + \tfrac\be\pi M_1 t)
-
\wt{v}_1 (t, x + \tfrac\be\pi M_2 t)
+
( \wt{v}_1  - \wt{v}_2) (t, x + \tfrac\be\pi M_2 t)
, 
\label{VVtd}
\end{align}
and then consider a high and low-frequency decomposition of the first contribution to obtain
\begin{align*}
\| &J^s_x \P_\HI V \|_{\wt{L^4_{T,x}}} \\
&
\le 
\| J^s_x \wt V \|_{\wt{L^4_{T,x}}}
+ \| \Pi_{\le \Ld}[ \wt{v}_1 (t, x + \tfrac\be\pi M_1 t)
-
\wt{v}_1 (t, x + \tfrac\be\pi M_2 t)] \|_{\wt{L^4_{T,x}}}
\\
&
\quad 
+ \| \Pi_{> \Ld}[ \wt{v}_1 (t, x + \tfrac\be\pi M_1 t)
-
\wt{v}_1 (t, x + \tfrac\be\pi M_2 t)] \|_{\wt{L^4_{T,x}}}
\\
&
\les \| J^s_x \wt V \|_{\wt{L^4_{T,x}}}
+ 
T^{\frac14} \Ld^{\frac14}
\|  (e^{-it \frac\be\pi (M_1-M_2) \xi } -1)| \ind_{|\xi|\le \Ld} \jb{\xi}^s \Ft_x(\wt{v}_1)(t,\xi)  \|_{L^\infty_T \l^2_\xi }
+ 
\|  \Pi_{>\Ld} \wt{v}_1 \|_{X^{s, \frac12+}_T}
\\
&
\les 
\| J^s_x \wt V \|_{\wt{L^4_{T,x}}}
+ 
T^{\frac14} \Ld^{\frac54} \|M_1-M_2\|_{L^\infty_T}
\| \wt{v}_1 \|_{L^\infty_T H^s_x}
+ 
T^\frac18 \|  \Pi_{>\Ld} \wt{v}_1 \|_{X^{s, \frac12}_T}
\\
&
\les T^\frac18
\big\{  \| \wt V\|_{X^{s, \frac12}_T} +  \Ld^{\frac54} \big|M\big(u_1(0)\big) - M\big(u_2(0)\big)\big| \times \| \wt v_1\|_{L^\infty_T H^s_x} 
+  \| \Pi_{>\Ld} \wt{v}_1 \|_{X^{s, \frac12}_T} \big\}
,
\end{align*}
where $\Ld \in \N$, $\Pi_{\le  \Ld}$ and $\Pi_{>\Ld}$ are the sharp frequency projectors, and we used \eqref{YsCTHs}, mass conservation on the second contribution, and \eqref{YsCTHs} on the third.
The estimate in \eqref{JsU} for $U$ follows. 

Similarly, 
for the $\wt U$ estimate in \eqref{JsU2}, we use the recovery formula in \eqref{PbHIu} for $\wt u_1, \wt u_2$ and obtain an expansion for $\P_{\HI} \wt U$, where most terms can be handled as in the proof of \eqref{JsU}. Here, the second term in \eqref{rec1} is replaced by $\P_{+, \HI} [ e^{-i\be F[\wt u_2]} \wt{V}]$ which can be estimated as $\I_N$ in \eqref{Jsu2Ia}. Instead, we need an analogous decomposition to \eqref{VVtd} for the last term $\P_{\HI} \wt{W}$:
\begin{align}
\wt W
&
= 
\TT_{u_1}[w_1] - \TT_{u_2}[w_2]
= w_1 (t, x - \tfrac\be\pi M_1 t) - w_1(t, x - \tfrac\be\pi M_2 t ) + W(t, x - \tfrac\be\pi M_2 t ), 
\notag
\end{align}
which we can estimate as above to obtain
\begin{align*}
\| J^s_x \P_\HI \wt W \|_{\wt{L^4_{T,x}}}
&
\les
T^\frac18
\big\{
\| W \|_{X^{s, \frac12}_T }
+ 
T^\frac14 \Ld^\frac54 \big|M\big(u_1(0)\big) - M\big(u_2(0)\big)\big|  \| w_1 \|_{L^\infty_T H^s_x } + \| \Pi_{>\Ld} w_1 \|_{X^{s ,\frac12+}_T} 
\big\}. 
\end{align*}

Now we consider \eqref{uHsbddiff}. We split $U=\P_{\leq M} U+ \P_{>M} U$. 
 For the contribution $\P_{\leq M} U$ we use the Duhamel formula for $\P_{\leq M} U$ and gain a factor of $T$ at the cost of a positive power of 
 $M$, as in \eqref{uHsPM}
We omit the similar details.

For the contribution due to $\P_{>M} U$, we adapt the argument used for establishing \eqref{JsU}.
For the terms \eqref{rec2}-\eqref{rec9}, we proceed as in the proof of \eqref{JsU} with \eqref{L4est2}-\eqref{umod2b} instead of \eqref{LinftyE11}-\eqref{LpE} whenever we see a difference of exponentials, 
which always leads to a gain of $M^{-\ta}$.
For \eqref{rec10}, there is no gain of powers of $M$. For the second term in \eqref{rec1}, we proceed as in \eqref{rec1term}-\eqref{rec1high}, together with \eqref{VVtd} to see a $\wt{V}$ contribution.

It only remains to more carefully estimate the first term in \eqref{rec1}. 
We write it as in \eqref{rec1term}:
\begin{align}
 \P_{>M}\PbHIp[ (e^{-i\be F_1}-e^{-i\be F_2})v_1]  
 & =  \P_{>M}\PbHIp[ \P_{\ll M}(e^{-i\be F_1}-e^{-i\be F_2})v_1]  \label{rec11}  \\
 &  \quad +  \P_{>M} \PbHIp[ \P_{\ges M}(e^{-i\be F_1}-e^{-i\be F_2})v_1]  
 \notag
 .
\end{align}
For the second term, as in \eqref{rec1high}, we gain a negative power of $M$ due to the inner projection: 
\begin{align*}
\| &J_x^{s}\P_{>M} \PbHIp[ \P_{\ges M}(e^{-i\be F_1}-e^{-i\be F_2})v_1] \|_{L^{\infty}_{T}L^2_x}  \\
&\les  
M^{-\ta} 
Q\big( R_\s(u_1,u_2) \big)
\big\{ 
\| U\|_{L^\infty_T H^\s_x} + 
\|  e^{-i\be F_1} - e^{-i \be F_2} \|_{L^{\infty}_{T,x}}
\big\} \| \wt{v}_1 \|_{X^{s, \frac12+}_T}
,
\end{align*}
for some polynomial $Q$ of degree at least one, where we used \eqref{L4est2}-\eqref{umod2b} and \eqref{YsCTHs}.
For the first term in \eqref{rec11} we keep both of the projections.
This completes the proof of \eqref{uHsbddiff}.

The proof of \eqref{uHsbddiff2} follows similarly to that of \eqref{uHsbddiff}. For the low frequency part $\P_{\le M } \wt U$, we use the Duhamel formula for $\wt{U}$ corresponding to \eqref{ueq2}. The only distinction is an additional contribution coming from the last term in \eqref{ueq2}, which we control as follows
\begin{align*}
   & T 
    \big\| \P_{\le M} 
    \big( 
    M_1 \dx \wt u_1 - M_2 \dx \wt u_2
    \big)
    \big\|_{L^\infty_T H^s_x} 
    \\
    & \quad 
    \les T M^3 Q \big( R_\frac12(\wt u_1,\wt u_2) \big) 
    \big\{ 
    \big| M(u_1(0)) - M(u_2(0)) \big| 
    +  \| \wt{U} \|_{L^\infty_T H^\frac12_x}
    \big\}
    .
\end{align*}
Similarly, for the large frequency part $\P_{>M} \wt U$, we use the decomposition in \eqref{rec1}-\eqref{rec10}, with $\wt U$, $\wt V$, $\wt W$, $\wt F_j$, $j=1,2$, and $\wt u_1$, proceeding as for \eqref{uHsbddiff}. However, we do not need to proceed via \eqref{VVtd} since we already see $\wt {V}$ term, and instead use the analogous rewriting to replace $\wt W$ by $W$, completing the proof. \qedhere

\end{proof}

\section{Trilinear estimates} \label{SEC:trilin}

We now show the crucial trilinear estimates.

\begin{proposition}\label{PROP:tri}
Let $s\geq \frac 12$, $0<\dl_0<\frac{1}{16}$, and $0<T\leq 1$. Then, for any $v,w\in X_{T}^{s,\frac 12+\dl_0}$, and $u_j\in L^{\infty}_{T}H_{x}^{\frac 12} \cap X_{T}^{- \frac 12,1}$ with $J^{\frac 12}_x u_j\in \wt{L^{4}_{T,x}}$, $j=2,3$, 
it holds that:
\begin{align}
\| \ind_{[0,T]} \Pbhip[ v \, \P_{-}\dx( \cj{u_2} u_3)] \|_{X^{s,-\frac 12+2\dl_0}} & \les  T^{\dl_0} \|v\|_{X_{T}^{s,\frac12+\dl_0}} B(u_2, u_3)
,
\label{vX1}\\
\|\ind_{[0,T]} \P_{-,\textup{hi}}[ w \, \P_{+}\dx(\cj{u_2} u_3)] \|_{X^{s,-\frac 12+2\dl_0}} &  \les T^{\dl_0}  \|w\|_{X_{T}^{s,\frac12+\dl_0}}B(u_2,u_3)
, 
\label{wX1}
\end{align}
where
\begin{align}
B(u_2, u_3) : =   &\prod_{j=2}^3 \big(  \|J_x^{\frac 12} u_j\|_{\wt{L^{4}_{T,x}}}  
+
\|u_j\|_{X^{-\frac12, 1}_T}  +\|u_j\|_{X^{\frac18, \frac38}_T} 
\big) \notag \\
&  + \max_{ \substack{ j_1,j_2 \in \{2,3\} \\ j_1 \neq j_2}} \|u_{j_1}\|_{L^{\infty}_{T}H_x^{\frac 12}}
\Big\{ \|J_{x}^{\frac 12}\PbHI u_{j_2}\|_{\wt{L^{4}_{T,x}}} + \|u_{j_2}\|_{X^{-\frac 12,1}_{T}}\Big\} 
.
\label{Bu}
\end{align}

\end{proposition}

\begin{proof}
We first consider \eqref{vX1}. We take any extension $u_j^{\dag}$ of $u_j$  on $[0,T]$, $j=2,3$. To simplify notation, we simply write $u_j^{\dag}$ as $u_j$.
 By duality and dyadic decomposition, we have 
\begin{align}
&\| \ind_{[0,T]} \Pbhip[ v \, \P_{-}\dx(\cj{ u_2}  u_3)] \|_{X^{s,-\frac 12+2\dl_0}}  \notag \\
 =& \sup_{ \|g\|_{X^{0,\frac 12-2\dl_0}} \leq 1} \bigg|
 \sum_{N,N_1,N_2, N_3, N_{23}} \int_{\R} \int_{\T} \jb{\dx}^{s} \cj{\P_{N}g} \cdot \Pbhip[ \P_{N_1}( \ind_{[0,T]}v) \, \P_{-}\P_{N_{23}}\dx( \cj{\P_{N_2} u_2} \P_{N_3}u_3)]dx dt
 \bigg| 
 .
 \label{Xdyad}
\end{align}
The projections $\P_{+}$ and $\P_{-}$ here imply that:
\begin{align}
N_{1} \ges N_{23} \vee N. \label{N1cond}
\end{align}
In particular, $N_1 \ges 1$.
Thus, we control \eqref{Xdyad} under the additional restriction \eqref{N1cond}. We let $
N_{\max}: = \max(N_1,N_2,N_3,N_{23},N).
$
We also note that 
\begin{align}
N_{23}\les N_2 \vee N_3. \label{N23}
\end{align}
The function $\ind_{[0,T]}v$ will always be placed in $X^{s,b}$ for some $0<b<\frac 12$. As multiplication by the sharp cut-off $\ind_{[0,T]}$ is a bounded operator in these spaces, we have 
\begin{align*}
\|\ind_{[0,T]}v\|_{X^{s,b}} \les \|v\|_{X^{s,b}_{T}}. 
\end{align*}
We will apply this inequality without further mention in the following, and to simplify notation, we denote $\ind_{[0,T]}v$ by $v_{T}$.

\smallskip
\noi
\underline{\textbf{Case 1:} $N_2 \sim N_3$:} 

\noi
By H\"{o}lder's inequality and \eqref{L4}, we have 
\begin{align}
\begin{split}
 &\bigg| \int_{\R} \int_{\T} \jb{\dx}^{s}\cj{\P_{N}g} \cdot \Pbhip[ \P_{N_1}v_{T} \, \P_{-}\P_{N_{23}}\dx( \cj{\P_{N_2}u_2} \P_{N_3}u_3)]dx dt \bigg| \\
 & \les N^{s} N_{23} \big\| \P_{N}  g\big\|_{L^{4}_{t,x}} \| \P_{N_1} v_{T}\|_{L^{4}_{t,x}} \| \P_{N_2}u_2\|_{L^{4}_{t,x}} \|\P_{N_3}u_3\|_{L^{4}_{t,x}}  \\
 & \les N^{s} N_{23}N_{1}^{-s} (N_{2}N_{3})^{-\frac 12}  \|\P_{N} g\|_{X^{0,\frac 38}}\| \P_{N_1}v_T\|_{X^{s,\frac 38}} \|J_x^{\frac 12} \P_{N_2}u_2\|_{L^{4}_{t,x}} \| J_x^{\frac 12} \P_{N_3}u_3\|_{L^{4}_{t,x}}.
 \end{split} \label{L4arg}
\end{align}
To sum over the dyadics, we first have for $(N_2,N_3,N_{23})$
\begin{align*}
\sum_{N_2 \sim N_3}&   \|J_x^{\frac 12} \P_{N_2}u_2\|_{L^{4}_{t,x}} \| J_x^{\frac 12} \P_{N_3}u_3\|_{L^{4}_{t,x}} \sum_{N_{23}\les N_2} \frac{N_{23}}{N_2}  
\les \| J_x^{\frac 12} u_2\|_{\wt{L^{4}_{t,x}}} \| J_x^{\frac 12} u_3\|_{\wt{L^{4}_{t,x}}}.
\end{align*}
Meanwhile, for $(N,N_1)$, with \eqref{N1cond}, we have 
\begin{align*}
\sum_{N} \| \P_{N}g\|_{X^{0,\frac 38}}  \sum_{N_1 \ges N}\bigg( \frac{N}{N_1}\bigg)^s \|\P_{N_1}v_T\|_{X^{s,\frac 38}} &\les \sum_{k =0}^{\infty}2^{-sk}  \sum_{N} \|\P_{N}g\|_{X^{0,\frac 38}} \| \P_{2^{k}N} v_T\|_{X^{s,\frac 38}} \\
& \les \| g\|_{X^{0,\frac 38}} \|v_T\|_{X^{s,\frac 38}}.
\end{align*}
This is acceptable as long as $\dl_0 \leq \frac{1}{16}$.

\smallskip
\noi
\underline{\textbf{Case 2:} $N_2\vee N_3 \gg N_2\wedge N_3$ and $N_2 \wedge N_3 \ges N$:}

\noi
In this case, we can also proceed by using \eqref{Xdyad} and the $L^4$-argument as in \eqref{L4arg}. By \eqref{N1cond}, we must have $N_1\sim N_2 \vee N_3 \sim N_{23} \sim N_{\max}$. Without loss of generality, assume that $N_2 \vee N_3 =N_2$. Then we bound this contribution to \eqref{Xdyad} by
\begin{align*}
\sum_{ \substack{N,N_3 \\ N_3 \ges N}} &N^s N_{3}^{-\frac 12} \|\P_{N_3}J_x^{\frac 12} u_3\|_{L^{4}_{t,x}} \|\P_{N} g\|_{X^{0,\frac 38}} \sum_{N_1\sim N_2} \| \P_{N_1}v_T\|_{X^{0,\frac 38}} \|\P_{N_2}J_x^{\frac 12} u_2\|_{L^{4}_{t,x}} \sum_{N_{23} \ges N} N_{23}N_1^{-s}N_{2}^{-\frac 12} \\
& \les \|v_T\|_{X^{s,\frac 38}} \|J_x^{\frac 12} u_2\|_{\wt{L^{4}_{t,x}}}  \sum_{ \substack{N,N_3 \\ N_3 \ges N}} \bigg(\frac{N}{N_3} \bigg)^{\frac 12}\|\P_{N_3}J_x^{\frac 12} u_3\|_{L^{4}_{t,x}} \|\P_{N} g\|_{X^{0,\frac 38}}  \\
&\les \|g\|_{X^{0,\frac 38}} \| v_T\|_{X^{0,\frac 38}}  \|J_x^{\frac 12} u_2\|_{\wt{L^{4}_{t,x}}}  \|J_x^{\frac 12} u_3\|_{\wt{L^{4}_{t,x}}} .
\end{align*}

\noi
\underline{\textbf{Case 3:} $N_2\vee N_3 \gg N_2\wedge N_3$ and $N_2 \wedge N_3 \ll N$:}

\smallskip
\noi
Here, we 
need to make use of the phase function, since the weight $(N_2 \land N_3)^{-\frac12}$ in Case 2 is not useful.  
We note the following resonance identity

\noi
\begin{align}
\s_1-\s_2 + \s_3 -\s= \xi_{1}^{2}-\xi_{2}^{2}+\xi_{3}^{2} -\xi^{2} = -2(\xi-\xi_1)(\xi-\xi_3) : = \Phi(\cj{\xi}), 
\notag
\end{align}
where $\s=\tau-\xi^2$, $\s_j=\tau_j - \xi_j^2$, $j=1,2,3$, $\tau=\tau_1-\tau_2+\tau_3$, and $\xi=\xi_1-\xi_2+\xi_3$.
Note that under the last condition, we have $\xi-\xi_1=\xi_3-\xi_2$.
Thus, we have 
\begin{align}
\s_{\max}:=\max( |\s_1|,|\s_2|,|\s_3|,|\s|) \ges (N_2 \vee N_3)|\xi-\xi_3|
\ges (N_2 \vee N_3) (N \lor N_3) =: K,
\label{nonres}
\end{align}
and we note that
\begin{align}
 N_{1} \sim N_{\max}. \label{N1max}
\end{align}
In \eqref{Xdyad} we sum over $N_{23}$ as we no longer need to use this dyadic restriction.
For fixed $(N,N_1,N_2,N_3)$ and $K$ as in \eqref{nonres}, we have
\begin{align}
&\int_{\R} \int_{\T} \jb{\dx}^{s}\cj{\P_{N}g} \cdot \Pbhip[ \P_{N_1}v_T \, \P_{-}\dx( \cj{\P_{N_2}u_2} \P_{N_3}u_3)]dx dt \notag  \\
& = \int_{\R} \int_{\T} \jb{\dx}^{s}\cj{ \Q_{\ges K}\P_{N}g} \cdot \Pbhip[ \P_{N_1}v_T \, \P_{-}\dx( \cj{\P_{N_2}u_2} \P_{N_3}u_3)]dx dt   \label{X0}  \\
& \hphantom{X} + \int_{\R} \int_{\T} \jb{\dx}^{s}\cj{ \Q_{\ll K}\P_{N}g} \cdot \Pbhip[ \Q_{\ges K}\P_{N_1}v_T  \cdot  \P_{-}\dx( \cj{\P_{N_2}u_2} \P_{N_3}u_3)]dx dt   \label{X1} \\ 
& \hphantom{X} + \int_{\R} \int_{\T} \jb{\dx}^{s}\cj{ \Q_{\ll K}\P_{N}g} \cdot \Pbhip[ \Q_{\ll K}\P_{N_1}v_T  \cdot  \P_{-}\dx( \cj{ \Q_{\ges K}\P_{N_2}u_2} \P_{N_3}u_3)]dx dt  \label{X2} \\
& \hphantom{X} + \int_{\R} \int_{\T} \jb{\dx}^{s}\cj{ \Q_{\ll K}\P_{N}g} \cdot \Pbhip[ \Q_{\ll K}\P_{N_1}v_T  \cdot  \P_{-}\dx( \cj{ \Q_{\ll K}\P_{N_2}u_2} \Q_{\ges K}\P_{N_3}u_3)]dx dt  \label{X3}\\
& \hphantom{X} + \int_{\R} \int_{\T} \jb{\dx}^{s}\cj{ \Q_{\ll K}\P_{N}g} \cdot \Pbhip[ \Q_{\ll K}\P_{N_1}v_T  \cdot \P_{-}\dx( \cj{ \Q_{\ll K}\P_{N_2}u_2} \Q_{\ll K}\P_{N_3}u_3)]dx dt ,  \label{Xres}
\end{align}
where the projections $\Q_{\ges K}$ and $\Q_{\ll K}$ are as in \eqref{Qproj}.
The terms in \eqref{X0}--\eqref{X3} are the non-resonant contributions while \eqref{Xres} is the nearly-resonant contribution.

\smallskip
\noi
\underline{Bounds for \eqref{X0}, \eqref{X1}, and \eqref{Xres}:}
The bounds for \eqref{X0}-\eqref{X1} follow exactly as in \cite[Bounds for (4.13)-(4.14) on pp.41-42]{CFL}, while \eqref{Xres} vanishes as in \cite[(4.21)]{CFL}. We omit details.

\smallskip
\noi
\underline{Bound for \eqref{X2}:}
By  H\"{o}lder's inequality, \eqref{L4}, and Bernstein's inequality, we have
\begin{align}
\begin{split}
&N^s (N_{2}\vee N_3)   \big\|\Q_{\ll K}\P_{N}g\big\|_{L^{4}_{t,x}}  \| \P_{N_1} \Q_{\ll K} v_{T}\|_{L^{4}_{t,x}} \|\P_{N_2} \Q_{\ges K} u_2\|_{L^{2}_{t,x}} \|\P_{N_3}u_3\|_{L^{\infty}_{t,x}} \\
& \les \frac{N^{s} (N_2 \vee N_3) N_2^{\frac 12}}{ N_1^{s} K}\|\P_{N}g\|_{X^{0,\frac 12 - 2 \dl_0}} \|\P_{N_1}v_{T}\|_{X^{s,\frac 38}} \|\P_{N_2}u_2\|_{X^{-\frac 12,1}} \| \P_{N_3}u_3\|_{L^{\infty}_{t}H^{\frac 12}_x}.
\end{split} \label{X2L4}  
\end{align}
We now control the ensuing multiplier. If $N\ges N_3$, then,
\begin{align}
 N^{s} (N_2 \vee N_3) N_2^{\frac 12} N_1^{-s} K^{-1}  \les 
 \begin{cases}
 N_{\max}^{-\frac 12} \quad& \text{if} \quad s\geq 1, \\
 N^{-1+s} N_{\max}^{\frac 12-s} \quad& \text{if} \quad \frac 12 \leq s <  1, 
 \end{cases}
\label{X2mult}
\end{align}
which always gives a negative power of $N_{\max}$, unless in the case $s=\frac12$ and $N_1 \sim N_2 \gg N \gg N_3$, where we instead use $N^{-\frac 12}$ to sum over $(N,N_3)$ and then the Cauchy-Schwarz argument to sum over $(N_1,N_2)$ under $N_1 \sim N_2$.
The remaining case when $N \ll N_3$ is better. It follows that
 $N_2 \ll N \ll N_3  \sim N_1$, and we argue as in \eqref{X2L4} with
\begin{align*}
\text{LHS of } \eqref{X2mult} \les N_{2}^{\frac 12}N_{3}^{-1} \les N_{\max}^{-\frac 12}, 
\end{align*}
from which the estimate easily follows.

\smallskip
\noi
\underline{Bound for \eqref{X3}:}
Recall that $N_1 \sim N_{\max}$.

\smallskip
\noi
\underline{$\bullet$ \textbf{Case A:} $N_3 \gg N_2$:}
Fix $0<\eps <\frac 13$. Note that by \eqref{nonres}, we have $K\ges  N_3^2 \gg N_2^2$.
 By H\"{o}lder's inequality, \eqref{L4}, Bernstein's inequality, and Lemma~\ref{LEM:Qbd}, we have 
\begin{align*}
&N^{s} N_3  \| \Q_{\ll K}\P_{N} g\|_{L^{4}_{t,x}}  \| \P_{N_1} \Q_{\ll K} v_{T}\|_{L^{4}_{t,x}} \| \P_{N_2} \Q_{\ll K} u_2  \|_{L_{t,x}^{\frac{2(1+\eps)}{\eps}}} \| \P_{N_3} \Q_{\ges K} u_3\|_{L^{2(1+\eps)}_{t,x}} \\
& \les 
N^{s}  N_3  N_{3}^{\frac{1}{2}+\frac{\eps}{2(1+\eps)}}
N_1^{-s} K^{-1+\frac{\eps}{2(1+\eps)}}
\|\P_{N}g\|_{X^{0,\frac 12- 2\dl_0}} \|\P_{N_1}v_{T}\|_{X^{s,\frac 38}}   \| \P_{N_2} u_2  \|_{L_{t,x}^{\frac{2(1+\eps)}{\eps}}} \| \P_{N_3}  u_3\|_{X^{-\frac 12,1}} \\
&\les N_{\max}^{-\frac 12 +\frac{3\eps}{2(1+\eps)}} \|\P_{N}g\|_{X^{0,\frac 12 - 2\dl_0}} \|\P_{N_1}v_{T}\|_{X^{s,\frac 38}}   \| J_x^{\frac 12}\P_{N_2} u_2  \|_{L^{\frac{2(1+\eps)}{\eps}}_{t}L^{2}_x } \| \P_{N_3}  u_3\|_{X^{-\frac 12,1}}, 
\end{align*}
where the power of $N_{\max}$ is always negative and we can easily perform all of the dyadic summations.

\smallskip
\noi
\underline{$\bullet$ \textbf{Case B:} $N_3 \ll N_2$:}
Fix $\eps>0$ sufficiently small to be chosen later.
 By duality, H\"{o}lder's inequality, \eqref{L4}, Bernstein's inequality,  and Lemma~\ref{LEM:Qbd}, we have 
\begin{align*}
&N^{s} N_2 \| \P_{N} \Q_{\ll K} g\|_{L^{\frac{2(1+\eps)}{\eps}}_{t,x}}  \| \P_{N_1} \Q_{\ll K} v_{T}\|_{L^{4}_{t,x}} \| \P_{N_2} \Q_{\ll K} u_2  \|_{L_{t,x}^{4}} \| \P_{N_3} \Q_{\ges K} u_3\|_{L^{2(1+\eps)}_{t,x}} \\
& \les \frac{N^{s+\frac{1}{2}-\frac{\eps}{2(1+\eps)}} N_{2}  N_{3}^{\frac{1}{2}+\frac{\eps}{2(1+\eps)}}  }{N_{1}^s  N_{2}^{\frac 18} K^{1-\frac{\eps}{2(1+\eps)}}} K^{2\dl_0 - \frac{\eps}{2(1+\eps)}} \|\P_{N}g\|_{X^{0,\frac 12- 2\dl_0}} \|\P_{N_1}v_T\|_{X^{s,\frac 38}}   \| \P_{N_2} u_2  \|_{X^{\frac 18,\frac38}} \| \P_{N_3}  u_3\|_{X^{-\frac 12 ,1}}, 
\end{align*}
by taking $0 < \frac{\eps}{2(1+\eps)}\le 2\dl_0$. 
We now consider the dyadic multiplier, for which we have
\begin{align*}
 \frac{N^{s+\frac{1}{2}-\frac{\eps}{2(1+\eps)}} N_{2}  N_{3}^{\frac{1}{2}+\frac{\eps}{2(1+\eps)}}  }{N_{1}^s  N_{2}^{\frac 18} K^{1-\frac{\eps}{2(1+\eps)}}} K^{2\dl_0 - \frac{\eps}{2(1+\eps)}}
 \les 
 \frac{ 
 N^{s-\frac12+2\dl_0} N_3^{\frac12 + \frac{\eps}{2(1+\eps)}} }{N_1^s N_2^{\frac18 - 2\dl_0}}
 \les 
 N_{\max}^{-\frac18 + 6\dl_0}, 
\end{align*}
by considering $N_1\sim N$ or $N_1 \sim N_2$, for small $\eps>0$ as above and $0<\dl_0<\frac{1}{12}$. 
 Again, this factor can then deal with the dyadic summations as for any $0<\eps<\frac 12$, we gain a negative power of $N_{\max}$.

\smallskip

The proof of \eqref{wX1} follows exactly the same arguments as \eqref{vX1} in view of the following observation: 
on the support of $\P_{-,\text{hi}}[ w \P_{+}\dx (\cj{u_2} u_3)]$, we must have 
$
|\xi_1| =|\xi|+|\xi_3-\xi_2|,
$
which leads to \eqref{N1cond} after dyadic decompositions in space.
\end{proof}

We now bound the remaining easier terms on the right-hand sides of \eqref{veq} and \eqref{weq}.

\begin{lemma}\label{LEM:easyterms}
Given $s\geq \frac 12$, $0<\dl_0<\frac 14$, $0<T\leq 1$,  and $\mathcal{Q}\in \{ \textup{Id}, \mathcal{G}_h\}$, we have
\begin{align}
\|  \ind_{[0,T]} \P_{\pm,\textup{hi}}[ e^{i\be F}u\mathcal{Q}(|u|^2)]\|_{X^{s,-\frac 12+ 2 \dl_0}} 
 &
 \les 
 T^{\dl_0} [\|\mathcal{Q}\|_{L^4 \to L^4}+\|\mathcal{Q}\|_{\op} ]  
 \|u\|_{L^\infty_T H^{\frac 12}_x}^3
 (1+  \|u\|_{L^\infty_T H^{\frac 12}_x}^2)
 \notag
 \\
 &
 \quad \times 
  \big(   \|u\|_{L^{\infty}_{T}H^{s}_x} \|J^{\frac12}_x u \|_{L^4_{T,x}
 } + \|J^{s}_x u\|_{L^4_{T,x}} \big)  
 \label{vbd5},  \\
  \|  \ind_{[0,T]} \P_{-,\textup{hi}}[  u \mathcal{Q}(|u|^2)]\|_{X^{s,-\frac 12+ 2\dl_0}} 
  &
  \les 
  T^{\dl_0}
  [\|\mathcal{Q}\|_{L^4 \to L^4}+\|\mathcal{Q}\|_{\op} ]  
   \|u\|_{L^{\infty}_{T}H_x^{\frac 12}}^{2}  \|J^{s}_x u\|_{L^{4}_{T,x}}
  \label{vbd6}
  .
\end{align}
\end{lemma}

\begin{proof}
Since
\begin{align}
\|\ind_{[0,T]} f\|_{X^{s, -\frac{1}{2}+2\dl_0}}  \les
\|f\|_{L^{2}_{T}H^{s}_x} \label{ZXembed}
, 
\end{align}
the estimates in \eqref{vbd5}-\eqref{vbd6} follow easily from first applying \eqref{ZXembed}
and then using the fractional Leibniz rule, Sobolev embedding, \eqref{YsCTHs}, \eqref{Lh}, and \eqref{eFg}. 
From \eqref{eFg}, we reduce to estimating
$
\| u \mathcal{Q} (|u|^2)\|_{L^2_{T}H^{s_1}_x}
$
for $s_1\in \{\frac 12,s\}$. By the fractional Leibniz rule, Cauchy-Schwarz inequality, and Sobolev embedding, we have
\begin{align}
\| u \mathcal{Q} (|u|^2)\|_{L^2_{T}H^{s_1}_x}
 & \les \|J^{s_1}_x u\|_{L^{4}_{T,x}} \| \mathcal{Q} (|u|^2)\|_{L^4_{T,x}} + \|  u\|_{L^{6}_{T,x}} \| J^{s_1}_x \mathcal{Q}(|u|^2)\|_{L^{3}_{T,x}} 
\notag
\\
 & \les T^{\frac 14} \|\mathcal{Q}\|_{L^{4}\to L^4}  \|u\|_{L^{\infty}_{T}H^{\frac 12}_x}^{2} \|J_x^{s_1}u\|_{L^{4}_{T,x}} + \|\mathcal{Q}\|_{L^3 \to L^3}\|  u\|_{L^{6}_{T,x}}  \|J_x^{s_1}u\|_{L^{4}_{T,x}} \|u\|_{L^{12}_{T,x}} 
\notag
\\
&
\les T^{\frac14} ( \| \mathcal{Q } \|_{\op} + \| \mathcal{Q} \|_{L^4\to L^4})   \|u\|_{L^{\infty}_{T}H^{\frac 12}_x}^{2} 
\|J_x^{s_1}u\|_{L^{4}_{T,x}}
 , 
\label{cubic}
\end{align} 
where the control of $\|\mathcal{Q}\|_{L^3\to L^3}$ in terms of the $L^2\to L^2$ and $L^4 \to L^4$ operator norms, comes from interpolation. 
\end{proof}

\begin{remark}\rm
It is clear that the results of Lemma~\ref{LEM:easyterms} generalise to multilinear estimates.
\end{remark}

\section{Well-posedness}\label{SEC:LWP}

By adapting the argument in \cite[Section 2.1]{GL} with \cite[Lemma 5.4]{MP}, we have the following local well-posedness result for \eqref{INLS} in high regularity. 

\begin{proposition}\label{PROP:LWPH2}
Let $s>\frac 32$. Then, given $R>0$, there is $T=T(R)>0$ such that for every $u_0\in H^{s}(\T)$ with $\|u_0\|_{H^{s}}\leq R$, there exists a unique solution $u\in C([-T,T];H^{s}(\T))$ of \eqref{INLS} with $u(0)=u_0$. Moreover, the flow map $u_0\mapsto u(t)$ is continuous from $H^{s}(\T)$ to $C([-T,T];H^{s}(\T))$. 
\end{proposition}

Our goal in this section is to extend the local well-posedness result in Proposition~\ref{PROP:LWPH2} to all regularities $s\geq \frac 12$. Namely, 
an improvement by a full derivative below this previous result. Our estimates in the previous sections are sensitive to this gap and (mildly) change between the regimens $s\geq 1$ and $s<1$; 
see for example \eqref{Jsu} and \eqref{uHsbd}. 
Thus, in order to give a rigorous proof of Theorem~\ref{THM:LWP}, we need a two-stage argument: 
\begin{enumerate}[(a)]
\item Extend the result of Proposition~\ref{PROP:LWPH2} to $H^{s}(\T)$ for any $1\leq s\leq \frac 32$;
\item Using part (a), we then cover the remaining range $\frac 12 \leq s<1$.
\end{enumerate}
As the details for proving part (a) are essentially the same as that for part (b), we focus on the lower regularity result and thus provide details only for part (b). Namely, in the following, we regard the local well-posedness of \eqref{INLS} in $H^s(\T)$ for any $s\geq 1$, as in the form of Proposition~\ref{PROP:LWPH2}, to have been proven already.

\subsection{A priori bounds} \label{SEC:apriori}

Similar to the setting in \cite{MP2}, we do not know that the solutions ensured by Proposition~\ref{PROP:LWPH2} are global-in-time.\footnote{In the defocusing case, the solutions can be extended globally in time. In the focusing case, we expect global solutions for small initial data but  for all large data. Thus, to handle both defocusing/focusing cases simultaneously, we do not assume, in general, that the $H^{1}$ solutions exist globally-in-time.}
 Thus, we first need to obtain a priori estimates for these $H^1$-solutions. Let $u_0\in  H^{1}(\T)$ be given, $u$ be the corresponding solution to \eqref{INLS} from Proposition~\ref{PROP:LWPH2}  and let $v$ and $w$ be defined as in \eqref{gauge}. Then,
 we define $\wt{u}$ and $\wt{v}$ as in \eqref{gauge2}
 and $L_{h}: = \|\GG_{h}\|_{\op}+\|\GG_{h}\|_{L^4\to L^4}$.

Given $0<T\leq 1$, $\frac 12\leq s \leq 1$, $0<\dl_0<\frac{1}{16}$, $b:=\frac 12+\dl_0$, and $b':=1-b-\dl_0=\frac 12-2\dl_0$, we~set
\begin{align}
\begin{split}
N_{T}^{s}(u) = \max\Big(  \|u\|_{L^{\infty}_{T}H^{s}_x}, &\|J^{s}_x u\|_{\wt{L^4_{T,x}}},  \|v(0)\|_{H^{s}}, \|w(0)\|_{H^{s}},  \\
&\|\ind_{[0,T]}\NN_{\wt{v}}(\wt{u})\|_{X^{s, -b'}}, \| \ind_{[0,T]}\NN_{w}(u)\|_{X^{s, -b'}}  \Big),
\end{split} \label{NTu}
\end{align}
where $\NN_{\wt{v}}$ and $\NN_{w}$ are the nonlinear terms defined in \eqref{veq2b} and \eqref{weq}.

By the Duhamel formulation of \eqref{weq} and \eqref{veq2b} and Lemma~\ref{LEM:linXsb}, we have
\begin{align}
\|\wt{v}\|_{X^{s,b}_{T}} & \les \|v(0)\|_{H^{s}} +T^{\dl_0} \|\ind_{[0,T]}\NN_{\wt{v}}(\wt{u})\|_{X^{s, -b'}} \les N_{T}^{s}(u), 
\label{vNs} \\
\|w\|_{X^{s,b}_{T}} & \les \|w(0)\|_{H^{s}} +T^{\dl_0} \|\ind_{[0,T]}\NN_{w}(u)\|_{X^{s, -b'}} \les N_{T}^{s}(u). 
\label{wNs}
\end{align}
From the smoothness of $u, \wt{u},\wt{v},$ and $w$, the map $T \mapsto N_{T}^{s}(u)$ is continuous and non-decreasing (see for instance \cite[Remark 3.7]{CLOP}). It follows from \eqref{gauge} and Lemma~\ref{LEM:L2prod} that
\begin{align}
\| v(0)\|_{H^{s}} + \|w(0)\|_{H^{s}} \les \|u_0\|_{H^{\frac 12}}^2 (1+\|u_0\|_{H^{\frac 12}}^2) \|u_0\|_{H^{s}} \les (1+\|u_0\|_{H^{\frac 12}})^4\|u_0\|_{H^{s}} .
\label{v0w0}
\end{align}
Using the first bounds in \eqref{vNs} and \eqref{wNs}, Proposition~\ref{PROP:tri}, and Lemma~\ref{LEM:easyterms}, we find~that
\begin{align} 
\lim_{T\to 0^{+}} \Big( \| \ind_{[0,T]} \NN_{\wt{v}}(\wt{u})\|_{X^{s,-b'}} +\| \ind_{[0,T]} \NN_{w}(u)\|_{X^{s,-b'}}  \Big)=0. \label{Nslim0}
\end{align}
Therefore, from \eqref{Nslim0}, \eqref{vNs}, \eqref{wNs}, and \eqref{v0w0}, we obtain
\begin{align}
\lim_{T\to 0^{+}} N_{T}^{s}(u)  \les (1+\|u_0\|_{H^{\frac 12}} )^4 \|u_0\|_{H^{s}}.
\notag
\end{align}

From \eqref{Bu} and \eqref{Xsbu}, we see that 
\begin{align}
B(u,u)+B(\wt{u},\wt{u})\les (1+T^\frac14 L_h^2)N_{T}^{\frac12}(u)^{2} \big( 1 + N_T^{1/2}(u)^4 \big). \label{Bus0}
\end{align}
Then, by Proposition~\ref{PROP:tri}, Lemma~\ref{LEM:easyterms}, \eqref{Bus0}, \eqref{vNs}-\eqref{wNs}, and \eqref{Xsbu}, we have
\begin{align}
\|\ind_{[0,T]}\NN_{\wt{v}}(\wt{u})\|_{X^{s, -b'}} +\|\ind_{[0,T]}\NN_{w}(u)\|_{X^{s, -b'}} 
& \les T^{\ta_1}(1+L_h) Q(N^{1/2}_{T}(u))N_{T}^{s}(u), \label{Nfu}
\end{align}
where $Q$ is some non-negative polynomial of at least degree $1$ (in fact, here it is of at least degree $2$) and $\ta_1>0$.
Next, combining \eqref{Jsu}, \eqref{uHsbd}, \eqref{v0w0}, \eqref{vNs}, \eqref{wNs}, and  \eqref{Nfu}, we have
\begin{align}
&\max(\|u\|_{L^{\infty}_{T}H^{s}_x} , \|J^{s}_x \PbHI u\|_{\wt{L^4_{T,x}}}) \notag \\
&\les \|u_0\|_{H^{s}}+ \|\wt{v}\|_{X^{s,b}_{T}}+\|w\|_{X^{s,b}_{T}} \notag  \\
& \hphantom{C}+\big(TM^3(1+L_h) +T^{\frac 14}+M^{-\ta_2}(1+\|\wt{v}\|_{X^{\frac 12,b}_{T}})\big)Q(N^{1/2}_{T}(u)) N_{T}^{s}(u) \notag \\
& \les (1+\|u_0\|_{H^{\frac 12}})^{4} \|u_0\|_{H^{s}}  + \big(T(1+L_h )M^3+M^{-\ta_2} \big)Q(N^{1/2}_{T}(u)) N_{T}^{s}(u)
\label{unorms}
\end{align}
for some polynomial $Q \geq 1$ of degree at least $1$ (which may change from line to line), any $0<T\leq 1$, some $\ta_2>0$, and for any $M\in \N$ sufficiently large.
Thus, from \eqref{v0w0}, \eqref{vNs}-\eqref{wNs}, \eqref{Nfu}, and \eqref{unorms}, we have
\begin{align}
\begin{split}
N_{T}^{s}(u) & \leq C_1 (1+\|u_0\|_{H^{\frac12}})^5 \|u_0\|_{H^{s}} \\
& \ 
+C_2 \big(T^{\ta_1}(1+L_h)M^3 +M^{-\ta_2})Q(N^{1/2}_{T}(u)\big) N_{T}^{s}(u),
\end{split} \label{NsTbd}
\end{align}
for some absolute constants $C_1,C_2>0$, $\ta_1, \ta_2>0$.\footnote{Note that the constants $C_1,C_2$ and the polynomial $Q$ in \eqref{NsTbd} and \eqref{Tast} depend on $s$ here. In the following, we only need to apply \eqref{NsTbd} for three regularities: $\frac 12, 2$ and some fixed $\frac12  <s\leq \frac 32$ so we may choose the largest of the constants and polynomials.}   
We now choose $s=\frac12$ and $M\gg 1$ depending on $\|u_0\|_{H^{\frac 12}}$ so that 
\begin{align}
C_{2}M^{-\ta_2}Q(4R) <\tfrac{1}{4} \quad \text{where} \quad R: = C_1 (1+\|u_0\|_{H^{\frac 12}})^5 \|u_0\|_{H^{\frac 12}}. \label{M}
\end{align}
Then, given this choice of $M$, we choose $T_{\ast}=T_{\ast}(M, L_h)>0$
so that 
\begin{align}
C_{2}T_{\ast}^{\ta_1}M^3 (1+L_h)Q (4R) <\tfrac{1}{4}. \label{Tast}
\end{align}
By a continuity argument with the choices \eqref{M} and \eqref{Tast}, we obtain
\begin{align}
N_{T}^{1/2}(u) \leq 2R.  \label{Ns0Tbd}
\end{align}
for any $0<T\leq T_{\ast}$.
Next, inserting \eqref{Ns0Tbd} into \eqref{NsTbd} with $s=1$, and reducing $T_*=T_*(M,L_h, \|u_0\|_{H^{\frac12}})$ if necessary, we find
\begin{align}
\|u\|_{L^{\infty}_{T} H^{ 1}_x} \leq N^{1}_{T}(u) \leq 2C_1(1+\|u_0\|_{H^{\frac 12}})^{5} \|u_0\|_{H^1}. 
\notag
\end{align}
This a priori bound implies that the maximal time of existence $T_{\max}$ for a $C_{T}H^1(\T)$ solution has the uniform lower bound $T_{\max}\geq  T_{\ast}$ which only depends on $\|u_0\|_{H^{\frac 12}}$ and not on $\|u_0\|_{H^{1}(\T)}$.  Finally, from \eqref{NsTbd} with $s\in (\frac 12,1)$ and using \eqref{Ns0Tbd},  for $0\le T \le T_*$, we also have that
\begin{align}
N_{T}^{s}(u) \leq 2C_1(1+\|u(0)\|_{H^{\frac 12}})^{5} \|u(0)\|_{H^{s}}. \label{Hsscale}
\end{align}

\begin{remark}\rm \label{REM:unifLh}
When $\sup_{1\leq h\leq \infty} L_{h}<\infty,$
we may choose $T_{\ast}>0$ uniformly in $h\geq 1$.
\end{remark}

\subsection{Tightness estimates}

The goal of this subsection is to establish bounds for the terms $\| \Pi_{>\Ld} \wt{v}\|_{X^{s,b}_{T}}$ and $\|\Pi_{>\Ld}w\|_{X^{s,b}_{T}}$, which appear in the difference estimates in Lemma~\ref{LEM:diffests}~(ii).

\begin{proposition} \label{PROP:tight}

Assume that the results of Section~\ref{SEC:apriori} hold true, with $R$ as in  \eqref{M} and $T_*>0$ so that \eqref{Ns0Tbd} and~\eqref{Hsscale} hold for $0\le T \le T_*$. Then, there exist $\ta_3>0$, constants $C_3>0$ and $C_4(R,L_{h})>0$, and $0<T_{\ast \ast} = T_{**}(R, L_h)\leq T_{\ast}$, such that 
\begin{align}
\| \Pi_{>\Ld} \wt{v}\|_{X^{s,b}_{T}} & \leq C_3 \|\Pi_{>\Ld}v(0)\|_{H^{s}} + C_4(R, L_{h}) \|u(0)\|_{H^s}\Ld^{-\ta_3}, \label{tightv} \\
\| \Pi_{>\Ld} w\|_{X^{s,b}_{T}} & \leq C_3 \|\Pi_{>\Ld}w(0)\|_{H^{s}} + C_4(R,L_{h}) \|u(0)\|_{H^s}\Ld^{-\ta_3}, 
\label{tightw}
\end{align}
for all $\Ld\in \N$ sufficiently large and $0<T\leq T_{\ast \ast}$.
\end{proposition}

The proof of Proposition~\ref{PROP:tight} requires some preliminary results.

\begin{lemma}\label{LEM:Xsby}
Given $0<T\leq 1$, let $u$ be a smooth solution to \eqref{INLS2} on $[0,T]$ and $\wt{u}$ be as in~\eqref{gauge2}. 
Define $\wt{y}: =e^{i\be \wt{F}}\wt{u} : = e^{i\be {F}[\wt u]}\wt{u}$. Then, for any $\frac 12 \leq s<2$, it holds that
\begin{align}
\begin{split}
\sup_{0\leq \ta \leq 1} \| \wt{y}\|_{X^{s-\ta,\ta}_{T}}&  \les  \|\wt{u}\|_{L^{\infty}_{T}H^{\frac 12}_x}(1+ \|\wt{u}\|_{L^{\infty}_{T}H^{\frac 12}_x}^5) \\
& \quad \times \Big[  (1 + \|J^{\frac 12}_x\wt{u}\|_{\wt{L^4_{T,x}}} ) \|\wt{u}\|_{L^{\infty}_{T}H^{s}_x}  +  (1+\|\GG_h\|_{\textup{op}}) \|J^{s}_x\wt{u}\|_{\wt{L^4_{T,x}}} \Big].
\end{split} 
\label{yXsb}
\end{align}
\end{lemma}
\begin{proof}
We proceed as in the proof of Lemma~\ref{LEM:uinfo}. 
Let $y^{\dag}:= S(-t)[ e^{i\be \wt{F}}\wt{u}]$ on $[0,T]$ and extend $y^\dag$ to $[-1,2]$ by setting $\dt y^\dag= 0$ on $[-1,2]\setminus [0,T]$. Then, $\eta(t) S(t) y^{\dag}(t)$ is an extension of $\wt{y}$ on $[0,T]$. Then, as in \eqref{uXsb1a}, it follows that 
\begin{align*}
\sup_{0\leq \ta \leq 1} \| \wt{y}\|_{X^{s-\ta,\ta}_{T}} & \les \| (\dt +i\dx^2 )(\wt{y})\|_{L^{2}_{T}H^{s-1}_x} + \| \wt{y}\|_{L^{\infty}_{T}H^s_x} 
.
\end{align*}
By \eqref{eFg}, we have 
\begin{align*}
\| \wt{y}\|_{L^{\infty}_{T}H^s_x} \les\|\wt{u}\|_{L^{\infty}_{T}H^{\frac 12}_x}^{2} ( 1+ \|\wt{u}\|_{L^{\infty}_{T}H^{\frac 12}_x}^{2} ) \|\wt{u}\|_{L^{\infty}_{T}H^{s}_x}.
\end{align*}
Next, following the computations in Lemma~\ref{LEM:Xsbu}, we find
\begin{align}\label{yeqn}
\begin{split}
(\dt +i\dx^2 )(\wt{y})  =& -2\be e^{i\be \wt{F}} \wt{u} \P_{-}\dx(|\wt{u}|^2) \\
&
+ ie^{i\be \wt{F}}
\big\{  
\tfrac\be\pi\big[ 
 \tfrac{\be}{4\pi} M(\wt u)^2 - P(\wt u)
\big] 
\wt{u} 
 -\be \wt{u}\GG_{h}(|\wt{u}|^2) +\g |\wt{u}|^2 \wt{u}  \big\}. 
 \end{split}
\end{align}

For the second contribution on the right-hand side of \eqref{yeqn}, we argue as in \eqref{uXsb1aa}-\eqref{uXsb1ab} and control it by 
\begin{align*}
T^{\frac 14} (1+\| \GG_{h}\|_{\text{op}}) \|\wt{u}\|_{L^{\infty}_{T}H^{\frac 12}_x}^{2}(1+\|\wt{u}\|_{L^{\infty}_{T}H^{\frac 12}_x}^{4}) \| J^{s}\wt{u}\|_{L^{4}_{T,x}}.
\end{align*}

\noi
For the contribution from the first term on the right-hand side of \eqref{yeqn}, we write it as
\begin{align*}
e^{i\be \wt{F}} \wt{u} \P_{-}\dx(|\wt{u}|^2)  = \dx\big( e^{i\be \wt{F}} \wt{u} \P_{-}(|\wt{u}|^2)\big) -  \dx( e^{i\be \wt{F}} \wt{u} )\P_{-}(|\wt{u}|^2) .
\end{align*}
Then, by \eqref{eFg} and the fractional Leibniz rule, we have 
\begin{align*}
&\| \dx\big( e^{i\be \wt{F}} \wt{u} \P_{-}(|\wt{u}|^2)\big) \|_{L^{2}_{T}H^{s-1}_x}  \\
&\les \| e^{i\be \wt{F}}  \big[\wt{u} \P_{-}(|\wt{u}|^2) \big] \|_{L^{2}_{T}H^{s}_x}  \\
& \les \|\wt{u}\|_{L^{\infty}_{T}H^{\frac 12}_x} 
\Big( \|\wt u\|_{L^\infty_T H^{\frac12}_x} \| \wt{u} \P_{-}(|\wt{u}|^2)\|_{L^2_T H^{s}_x} + (1+ \|\wt{u}\|_{L^{\infty}_{T}H^{\frac 12}_x}^2 ) \| \wt{u}\|_{L^{\infty}_{T}H^{s}_x} \|\wt{u} \P_{-}(|\wt{u}|^2)\|_{L^{2}_{T}H^{\frac 12}_x} \Big)  \\
& \les T^\frac14 \|\wt{u}\|_{L^{\infty}_{T}H^{\frac 12}_x}^4
(1+ \|\wt{u}\|_{L^{\infty}_{T}H^{\frac 12}_x}^2 )  
\big(  \| J^{s}\wt{u}\|_{\wt{L^{4}_{T,x}}}+ \|\wt{u}\|_{L^{\infty}_{T}H^{s}_x}  \big).
\end{align*}

\noi
Finally, by splitting $\wt{u}= (\P_{\ll N}+\wt{\P}_{\ges N} )\wt{u}$, and using \eqref{LinftyE11}-\eqref{LpE}, we find
\begin{align*}
\| J_x^{s} ( e^{i\be \wt{F}} \wt{u}) \|_{\wt{L^{4}_{T,x}}} 
& \les (1+ \|\wt{u}\|_{L^{\infty}_{T}H^{\frac 12}_x}^2 ) \Big( \|J^{s}\wt{u}\|_{\wt{L^{4}_{T,x}}} +   \|\wt{u}\|_{L^{\infty}_{T}H^{\frac 12}_x}^2 \|\wt{u}\|_{L^{\infty}_{T}H^{s}_x}\Big). 
\end{align*}
It then follows from \eqref{prodest1} that 
\begin{align*}
\|\dx( e^{i\be \wt{F}} \wt{u} )\P_{-}(|\wt{u}|^2)\|_{L^{2}_{T}H^{s-1}_{x}} 
& \les 
T^\frac14 \|\wt{u}\|_{L^{\infty}_{T}H^{\frac 12}_x} (1+ \|\wt{u}\|_{L^{\infty}_{T}H^{\frac 12}_x}^3 )  \\
 &\quad \times \Big(  \|J^{s}_x \wt{u}\|_{\wt{L^{4}_{T,x}}}+ \|\wt{u}\|_{L^{\infty}_{T}H^{\frac 12}_x}^2 \|\wt{u}\|_{L^{\infty}_{T}H^{s}_x} + \|J^{\frac 12}_x \wt{u}\|_{\wt{L^{4}_{T,x}}} \|\wt{u}\|_{L^{\infty}_{T}H^{s}_x} \Big).
\end{align*}
Combining these results then yields \eqref{yXsb}.
\end{proof}

We now prove a tightness (on the frequency side) result for the trilinear form $y \mapsto \P_{\pm} [ y \mathcal{Q}_{h}(|y|^2)]$ in Fourier restriction norm spaces. In particular, we exploit a nonlinear smoothing effect for some parts of this contribution.

\begin{lemma} \label{LEM:tightv}
Let $0<T\leq 1$, $s\geq \frac 12$, $b'=\frac12-2\dl$, $b= \frac12+\dl$, for small $0<\dl<1$, and $y\in  X^{s-\ta,\ta}_{T}$ for every $0\leq \ta\leq 1$ and such that $\P_{ +,\textup{hi}}y\in X^{s,b}_{T}$ or $\P_{ -,\textup{hi}}y\in X^{s,b}_{T}$. Then, for $\mathcal{Q}_{h}\in \{ \textup{Id}, \GG_{h}\}$, then for all $\Ld \in \mathbb{N}$ sufficiently large, it holds that
\begin{align}
\begin{split}
\| \P_{\pm} \Pi_{>\Ld} \big[  y \mathcal{Q}_{h}(|y|^2)    \big]\|_{X^{s,-b'}_{T}} 
&  \les \Ld^{-\frac{1}{10}} \|\mathcal{Q}_{h}\|_{\textup{op}} \|y\|_{\mathcal{X}^{\frac 12}_{T}}^{2}  \Big[ \|y\|_{\mathcal{X}^{s}_{T}} + \|\P_{\pm ,\textup{hi}} y\|_{X^{s,b}_{T}}  \Big] \\
& \qquad +  T^{\frac 12}\|\mathcal{Q}_{h}\|_{\textup{op}}\|y\|_{L^{\infty}_{T}L^2_x}^{2} \| \Pi_{>\Ld} \P_{\pm}y\|_{X^{s,b}_{T}},
\end{split} \label{ysmoothing}
\end{align}
where $\| y\|_{\mathcal{X}^{s}_{T}} : = \max_{0\leq \ta \leq 1} \|y\|_{X^{s-\ta,\ta}_{T}}.$
\end{lemma} 
\begin{proof}

We begin by writing 
\begin{align}
\mathcal{F}_{x}\{ y\mathcal{Q}_{h}(|y|^2)\}(\xi) & = \sum_{\xi=\xi_1 -\xi_2+\xi_3}  \ft{\mathcal{Q}_{h}}(\xi_3-\xi_2) \ft y (\xi_1) \cj{ \ft y(\xi_2)} \ft y(\xi_3) \notag \\
&  = \ft{\mathcal{Q}_{h}}(0) M(y) \ft y(\xi) +\mathcal{F}_x\{ \NN_{h}(y) +  \RR_{1}(y)+\RR_{2}(y)\}(\xi)\label{cubedecomp}
 \end{align} 
 where
 \begin{align}
 \mathcal{F}\{ \NN_{h}(y)\}(\xi)   &:= \sum_{\substack{\xi=\xi_1 -\xi_2+\xi_3 \\ \xi_1 ,\xi_3\neq \xi }}  \ft{\mathcal{Q}_{h}}(\xi_3-\xi_2) \ft y (\xi_1) \cj{ \ft y(\xi_2)} \ft y(\xi_3), \label{Nh} \\
\mathcal{F}\{ \RR_{1}(y)\}(\xi)  : =  \ft y(\xi) \sum_{\xi'} &\ft{\mathcal{Q}_{h}}(\xi-\xi') |\ft y(\xi')|^2, 
\quad \mathcal{F}\{ \RR_{2}(y)\}(\xi)   : = \ft{\mathcal{Q}_{h}}(0) |\ft y(\xi)|^2 \ft y(\xi), 
\notag
\end{align}
and proceed to estimate the contributions from each piece separately.
 Notice that $\|\ft{\mathcal{Q}_{h}}(\xi)\|_{\l^{\infty}_{\xi}} = \| \mathcal{Q}_{h}\|_{\text{op}}$.
From \eqref{YsCTHs}, we have 
\begin{align*}
\| \P_{\pm}\Pi_{>\Ld} \mathcal{R}_{2}(y)\|_{X^{s,b'}_{T}} \leq \| \P_{\pm}\Pi_{>\Ld} \mathcal{R}_{2}(y)\|_{L^2_{T}H^{s}_x}  \les T^{\frac 12} \|\mathcal{Q}_{h}\|_{\textup{op}}\|y\|_{L^{\infty}_{T}L^2_x}^{2} \| \Pi_{>\Ld} \P_{\pm }y\|_{X^{s,b}_{T}}.
\end{align*}
A similar bound holds for the contribution from $\mathcal{R}_1$, by placing $\ft{\mathcal{Q}_{h}}(\xi)$ into~$\l^{\infty}_{\xi}$. 
Similarly, the first term in \eqref{cubedecomp} is controlled by the second term on the right-hand side of~\eqref{ysmoothing}.

Thus, it remains to control the contribution in \eqref{Nh}, which is the main difficulty.
Let $y^{\dag}$ be an arbitrary extension of $y$ on $[0,T]$. We define the following set
\begin{align*}
\G_{\Ld, \pm} : = \{ (\xi_1,\xi_2,\xi_3,\xi_4)\in \Z^{4} \,: & \, \xi_1-\xi_2+\xi_3-\xi_4=0, 
\\
&  \xi_1, \xi_3 \neq \xi_4,  \,\, \sgn(\xi_4)=\pm 1, \,\, |\xi_4|>\Ld\}, 
\end{align*}
and $\s_{\l}:= |\tau_{\l}-\xi_{\l}^{2}|$, $\l=1, \ldots, 4$,  to be the modulations. Now, by duality, it suffices to control
\begin{align*}
\intt_{\tau_1-\tau_2+\tau_3-\tau_4=0} \sum_{\G_{\Ld, \pm}} \jb{\xi_4}^{s} \ft{ \mathcal{Q}_{h}}(\xi_3-\xi_2)
 \frac{ \ft g(\tau_4,\xi_4)}{ \jb{\s_4}^{\frac 12-2\dl}}  \ft{ y^{\dag}}(\tau_1,\xi_1)  \cj{\ft{ y^{\dag}}(\tau_2,\xi_2)}\ft{ y^{\dag}}(\tau_3,\xi_3) d\tau_1 d\tau_2 d\tau_3
\end{align*}
for $g\in L^2_{t,x}$. We dyadically decompose the spatial frequencies of all of the functions and use the boundedness of $\mathcal{Q}_{h}(\xi)$ to further reduce to controlling
\begin{align*}
\sum_{ \substack{ N_1, N_2,N_3,N_4 \\ N_4 \ges \Ld}} &\intt_{\tau_1-\tau_2+\tau_3-\tau_4=0} \sum_{\G_{\Ld, \pm}} \jb{\xi_4}^{s} 
 \frac{ |\ft{ \P_{N_4}g}(\tau_4,\xi_4)|}{ \jb{\s_4}^{\frac 12-2\dl}}  \prod_{j=1}^{3} |\ft{ \P_{N_j} y^{\dag}}(\tau_j,\xi_j) |d\tau_1 d\tau_2 d\tau_3  =: \sum_{ \substack{ N_1, N_2,N_3,N_4 \\ N_4 \ges \Ld}}  I_{\cj{N}}, 
 \end{align*}
 where $\cj{N} = (N_1,N_2,N_3,N_4)$. 
By symmetry, we may assume that $N_1 \geq N_3$, 
we let $N_{\max} \ge N_{\med} \ge N_{\min}$ be the decreasing rearrangement of $N_1, N_2,N_3$,
and
$\s_{\MAX} = \max_{j=1, \ldots, 4} |\s_j|$. Then, the phase function satisfies
\begin{align}
|\Phi(\cj{\xi})| =| \xi_1^2 -\xi_2^2 +\xi_3^2 -\xi^2| \sim |\xi_1 - \xi_4| |\xi_3 -\xi_4| \les \s_{\MAX}. 
\label{mods21}
\end{align}
We split the summation into a number of cases.

\smallskip 
\noi 
$\bullet$ \underline{\textbf{Case 1:} $|\Phi(\cj\xi)| \ges N_{\max} N_4$}

\noi
\underline{\textbf{Subcase 1.1:} $N_{\max} \gg N_{\med}$}
\\
Here, we have $N_{\max}\sim N_4$. 
If $N_{\max} = N_1$, then $\sgn(\xi_1) = \sgn(\xi_4)$ and we can place $\P_{\pm, \hi}$ for free into the factor with $\P_{N_1}$. Then, 
from Parseval's theorem, H\"older's inequality, \eqref{L4}, and Bernstein's inequality, we have
\begin{align}
    I_{\cj{N}}
    &
    \les N_4^s \bigg\| \Ft_{t,x}^{-1} \bigg( \frac{|\ft{\P_{N_4} g}|}{ \jb{\s_4}^{\frac12-2\dl}} \bigg) \bigg\|_{L^4_{t,x}} \| \P_{N_1} \P_{\pm, \hi} y^\dagger \|_{L^4_{t,x}} \| \P_{N_2} y^\dagger \|_{L^4_{t,x}} \|\P_{N_3} y^\dagger \|_{L^4_{t,x}} 
    \nonumber
    \\
    & 
    \les N_4^s \| g\|_{X^{0,-\frac18+ 2\dl}} \| \P_{N_1} \P_{\pm, \hi} y^\dagger\|_{X^{0,\frac38}} \prod_{j=2}^3 \| \P_{N_j} y^\dagger \|_{X^{0,\frac38}}
    \nonumber
    \\
    &
    \les N_4^s (N_{\max} N_4)^{-\frac18 +2\dl}
    N_{1}^{-s}
    N_{\min}^{-\frac18}
    \|g\|_{L^2_{t,x}}
    \|\P_{\pm,\text{hi}} y^\dagger \|_{X^{s, \frac12+\dl}}
    \| y^\dagger \|_{X^{\frac12-\frac38,\frac38}} \| y^\dagger \|_{X^{0,\frac12}}
    \label{N1N2}
    \\
    &
    \les 
    N_{\max}^{-\frac14+4\dl}
    \|g\|_{L^2_{t,x}}
    \|\P_{\pm,\text{hi}} y^\dagger \|_{X^{s, \frac12+\dl}}
    \| y^\dagger \|^2_{\mathcal{X}^{\frac12}}
    \nonumber
    \end{align}
    where the second factor in the third inequality comes from gaining a power of $|\Phi(\cj\xi)|^{-\frac18+2\dl}$ by considering the different choices of $\s_{\MAX}$. Using the negative power of $N_{\max}$, we can sum in the dyadics and gain a power of $\Ld^{-\frac{1}{10}}$ for $\dl>0$ sufficiently small.
    
Next, consider $N_{\max} = N_2$. If $\s_{\MAX} = |\s_4|$, we place factors with frequencies comparable to $(N_4, N_{\med}, N_{\max}, N_{\min})$
in $(L^2, L^4, L^4, L^\infty)$, together with \eqref{L4}, Bernstein's inequality, and \eqref{YsCTHs} to obtain
\begin{align}
    I_{\cj{N}}
    &
    \les N_4^s (N_{\max} N_4)^{-\frac12 + 2\dl} N_{\med}^{-\frac18} N_{\max}^{-s+\frac38} N_{\min}^{\frac12+} 
    \| g\|_{L^2_{t,x}} \|y^\dagger\|_{X^{\frac12-\frac38, \frac38}} \|y^\dagger\|_{X^{s-\frac38, \frac38}} \|y^\dagger\|_{X^{0-, \frac12+}}
    \label{N1N2s1a}
    \\
    &
    \les 
    N_{\max}^{-\frac14+4\dl+ } 
    \| g\|_{L^2_{t,x}} \|y^\dagger\|_{\mathcal{X}^s} \|y^\dagger\|^2_{\mathcal{X}^\frac12}
    \label{N1N2s1}
    .
\end{align}
Now let $\s_{\max}, \s_{\med}, \s_{\min}$ denote the modulations of the factors associated with frequencies around $N_{\max}, N_{\med}, N_{\min}$, respectively.
If $\s_{\MAX} = |\s_{\med}|$, we instead place the factors 
in $(L^4, L^2, L^4, L^\infty)$:
\begin{align}
    I_{\cj{N}}
    &
    \les 
    N_4^s N_{\med}^{\frac12+} (N_{\max} N_4)^{-1} N_{\max}^{-s+\frac38} N_{\min}^{\frac12+} 
    \| g\|_{L^2_{t,x}} \|y^\dagger\|_{X^{\frac12-1, 1}} \|y^\dagger\|_{X^{s-\frac38, \frac38}} \|y^\dagger\|_{X^{0-, \frac12+}}
    \label{N1N2s1b}
\end{align}
which is controlled by \eqref{N1N2s1} for small $\dl$. 
If $\s_{\MAX} = |\s_{\min}|$, we exchange the roles of the second and fourth factors in \eqref{N1N2s1b}. 
Lastly, if $\s_{\MAX} = |\s_{\max}|$, we place the factors in $(L^4, L^4, L^2, L^\infty)$:
\begin{align}
    I_{\cj{N}}
    &
    \les 
    N_4^s N_{\med}^{-\frac18}  N_{\max}^{-s+1} (N_{\max} N_4)^{-1} N_{\min}^{\frac12+} 
    \| g\|_{L^2_{t,x}}  \|y^\dagger\|_{X^{\frac12-\frac38, \frac38}}  \|y^\dagger\|_{X^{s-1, 1}}\|y^\dagger\|_{X^{0-, \frac12+}}, 
    \label{N1N2s1c}
\end{align}
which is also controlled by \eqref{N1N2s1}. 

\smallskip

\noi
\underline{\textbf{Subcase 1.2:} $N_{\max} \sim N_{\med} \ges N_4$}
\\
In this region, we proceed as in the part of Subcase 1.1 where $N_{\max}=N_2$, 
applying the estimates in \eqref{N1N2s1a}, \eqref{N1N2s1b}, and \eqref{N1N2s1c} to obtain
\begin{align*}
    I_{\cj{N}}
    &
    \les 
    \big( 
    N_4^{s-\frac12+2\dl} N_{\max}^{-s - \frac14 + 2\dl } 
    +
    N_4^{s-1} N_{\max}^{-s - \frac18 +} 
    \big) 
    N_{\min}^{\frac12+} 
    \|g\|_{L^2_{t,x}}
    \|y^\dagger\|_{\mathcal{X}^s} \|y^\dagger \|_{\mathcal{X}^\frac12}^2
    \\
    &
    \les 
    N_{\max}^{-\frac18 + 4\dl +  }
    \|g\|_{L^2_{t,x}}
    \|y^\dagger\|_{\mathcal{X}^s} \|y^\dagger \|_{\mathcal{X}^\frac12}^2, 
\end{align*}
and the estimate follows for $\dl>0$ small.

\smallskip 
\noi 
$\bullet$ \underline{\textbf{Case 2:} $|\Phi(\cj\xi)| \ll N_{\max} N_4$}
\\
\noi
\underline{\textbf{Subcase 2.1:} $N_{\max}  \gg N_{\med}$}
\\
Here, $N_4 \sim N_{\max}$. 
If $N_{\max} = N_1$, then $N_1 \sim N_4 \gg N_2, N_3$, and thus $|\xi_1-\xi_2| \sim N_1$. From the assumption on the phase and the restriction to $\G_{\Ld, \pm}$, this implies that $1 \le |\xi_3-\xi_2| = |\xi_1-\xi_4| \ll N_4 \sim N_1 $ and $\sgn(\xi_1) = \sgn(\xi_4)$, allowing us to add $\P_{\pm, \hi}$ to $\P_{N_1}$. 
Thus, we have $|\Phi(\cj\xi)| \ges |\xi_1-\xi_2| \sim N_1$,
and we proceed as in \eqref{N1N2}, 
but gaining only $N_{\max}^{-\frac18+2\dl}$ from the~phase:
\begin{align*}
   I_{\cj{N}}
    &
    \les 
    N_4^s (N_{\max})^{-\frac18+2\dl} N_{\max}^{-s} N_{\min}^{-\frac18} \| g\|_{L^2_{t,x}} \|\P_{\pm, \hi} y^\dagger\|_{X^{s, \frac12+\dl}} \|y^\dagger \|^2_{\mathcal{X}^{\frac12}}
\end{align*}
which gives a negative power of the maximum frequency, as intended.

The case $N_{\max} = N_2 \sim N_4 \gg N_1 \ge N_3$ cannot happen, since $|\Phi(\cj\xi)| \sim |\xi_1 - \xi_2| |\xi_3-\xi_2| \sim N_2^2 \sim N_{\max} N_4$ which contradicts the phase assumption.

\smallskip 
\noi 
\underline{\textbf{Subcase 2.2:} $N_{\max} \sim N_{\med} \ges N_{4}$}
\\
Since $N_1\ge N_3$, we have $N_1 \sim N_{\max} \sim N_{\med} \gg N_4$. 
From the factorization in \eqref{mods21}, if $|\xi_1 - \xi_4| \ges N_1$, then $|\xi_3-\xi_4| \ll N_4 \sim N_3$ and hence $\sgn(\xi_3) = \sgn(\xi_4)$ allowing us to add $\P_{\pm, \hi}$ to the $\P_{N_3}$ contribution. 
Similarly, if $|\xi_3 - \xi_4| \ges N_1$, then $|\xi_1 - \xi_4| \ll N_4 \sim N_1$ and $\sgn(\xi_1) = \sgn(\xi_4)$, allowing us to add $\P_{\pm, \hi}$ to the $\P_{N_1}$ term instead. In both cases, proceeding as in \eqref{N1N2}, gaining only $N_{\max}$ from the phase and $N_{j}^{-s} \sim N_4^{-s}$ from the term $\P_{\pm, \hi} \P_{N_j} y^\dagger$, $j=1,3$, we obtain
\begin{align*}
    I_{\cj{N}}
    &
    \les 
    N_4^s (N_{\max})^{-\frac18+2\dl} N_j^{-s} 
    \|g\|_{L^2_{t,x}}
    \|\P_{\pm,\text{hi}} y^\dagger \|_{X^{s, \frac12+\dl}}
    \| y^\dagger \|^2_{\mathcal{X}^{\frac12}},
\end{align*}
and the estimate follows. 

It remains to consider $|\xi_1-\xi_4|, |\xi_3-\xi_4| \ll N_1 $, which implies $N_1 \sim N_2 \sim N_3 \sim N_4$ and $\sgn(\xi_1) = \sgn(\xi_4)$. Then, we proceed as in \eqref{N1N2}, with no gain from the phase function:
\begin{align*}
    I_{\cj{N}}
    \les N_4^s N_{1}^{-s} N_{\min}^{-\frac18} 
    \|g\|_{L^2_{t,x}}
    \| \P_{\pm, \hi} y^\dagger \|_{X^{s, \frac12+\dl}} \|y^\dagger\|_{\mathcal{X}^{\frac12}}^2, 
\end{align*}
which again gives a negative power of the maximum frequency, completing the estimate for $\mathcal{N}_h$ in \eqref{Nh}, and hence also that of \eqref{ysmoothing}. 
\end{proof}

With the two auxiliary estimates in Lemmas~\ref{LEM:Xsby}-\ref{LEM:tightv}, we now show tightness for $\wt v$ and $w$.

\begin{proof}[Proof of Proposition~\ref{PROP:tight}]

Let $0\le T \le T_{**} \le T_{*}$, with $T_{**}$ to be chosen later.
We begin with \eqref{tightv}. 
From the Duhamel formula for \eqref{veq2b}, the linear estimates in Lemma~\ref{LEM:linXsb}, and the conservation of $M(\wt u)$ and $P(\wt u)$, we obtain:
\begin{align*}
\|\Pi_{>\Ld}\wt{v}\|_{X^{s,b}_{T}} &  \leq  C_0\|\Pi_{>\Ld}\wt{v}(0)\|_{H^s} +C_0 T^{\dl_0} \| \Pi_{>\Ld}\P_{+}[ \wt{v}\P_{-}\dx( |\wt{u}|^2)]\|_{X^{s,-b'}_{T}} \\
&  +C_0 T^{\dl_0} \max_{\mathcal{Q}_h \in\{\Id, \GG_{h}\}} \| \P_{+} \Pi_{>\Ld}[  e^{i\be \wt{F}}  \wt{u} \mathcal{Q}_{h}(|\wt{u}|^2)  ]\|_{X^{s,-b'}_{T}}  \\
&+ C_0 T^{\dl_0} \|\wt{u}(0)\|_{H^{\frac 12}}^{2} \big( 1 + \|\wt u(0)\|^2_{L^2} \big) \| \Pi_{>\Ld} \wt{v}\|_{X^{s,-b'}_{T}} .
\end{align*}
By reducing $ T_{**}$, if necessary, so that $C_0 T_{**}^{\dl_0} 4 R^{2}(1+4 R^2) \leq \frac 12$, it follows from \eqref{Ns0Tbd} and \eqref{NTu}~that 
\begin{align*}
\|\Pi_{>\Ld}\wt{v}\|_{X^{s,b}_{T}} &  \leq  2C_0\|\Pi_{>\Ld}\wt{v}(0)\|_{H^s} +2C_0 T^{\dl_0} \| \Pi_{>\Ld}\P_{+}[ \wt{v}\P_{-}\dx( |\wt{u}|^2)]\|_{X^{s,-b'}_{T}} \\
&  +2C_0 T^{\dl_0} \max_{\mathcal{Q}_h \in\{\Id, \GG_{h}\}} \| \P_+ \Pi_{>\Ld}[  e^{i\be \wt{F}}  \wt{u} \mathcal{Q}_{h}(|\wt{u}|^2)  ]\|_{X^{s,-b'}_{T}}.
\end{align*}
In view of the signs of the frequencies, we have the (key) identity 
\begin{align*}
 \Pi_{>\Ld}\P_{+}[ \wt{v}\P_{-}\dx( |\wt{u}|^2)] =  \Pi_{>\Ld}\P_{+}[ \Pi_{>\Ld}\wt{v} \cdot\P_{-}\dx( |\wt{u}|^2)].
\end{align*} 
It then follows from Proposition~\ref{PROP:tri}, 
\eqref{Xsbu} in Lemma~\ref{LEM:uinfo}, and \eqref{Ns0Tbd}
that 
\begin{align*}
\| \Pi_{>\Ld}\P_{+}[ \wt{v}\P_{-}\dx( |\wt{u}|^2)]\|_{X^{s,b'}_{T}}  \leq C(R) \|\Pi_{>\Ld}\wt{v}\|_{X^{s,b}_{T}}.
\end{align*}

\noi 
Thus, by again reducing $T_{**}$ if necessary, we have 
\begin{align}
\|\Pi_{>\Ld}\wt{v}\|_{X^{s,b}_{T}}  \leq  4C_0 \Big(\|\Pi_{>\Ld}\wt{v}(0)\|_{H^s}   + T^{\dl_0} \max_{\mathcal{Q}_h \in\{\Id, \GG_{h}\}} \| \P_{+}\Pi_{>\Ld}[  e^{i\be \wt{F}}  \wt{u} \mathcal{Q}_{h}(|\wt{u}|^2)  ]\|_{X^{s,-b'}_{T}} \Big). \label{tightv0}
\end{align}
It remains to control the last on the right-hand side of \eqref{tightv0}. We write 
\begin{align*}
e^{i\be \wt{F}}  \wt{u} \mathcal{Q}_{h}(|\wt{u}|^2)  = \wt{y} \mathcal{Q}_{h}(|\wt{y}|^2) ,
\end{align*}
where $\wt{y}:= e^{i\be \wt{F}} \wt{u}$. It then follows from Lemma~\ref{LEM:tightv}, Lemma~\ref{LEM:Xsby}, \eqref{Ns0Tbd}, and \eqref{Hsscale}, that 
\begin{align*}
\max_{\mathcal{Q}_h \in\{\Id, \GG_{h}\}}& \| \P_{+}\Pi_{>\Ld}[  e^{i\be \wt{F}}  \wt{u} \mathcal{Q}_{h}(|\wt{u}|^2)  ]\|_{X^{s,-b'}_{T}}  \\
& \leq  C(1+L_{h}) \Ld^{-\frac{1}{10}} \| \wt{y}\|_{\mathcal{X}^{\frac{1}{2}}_{T}}^2 \big[  \| \wt{y}\|_{\mathcal{X}^{s}_{T}}+\|\wt{v}\|_{X^{s,b}_{T}}  \big] +C(1 + L_{h}) T^{\frac12}R^2 \|\Pi_{>\Ld}\wt{v}\|_{X^{s,b}_{T}} \\
&\leq C(R,L_{h})\Ld^{-\frac{1}{10}} \|u(0)\|_{H^{s}} + C(R,L_{h})T^{\frac 12}  \|\Pi_{>\Ld}\wt{v}\|_{X^{s,b}_{T}}
, 
\end{align*}
recalling that $L_h : = \| \GG_h\|_{\op} + \|\GG_h\|_{L^4\to L^4}$
Inserting this into \eqref{tightv0} and reducing $T_{**}$ further, if necessary, then establishes \eqref{tightv}.

The estimate for \eqref{tightw} is done similarly, using the Duhamel formula for $w$ based on \eqref{weq}. We note again the key identity 
\begin{align*}
\Pi_{>\Ld}\P_{-}[ w\P_{+}\dx( |u|^2)] = \Pi_{>\Ld}\P_{-}[ \Pi_{>\Ld}w \cdot \P_{+}\dx( |u|^2)],
\end{align*}
which, combined with Proposition~\ref{PROP:tri}, allows us to handle the contribution from the main term in \eqref{weq}, as we did for the main term in the equation for $\wt{v}$ above. The contribution from the remaining nonlinear terms in \eqref{weq} are handled by using Lemma~\ref{LEM:tightv} (with $y=u$) and recalling that $\P_{-,\text{hi}}u=w$. 
This completes the proof of Proposition~\ref{PROP:tight}.
\end{proof}

\subsection{Difference estimates}\label{SEC:diffs}

In this subsection, our aim is to control the difference of solutions and their corresponding gauge variables in suitable norms. 
Before stating the main result of this subsection (Lemma~\ref{LEM:sdiff}), we introduce some notations,  for $\frac12 \le s\le 1 $, $b=\frac12+\dl_0$, and $b'=\frac12-2\dl_0$, as before:
\begin{align}
\| f  \|_{S^{s}(T)}&:=  \|J^{s}_x \PbHI f\|_{\wt{L^{4}_{T,x}}} + \|f\|_{L^{\infty}_{T}H^{s}_x},  \label{Bs0} 
\\
\|({\bf{U}}, {\bf{\wt{U}}}, {\bf{ \wt{V}}}, {\bf{W}})\|_{R^{s}(T)}& : = \|{ \bf U}\|_{S^{s}(T)} + \|{ \bf \wt{U}}\|_{S^{s}(T)} + \|{\bf \wt{V}}\|_{X^{s,b}_{T}} + \| {\bf W}\|_{X^{s,b}_{T}}
.
\label{RsT}
\end{align}

\begin{lemma}
\label{LEM:sdiff}
Let $\frac12 \le s \le 1$, $u_j$, $j=1,2$ be two solutions to \eqref{INLS} with initial data $u_j(0)\in H^{1}(\T)$, $j=1,2$, which exist at least on the same time interval $[0,T_{0}]$ for $T_{0}>0$ only depending on $L_h : = \|\GG_h\|_{\op} + \|\GG_h\|_{L^4\to L^4}$ and 
\begin{align}
K_{s}: = \max_{j} \|u_{j}(0)\|_{H^{s}(\T)},  \label{L}
\end{align}
with $K:=K_{\frac12}$. 
Let $F_{j}=F_{j}[u_j]$ and $(v_j, w_j)$ be as in \eqref{F} and \eqref{gauge}, respectively, and $(\wt{u}_j,\wt{v}_j)$ as in \eqref{gauge2}, $j=1,2$.   
Also, define the differences 
$U:=u_1-u_2$, $\wt U = \wt u_1 - \wt u_2$, $\wt{V} := \wt{v}_1 - \wt{v}_2$, and $W: = w_1-w_2$, then the following estimates hold 
\begin{align}
\| (U, \wt{U}, \wt{V},W)\|_{R^{\frac 12}(T)} & 
\leq  
C(K)
\big( 1+K+T^{\frac{5}{8}}\Ld^{\frac{5}{4}}C_0(K)K+T^{\frac 14}C_0(K)
\big)
\|U(0)\|_{H^{ \frac 12}}  
\notag
\\
&   + C(K) (1+L_h)(1+T^{\frac 18}C_0 (K))\big[ K \Ld^{-\ta_3} + \|\Pi_{>\Ld}u_1(0)\|_{H^{\frac 12}} \big]
\label{sdiff2}
, 
\\
\| (U, \wt{U}, \wt{V},W)\|_{R^{s}(T)}&  \leq 
 C(K) K_s  (1+T^{\frac{5}{8}}\Ld^{\frac{5}{4}} ) \|U(0)\|_{H^{ s}}  
\notag 
\\
&   +C(K) (1+T^{\frac 18}C_0 (K))\big[ K \Ld^{-\ta_3} + \|\Pi_{>\Ld}u_1(0)\|_{H^{s}} \big]
\notag 
\\
& +C(K)\big[  T^\frac34 M^3 + K_s    \big]\|(U,\wt{U},\wt{V},W)\|_{R^{\frac 12}(T)}
, 
 \label{sdiff3}
\end{align}
for some  constants $C(K), C_0(K)>0$ depending only on $K$,$\Ld, M>0$ sufficiently large, and for $0 \le T \le T_0 \le 1$ sufficiently small, where $T_0 = T_0(K, L_h)$ for \eqref{sdiff2} and $T_0 = T_0(K_s,L_h, M)$ for \eqref{sdiff3}. 
\end{lemma}

Before proceeding to the proof, we prove some additional estimates involving differences of $F_j$ and of $e^{i\be F_j}$. In particular, \eqref{expdiffM} is crucial in proving \eqref{sdiff2} at the endpoint $s=\frac12$.

\begin{lemma}\label{LEM:expdiffM}
Let $0\le T \le 1$, $u_j$ be smooth solutions to \eqref{INLS}, $\wt u_j, F_j, \wt F_j, U, \wt U, K$ be as in Lemma~\ref{sdiff}. Then, the following estimates hold
\begin{align}
\| F_1 - F_{2}\|_{L^{\infty}_{x}} & \les   (\|u_1\|_{H^{\frac 12}_x} + \|u_2\|_{H^{\frac 12}_x})\|U\|_{H^{\frac 12}_x}
 \label{Ft0}
 , 
 \\
 \|\P_{\ll M}[e^{-i\be \wt{F}_1}-e^{-i\be \wt{F}_2}]\|_{L^{\infty}_{T,x}} &  \les   K\| {U}(0)\|_{H^{\frac 12}}
 + C(K) \big( M^{-\ta_4} + TM^{\frac 32} \big) \| \wt{U}\|_{L^{\infty}_{T}H^{\frac 12}_x}  ,
 \label{expdiffM}
\end{align}
for some $\ta_4>0$ and $C(K)>0$ only depending on $K$. Also, \eqref{Ft0}-\eqref{expdiffM} hold when exchanging $(\wt u_j, \wt F_j, \wt U)$ with $(u_j, F_j, U)$.

\end{lemma}

\begin{proof}

We omit details of the proof of \eqref{Ft0}, as it follows from \eqref{F} and Bernstein's inequality; see \eqref{L4est1}, for example.

To show \eqref{expdiffM}, we first write
\begin{align*}
e^{-i\be \wt{F_1}} - e^{-i\be \wt{F_2}} = -i\be (\wt{F_1}-\wt{F_2}) \int_{0}^{1} e^{-i\be ( \mu \wt{F_1} +(1-\mu)\wt{F_2})} d\mu =: (\wt{F_1}-\wt{F_2})  G(\wt{F_1},\wt{F_2}),
\end{align*}
and decompose the relevant term as follows
\begin{align}
\P_{\ll M}[ e^{-i\be \wt{F_1}} - e^{-i\be \wt{F_2}}]
& = \P_{\ll M} \big[ \P_{>2^{-10}M}( \wt{F_1}-\wt{F_2} )  G(\wt{F_1},\wt{F_2}) \big]  \label{expdiffM1}\\
& \quad + \P_{\ll M} \big[ \P_{\leq 2^{-10}M}( \wt{F_1}-\wt{F_2} )  \P_{>2^{-10}M}G(\wt{F_1},\wt{F_2}) \big]  \label{expdiffM2} \\
& \quad +\P_{\ll M} \big[ \P_{\leq 2^{-10}M}( \wt{F_1}-\wt{F_2} )  \P_{\leq 2^{-10}M} G(\wt{F_1},\wt{F_2}) \big].  \label{expdiffM3}
\end{align}
From \eqref{L4est1}, we have 
\begin{align*}
\| \eqref{expdiffM1}\|_{L^{\infty}_{T,x}} \les M^{-1} \| \P_{>2^{-10}M}(|\wt u_1|^2 - |\wt u_2|^2)\|_{L^{\infty}_{T, x}} \les M^{-1+} K \| \wt{U}\|_{L^{\infty}_{T}H^{\frac 12}_x}.
\end{align*}
For the second contribution, it follows from \eqref{Ft0} and \eqref{LinftyE11} that   
\begin{align*}
\| \eqref{expdiffM2}\|_{L^{\infty}_{T,x}} &
\les \| \wt{F_1}-\wt{F_2}\|_{L^{\infty}_{T,x}} \|\P_{>2^{-10}M} G\|_{L^{\infty}_{T,x}}  
\\
& 
\les K \|\wt{U}\|_{L^\infty_T H^\frac12_x}
\Big( \max_{j=1,2} \sup_{0\le \mu\le1} \| \P_{>2^{-20}M} e^{-i\be \mu \wt{F}_j} \|_{L^\infty_{T,x}}
\Big)
\\
&
\les M^{-1+ } K^3 \|\wt{U}\|_{L^\infty_T H^\frac12_x}
.
\end{align*}
As for the main term \eqref{expdiffM3}, by the fundamental theorem of calculus in time, \eqref{Ft0}, \eqref{dtF}, and \eqref{uniqueL2}, we have 
\begin{align*}
\| \eqref{expdiffM3}\|_{L^{\infty}_{T,x}} & \leq  C(K) \|\wt{U}(0)\|_{H^{\frac 12}} +  CTM^{\frac 32+} \| \dt\wt{F}_1 -\dt \wt{F}_2\|_{L^{\infty}_{T}H^{-1}_x}  \\
& \leq C(K) \|\wt{U}(0)\|_{H^{\frac 12}} + C(K)TM^{\frac 32+} \| \wt{U}\|_{L^{\infty}_{T}H^{\frac 12}_{x}}. 
\end{align*}
This proves \eqref{expdiffM}. The only difference for the ungauged variables $e^{-i\be F_j}$ is that there is one fewer (harmless) term in $ \dt F_j$ as compared to $\dt\wt{F}_j$, $j=1,2$. We omit the similar details.
\end{proof}

\begin{proof}[Proof of Lemma~\ref{LEM:sdiff}]

Let $s_1 \in\{\frac12, s\}$. 
We first recall a priori bounds for $u_j, \wt{v}_j, w_j$, $j=1,2$. 
By the results of Subsection~\ref{SEC:apriori}, namely \eqref{Ns0Tbd}, \eqref{M}, and \eqref{Hsscale}, we have the bounds
\begin{align}
\begin{split}
\max\big( \|u_{j} \|_{L^{\infty}_{T_{*}} H^{s_1}_x} 
, \| J^{s_1}_x u_j \|_{\wt{L^4_{T_*,x}}}
\big) 
&\leq 2C_1(1+\|u_{j}(0)\|_{H^{\frac 12}})^5 \|u_j(0)\|_{H^{s_1}} \leq Q_0 (K)K_{s_1}, 
\end{split} \label{ujsmall}
\end{align}
for $j=1,2$, $T_{\ast}>0$ is as in \eqref{Tast}, and $Q_0(x) = 2C_1 (1+x)^5 $. 
Then, 
from \eqref{vNs}-\eqref{wNs} and \eqref{ujsmall}, we~get
 \begin{align}
 \begin{split}
\max_{j=1,2} \big\{ \|\wt{v}_{j}\|_{X^{s_1,b}_T}+\|w_{j}\|_{X^{s_1,b}_T}\big\}& \les  \max_{j=1,2}  N_{T}^{s_1}(u_j) \leq Q_0 (K)K_{s_1},
\end{split}
 \label{vwnorms}
\end{align}
for any $0<T\leq T_{\ast}$.

\smallskip

We start by estimating the $\wt{V}$ and $W$ contributions in the $R^{s_1}(T)$-norm in \eqref{RsT}. 
At time $t=0$, from \eqref{gauge2}, \eqref{TT}, and \eqref{gauge}, we have
\begin{align*}
\wt{V}(0) = V(0) 
= 
\P_{+, \hi}\big[ (e^{i\be F_1} - e^{i\be F_2}) u_1\big](0) + \P_{+, \hi} \big[ e^{i\be F_2} U \big] (0)
 \quad \text{and} \quad 
W(0) =\P_{-, \text{hi}} U(0).
\end{align*}
Then, from \eqref{eFgdiff}, \eqref{eFg}, \eqref{ujsmall}, the mean value theorem, and \eqref{Ft0}, we have
\begin{align}
\begin{split}
\| \wt{V}(0)\|_{H^{s_1}}& \les Q(K)\big[ K \| {F_1}(0)- {F_2}(0)\|_{L^{\infty}} +  \| {U}(0)\|_{H^{s_1}} + (1 + K_{s_1}) \| U(0)\|_{H^\frac12} \big]
\\
&
\les Q(K) \big[  \| {U}(0)\|_{H^{s_1}} + (1 + K_{s_1}) \| U(0)\|_{H^\frac12} \big]
,
\\
\| W(0)\|_{H^{s_1}}  & \les \| U(0)\|_{H^{s_1}}
, 
\end{split} \label{V0W0}
\end{align}
for some polynomial $Q\ge0 $. 
Next, from \eqref{veq2b} and \eqref{weq}, we obtain equations for $\wt{V}$ and $W$:
\begin{align}
\dt \wt{V}+i\dx^2 \wt{V}  &=\NN_{\wt{v}_1}(\wt{u}_2)-\NN_{\wt{v}_2}(\wt{u}_2) \notag \\
 &= -2\be\Pbhip[ \wt{V} \P_{-}\dx(|\wt{u}_1|^2)]  -2\be\Pbhip[ \wt{v}_2 \P_{-}\dx(|\wt{u}_1|^2-|\wt{u}_2|^2)] \notag\\
&\hphantom{XX} +i\g \Pbhip[ (e^{i \be \wt{F_1}}-e^{i \be \wt{F_2}}) |\wt{u}_1|^2 \wt{u}_1] +i\g \Pbhip[ e^{i\be \wt{F_2}}( |\wt{u}_1|^2 \wt{u_1} - |\wt{u}_2|^2 \wt{u}_2)] \notag\\
& \hphantom{XX}  -i\be \Pbhip[(e^{i \be \wt{F_1}}-e^{i\be \wt{F_2}}) \wt{u}_1 \GG_{h}(|\wt{u}_1|^2)]   -i\be \Pbhip[ e^{i\be\wt{F_2}} (\wt{u}_1 \GG_{h}(|\wt{u}_1|^2)-\wt{u_2} \GG_{h}(|\wt{u}_2|^2 ))] \notag \\
& \hphantom{XX} +i\tfrac{\be}{\pi} \big[ \tfrac{\be}{4\pi}\big(M(\wt{u}_1)^2- M(\wt{u}_2)^2 \big) - P(\wt{u}_1)+ P(\wt{u}_2) \big]\wt{v}_1 + i\tfrac{\be}{\pi} \big[ \tfrac{\be}{4\pi}M(\wt{u}_2)^2  -P(\wt{u}_2) \big]\wt{V} \notag\\
&  =:\sum_{j=1}^{8} \mathcal{A}_{j}\label{Veqn}
, 
\\
\dt W+i\dx^2 W  &=\NN_{w_1}(u_1)-\NN_{w_2}(u_2) 
\notag
\\
&= 2\be \P_{-,\textup{hi}}[ W \P_{+}\dx(|u_1|^2)]+2\be \P_{-,\textup{hi}}[ w_2 \P_{+}\dx(|u_1|^2-|u_2|^2)] 
\notag
\\
&\hphantom{XX} -i\be \P_{-,\text{hi}}[ u_1 \GG_{h}(|u_1|^2)-u_2 \GG_{h}(|u_2|^2)] +  i\g \P_{-,\text{hi}}[ |u_1|^2 u_1 - |u_2|^2 u_2]
\notag
\\
& =: \sum_{j=1}^{4} \mathcal{B}_{j}.
\label{Weqn}
\end{align}

\noi 
We now estimate these terms using the Duhamel formulation of these equations and Lemma~\ref{LEM:linXsb}.
From Proposition~\ref{PROP:tri}, we have 
\begin{align}
\begin{split}
\max_{j=1,2}\| \mathcal{A}_j \|_{X^{s_1,-b'}_T } &  \les  T^{\dl_0} \Big(B(\wt{u}_1,\wt{u}_1)\|\wt{V}\|_{X^{s_1,b}_T }  +\|\wt{v}_2\|_{X^{s_1,b}_T}   [B(\wt{U}, \wt{u}_2)+ B(\wt{U},\wt{u}_2)]  \Big), 
\\ 
\max_{j=1,2}\| \mathcal{B}_j \|_{X^{s_1,-b'}_T } &  \les  T^{\dl_0} \Big(B(u_1,u_1)\|W\|_{X^{s_1,b}_T }  +  \|w_2\|_{X^{s_1,b}_T}   [B(U, u_1)+ B(U,u_2)] \Big), 
\end{split} \label{A12}
\end{align}
for $B$ as in \eqref{Bu}. 
Next, from \eqref{ZXembed}, \eqref{eFgdiff}, \eqref{eFg}, \eqref{Ft0}, the multilinear version of \eqref{cubic}, we have
\begin{align}
\begin{split}
\max_{j=3,\ldots, 6}\| \mathcal{A}_j \|_{X^{s_1,-b'}_T} & \les T^{\dl_0} (1 + L_h) C_0(K)[ K_{s_1} \|\wt{U}\|_{S^{\frac{1}{2}}(T)} +K \|\wt{U}\|_{S^{s_1}(T)} \big] , 
 \\
\max_{j=3,4} \| \mathcal{B}_{j}\|_{X^{s_1,-b'}_T} &\les T^{\dl_0}(1 + L_h)C_0(K)[ K_{s_1} \|U\|_{S^{\frac 12}(T)} +K \|U\|_{S^{s_1}(T)} \big] 
, 
\end{split} \label{A8}
\end{align}
for some $C_0(K)>0$ depending only on $K$.
Using the conservation of mass and momentum,  and \eqref{vwnorms}, we have
\begin{align}
\| \mathcal{A}_{7}+\mathcal{A}_{8}\|_{X^{s_1,-b'}_T}& \leq  T^{\frac 12} Q_{0}(K) 
 \|\wt{U}\|_{L^{\infty}_{T} H^{\frac 12}_x} +  T^{\frac 12} Q_{0}(K)^2 \|\wt{V}\|_{X^{s_1,b}_T}. 
 \label{A9}
 \end{align}
Combining \eqref{A12}, \eqref{A8}, \eqref{A9},   and recalling \eqref{Bu}, \eqref{NTu}, \eqref{Ns0Tbd},  and \eqref{vwnorms},
we then obtain
\begin{align}\label{VWbd}
\begin{split}
\|\wt{V}\|_{X^{s_1,b}_{T}} + \|W\|_{X^{s_1,b}_T}   &
\leq 
C_1 [\| \wt{V}(0)\|_{H^{s_1}} +  \| W(0)\|_{H^{s_1}}] 
\\
&
+C_0(K)T^{\ta} (1+L_h)  
\Big\{
 \| U\|_{\mathcal{X}^{\frac 12}_T}  +\| \wt{U}\|_{\mathcal{X}^{\frac 12}_T}+\|(U,\wt{U},\wt{V},W)\|_{R^{s_1}(T)} 
\\
&
\hspace{3.5cm}
+K_{s_1}(\|U\|_{S^{\frac 12}(T)} +\|\wt{U}\|_{S^{\frac 12}(T)}) 
\Big\},
\end{split} 
\end{align}
for some $\ta>0$. 

We now proceed to estimating the $\mathcal{X}^{\frac12}_T$- and $S^{s_1}(T)$-norms of $U, \wt{U}$ in \eqref{RsT}.
 From \eqref{XsbU}, \eqref{ujsmall}, and \eqref{Bs0}, we have 
\begin{align}
\begin{split}
\| U\|_{\mathcal{X}^{\frac 12}_T}
+
\| \wt U\|_{\mathcal{X}^{\frac 12}_T}
 &\leq C(K) [1+T^{\frac 14}(1+ L_h)] \|U\|_{S^{\frac 12}(T)}
  \end{split}\label{XsbU2}
\end{align}
for any $0<T\leq T_{\ast}$. 

Combining \eqref{JsU}, \eqref{JsU2}, \eqref{uHsbddiff}, \eqref{uHsbddiff2}, with \eqref{ujsmall}, \eqref{vwnorms}, and \eqref{Ft0}, we obtain
\begin{align}
\| U\|_{S^{s_1}(T)} & \leq C_1(1+T^{\frac{5}{8}}\Ld^{\frac{5}{4}}C_0(K)K_{s_1}) \|U(0)\|_{H^{ s_1}}  \notag \\
&+ (T^{\frac 18}C_0(K)+1)\big[ \| \wt{V}\|_{X^{s_1,b}_{T}}+\|W\|_{X^{s_1,b}_{T}} +\|\Pi_{>\Ld}\wt{v}_1\|_{X^{s_1,b}_{T}}\big]  \notag\\
&
+C_2(K) \big[T^{\frac34} M^3(1+L_h) +T^{\frac{1}{8}}(1+K_{s_1})  \big]\|U\|_{S^{\frac{1}{2}}(T)}  \notag\\
& +  \|\P_{\ll M}[e^{-i\be \wt{F_1}}-e^{-i\be \wt{F_2}}]\|_{L^{\infty}_{T,x}}  \|\P_{\ges M}\wt{v}_1\|_{X^{s_1,b}_{T}} \notag\\
&  + M^{-\ta} C_{2}(K)\big\{ \| \wt{V}\|_{X^{s_1,b}_{T}} + \|U\|_{S^{s_1}(T)}\big\},
 \label{UBsbd} \\
 \| \wt{U}\|_{S^{s_1}(T)} 
& \leq
 C_{1}(1+T^{\frac{5}{8}}\Ld^{\frac{5}{4}}C_0(K)K_{s_1}) \|U(0)\|_{H^{ s_1}}  \notag \\
&+ (T^{\frac 18}C_0(K)+1)\big[ \| \wt{V}\|_{X^{s_1,b}_{T}}+\|W\|_{X^{s_1,b}_{T}} +\|\Pi_{>\Ld}w_1\|_{X^{s_1,b}_{T}}\big]  \notag\\
&
+C_2(K) \big[T^{\frac34}M^3(1+L_h) +T^{\frac{1}{8}}(1+K_{s_1})  \big]\|\wt{U}\|_{S^{\frac{1}{2}}(T)}  \notag\\
&+  \|\P_{\ll M}[e^{-i\be F_1}-e^{-i\be F_2}]\|_{L^{\infty}_{T,x}}  \|\P_{\ges M}w_1\|_{X^{s_1,b}_{T}} \notag\\
&  + M^{-\ta} C_{2}(K)\big\{ \| \wt{V}\|_{X^{s_1,b}_{T}} + \|\wt{U}\|_{S^{s_1}(T)}\big\}
 .
\label{UBsbd2} 
\end{align}
Then, putting   \eqref{UBsbd}, \eqref{UBsbd2}, \eqref{VWbd}, \eqref{XsbU2}, \eqref{V0W0},together, we obtain
\begin{align}
\| (U, \wt{U}, \wt{V},W)\|_{R^{s_1}(T)} 
&
\leq 
 C_{1}\big( 1+T^{\frac{5}{8}}\Ld^{\frac{5}{4}}C_0(K)K_{s_1} +T^{\frac 14}C_0(K) + C_2(K) K_{s_1} \big) \|U(0)\|_{H^{ s_1}}  
 \notag
 \\
&   +(1+T^{\frac 18}C_0 (K))\big[ \|\Pi_{>\Ld}\wt{v}_1\|_{X^{s_1,b}_{T}} + \|\Pi_{>\Ld}w_1\|_{X^{s_1,b}_{T}}  \big] 
\notag
\\
& + C_2(K) \big\{ T^{\ta}(1+L_{h}) + M^{-\ta}  \big\} \| (U, \wt{U}, \wt{V},W)\|_{R^{s_1}(T)} 
\notag
\\
& + C_{2}(K)(1+L_{h})\big[ T^{\ta}K_{s_1}+T^{\frac34} M^{3} \big] \| (U, \wt{U}, \wt{V},W)\|_{R^{\frac 12}(T)} 
\notag
\\
& +\|\P_{\ll M}[e^{-i\be \wt{F_1}}-e^{-i\be \wt{F_2}}]\|_{L^{\infty}_{T,x}} \|\P_{\ges M}\wt{v}_1\|_{X^{s_1,b}_{T}}  
\notag
\\
& +\|\P_{\ll M}[e^{-i\be F_1}-e^{-i\be F_2}]\|_{L^{\infty}_{T,x}} \|\P_{\ges M}w_1\|_{X^{s_1,b}_{T}}  
.
\label{sdiff}
\end{align}
By choosing $T=T(K,L_{h})>0$ small enough and $M=M(K)>0$ large enough so that $C_{2}(K)T^{\ta}(1+L_{h})  <\frac{1}{4}$ and $M^{-\ta} C_2(K) < \frac14$, we can absorb the third term on the right-hand side of \eqref{sdiff} by the left-hand side. 

Next, we estimate the fifth and sixth terms on the right-hand side of \eqref{sdiff}.
Note that from \eqref{gauge}, fractional Leibniz, \eqref{LinftyE11}, and \eqref{LpE}, we have 
\begin{align}
& \| J_x^{s_1} \P_{\ges M}[ e^{i\be F_1(0)}u_1(0)]\|_{L^2_x} 
\les M^{-\frac 14} (1+\|u_1(0)\|_{H^{\frac 12}_x})^2 \|u_1(0)\|_{H^{\frac 12}_x}^2 + \|\P_{\ges M}u_1(0)\|_{H^{s_1}_x}. 
\label{Pv1data}
\end{align}

\noi 
Then by the Duhamel formula for $\wt{v_1}$ and $\wt{w}_1$,  Lemma~\ref{LEM:linXsb}, \eqref{Pv1data}, Proposition~\ref{PROP:tri}, Lemma~\ref{LEM:easyterms}, \eqref{XsbU2}, \eqref{ujsmall}, and \eqref{vwnorms}, we have
\begin{align}
\|\P_{\ges M}\wt{v}_1\|_{X^{s_1,b}_{T}}  
&
\les \|\P_{\ges M}u_1(0)\|_{H^{s_1}_x} + M^{-\frac 14} C_3(K)K +  T^{\dl_0} (1+L_h) Q_{0}(K)K_{s_1}\label{tightvM}
, 
\\
 \|\P_{\ges M}w_1\|_{X^{s_1,b}_{T}}  
 &
 \les \|\P_{\ges M}u_1(0)\|_{H^{s_1}_x} +  T^{\dl_0} Q_{0}(K)K_{s_1}. 
 \label{tightwM}
\end{align}

If $s>\frac12$, note that
\begin{align}
\| \P_{\ges M}u_1 (0)\|_{H^{\frac 12}} \les M^{-(s-\frac 12)} \|u_1(0)\|_{H^{s}}. \label{Mdata}
\end{align}
Then, we start by using \eqref{sdiff} when $s_1 = \frac12$, which is then used to obtain the estimate for $s_1=s$.  In particular, for $s_1=\frac12$,  combining \eqref{sdiff} (for sufficiently small $T>0$ and large $M>0$), \eqref{tightv}, \eqref{tightwM}, \eqref{Mdata}, mean value theorem, \eqref{Ft0}, Proposition~\ref{PROP:tight}, and \eqref{Pv1data},  we~get
\begin{align*}
\| (U, \wt{U}, \wt{V},W)\|_{R^{\frac12}(T)} 
& \leq 
 2  C_{1}\big( 1+T^{\frac{5}{8}}\Ld^{\frac{5}{4}}C_0(K)K +T^{\frac 14}C_0(K) + 
 C_2(K) K\big) \|U(0)\|_{H^{ \frac12}}  
\\
&   + 2 (1+T^{\frac 18}C_0 (K))
\big[
\Ld^{-\ta_3} C_2(K) + \| \Pi_{> \Ld} u_1(0) \|_{H^\frac12}
  \big]
 \\
& + 2 C_{2}(K)(1+L_{h})\big[ T^{\ta}K +T^{\frac34} M^{3} 
+ M^{-(s-\frac12)} \| u_1(0)\|_{H^s} 
\big] \| (U, \wt{U}, \wt{V},W)\|_{R^{\frac 12}(T)} 
, 
\end{align*}
where we take $M \ge M_0= M_0(K, L_h, \|u_1(0)\|_{H^s})\gg1$ potentially larger  so that 
$C_2(K)(1+L_h) M^{-(s-\frac12)} \|u_1(0)\|_{H^s} < \frac{1}{100}$ and $0 < T \le T_0 = T_0(K, L_h, M) \le 1$ possibly smaller to guarantee that 
$
C_2(K) (1+L_h) [ T^\ta K + T^{\frac34} M^3 ]  
< \tfrac{1}{100}
$, to hide the last contribution on the left-hand side:
\begin{align}
\begin{split}
\| (U, \wt{U}, \wt{V},W)\|_{R^{\frac 12}(T)} 
& \leq 
 4 C_{1}\big( 1+T^{\frac{5}{8}}\Ld^{\frac{5}{4}}C_0(K)K +T^{\frac 14}C_0(K) + C_2(K) K\big) \|U(0)\|_{H^{ \frac12}}  
\\
&   + 4 (1+T^{\frac 18}C_0 (K))
\big[
\Ld^{-\ta_3}
C_2(K) + \| \Pi_{> \Ld} u_1(0) \|_{H^\frac12}
  \big]
\end{split}
\notag
\end{align}
for any $0<T\leq T_{0} = T_0(L_h, K_s)$. Then, for the same choice of $0 \le T \le T_0$ and $M \ge M_0 \gg1$, combining \eqref{sdiff} for $s_1=s$ with Proposition~\ref{PROP:tight}, and \eqref{Pv1data} gives \eqref{sdiff3}.

It remains to show the estimate at the endpoint $s=\frac12$, namely \eqref{sdiff2}, where we cannot use \eqref{Mdata}, and thus need a more careful argument to avoid picking $T_0$ depending on the profile of $u_1(0)$ instead of only its norm.  Using \eqref{tightvM},  \eqref{tightwM}, and \eqref{expdiffM}, we control the fifth and sixth terms in \eqref{sdiff} by 
\begin{align*}
& C K \big( \| \P_{\ll M } [ e^{-i\be\wt{F}_1} - e^{-i\be \wt{F}_2}] \|_{L^\infty_{T,x}}
+
\| \P_{\ll M } [ e^{-i\be{F}_1} - e^{-i\be {F}_2}] \|_{L^\infty_{T,x}} \big)
\\
&
\qquad
+ 
C T^{\dl_0} (1+L_h) Q_0(K) K^2
 \big( \| \wt U \|_{S^\frac12(T)} + \|  U \|_{S^\frac12(T)}
\big) 
\\
&
\le 
CK^2  \|  U (0)\|_{H^\frac12}
+ 
C'(K) 
\big\{ M^{-\ta_4} + TM^{\frac32}
+
T^{\dl_0} (1+L_h) 
\big\}
 \big( \| \wt U \|_{S^\frac12(T)} + \|  U \|_{S^\frac12(T)}
\big) . 
\end{align*}
Replacing the above estimate in \eqref{sdiff}, together with Proposition~\ref{PROP:tight}, \eqref{Pv1data}, and taking $M\ge M_0 = M_0(K, L_h) \gg1 $ potentially larger and $0<T\le T_0 = T_0(K,L_h, M)\le 1$ smaller, to hide the terms depending on $\| (U, \wt U, \wt V, W) \|_{S^{\frac12}(T)}$ on the left-hand side, completing the proof of \eqref{sdiff2}. 
\end{proof}

\subsection{Proof of Theorem~\ref{THM:LWP}}

Now that we have \eqref{sdiff2} and \eqref{sdiff3} in Lemma~\ref{LEM:sdiff}, the local well-posedness of \eqref{INLS} follows from a modification of a standard argument. See \cite[Section 4C]{MP}, for example. 
Given $u_0 \in H^{s}(\T)$, we consider the sequence of smooth initial data $\{ u_{0,j}\}_{j\in \N}$ defined by $u_{0,j} = \P_{\leq j} u_0.$
Then, $\{ u_{0,j}\}_{j\in \N}$ belongs to $H^{\infty}(\T)$ and by Proposition~\ref{PROP:LWPH2}, the associated sequence of solutions $\{u_{j}\}_{j\in \N}$ to \eqref{INLS} belong to $C([0,T_{\ast}]; H^{\infty}(\T))$, for some $T_* = T_*(\|u_0\|_{H^s})>0$. 
From Subsection~\ref{SEC:apriori} and \eqref{Hsscale}, we then have that $N^{s}_{T_{\ast}}(u_{j})\leq Q_{0}(K) K_s$, uniformly in $j$, where $Q_0(x) = (1+x)^5$ and $K$ is as in \eqref{L}.
Note that in place of \eqref{Mdata}, we may choose $M$ uniformly in $j$ since
\begin{align*}
 \| \P_{\ges M} u_{0,j}\|_{H^{\frac 12}_x}
 \le 
 \| \P_{\ges M} u_{0}\|_{H^{\frac 12}_x}
 \les M^{-(s-\frac 12)} \|  u_{0}\|_{H^s_x}.
\end{align*}
Then, from Lemma~\ref{LEM:sdiff}, there exists $0< T_0 \le T_\ast$ such that both \eqref{sdiff2} and \eqref{sdiff3} hold for  $0<T\leq T_{0}$, where all constants are independent of $j$.   

Let $\wt{u}_j, \wt{v}_j$ be the gauge variables in \eqref{gauge2} associated with $u_j$, and ${v}_j, w_j$ as in \eqref{gauge}. It follows from \eqref{sdiff2} that, for $j,j'\in\N$, we have
\begin{align*}
\| (u_j-u_{j'}, \wt{u}_j-\wt{u}_{j'},& \wt{v}_{j}-\wt{v}_{j'}, w_{j}-w_{j'})\|_{R^{\frac 12}(T)}   \\
&
\leq  
C(K)
\big( 1+K+T^{\frac{5}{8}}\Ld^{\frac{5}{4}}C_0(K)K+T^{\frac 14}C_0(K)
\big)
\|(u_j - u_{j'})(0)\|_{H^{ \frac 12}}  
\notag
\\
& 
\quad 
+ C(K) (1+L_h)(1+T^{\frac 18}C_0 (K))\big[ K \Ld^{-\ta_3} + \|\Pi_{>\Ld}u_0\|_{H^{\frac 12}} \big]
.
\end{align*}
Thus, taking limits in $j,j'$, followed by taking $\Ld\to\infty$, allows us to 
conclude that $\{u_{j}\}_{j\in\N}$ is a Cauchy sequence in $S^{\frac 12}(T_0)$,
while $\{\wt{v}_j\}_{j\in\N}, \{w_j\}_{j\in\N}$ are Cauchy in $X^{\frac 12,b}_{T_{0}}$.
Consequently, $\{u_{j}\}_{j\in\N}$ converges to a limit $u \in S^{\frac 12}(T_{0})$ and $\{\wt{v}_j\}_{j\in\N}$ and $\{w_j\}_{j\in\N}$ converge to limits $\wt{v}$ and $w$ in  $X^{\frac 12,b}_{T_{0}}$, respectively.
By using this convergence property in $R^{\frac 12}(T)$ in \eqref{sdiff3}, by a similar argument, we upgrade the convergence of $\{u_j\}_{j\in\N}, \{\wt v_j\}_{j\in\N}, \{w_j\}_{j\in\N}$ to $R^{s}(T)$.
 Moreover, a computation shows that $u$ is a distributional solution to \eqref{INLS}, with $w=\P_{-,\text{hi}}u$ and $v= \Pbhip[ e^{i\be F[u]}u]$.

We now move onto the uniqueness property of these solutions, as in Theorem~\ref{THM:LWP}~(i). 
Let $\underline{u}$ be another solution to \eqref{INLS} on $[0,T]$, for some $T>0$, with $\underline{u}\vert_{t=0} =u_0$ and belonging to the same class as $u$ in \eqref{uclass}. We then define $K$ as in \eqref{L} and, by reducing $T$ if necessary, we have that  $u$ and $\underline{u}$ satisfy \eqref{ujsmall}.
 We define the gauged variables $\wt{\underline{v}} : = \P_{+, \hi}[e^{i\be F[\underline{u}]} \underline{u}]$ and $\underline{w}= \P_{-,\text{hi}} \underline{u}$. 
By \eqref{v0w0}, 
\begin{align*}
\| \underline{v}(0)\|_{H^{ \frac 12}}+ \| \underline{w}(0)\|_{H^{ \frac 12}} \leq Q_{0}(K). 
\end{align*}
By the dominated convergence theorem, we can find $N>0$ such that 
\begin{align*}
\|\ind_{[0,T]} \P_{> N}\NN_{\underline{\wt{v}}}(\wt{\underline{u}})\|_{X^{\frac 12, -b'}}+ \| \ind_{[0,T]} \P_{> N}\NN_{\underline{w}}(\underline{u})\|_{X^{\frac 12, -b'}} \leq \tfrac{1}{2}Q_{0}(K).
\end{align*}
Also, from \eqref{ZXembed}, \eqref{weq}, \eqref{veq2b},  \eqref{uniqueL2} and similar computations as in Lemma~\ref{LEM:easyterms}, we have
\begin{align*}
\|\ind_{[0,T]} \P_{\leq N}\NN_{\underline{\wt{v}}}(\underline{u})\|_{X^{\frac 12, -b'}}&+ \| \ind_{[0,T]} \P_{\leq N}\NN_{\underline{w}}(\underline{u})\|_{X^{\frac 12, -b'}} \\
& \les T^{\ta}N^{-\frac 12} \big(   \|\NN_{\underline{\wt{v}}}(\wt{\underline{u}})\|_{L^{2}_{T}H_x^{-1}} +  \|\NN_{\underline{w}}(\underline{u})\|_{L^{2}_{T}H^{-1}_{x}}      \big) \\
&\les T^{\ta}N^{-\frac 12}(1 + \|\wt{\underline{v}}\|_{L^{2}_{T}H_x^{\frac 12}}) \les T^{\ta}N^{-\frac 12}(1+ \|\wt{\underline{v}}\|_{X^{\frac 12,0}_{T}}).
\end{align*}
Then, by reducing $T>0$ if necessary, it follows that $N_{T}^{1/2}(\underline{u})\leq 2R$, with $R$ as in \eqref{M}, thus satisfying \eqref{Ns0Tbd}. Then, we may apply \eqref{sdiff2} to $ (u - \underline{u}, \wt u - \wt{\underline{u}}, \wt v - \wt{\underline{v}}, w - \underline{w})$ and take $\Ld \to \infty$ to see that $\underline{u} \equiv u$ on $[0,T]$. This establishes the uniqueness claim (iterating in time, if necessary).

As for the continuous dependence, we can follow a similar argument as in \cite[Section 4C]{MP} relying on \eqref{sdiff2} and \eqref{sdiff3}. We point out that the choices of $\Ld$ and $M$ can be made uniformly over any convergent approximating sequence of initial data.

\section{The infinite depth limit} \label{SEC:conv}

In this section, we establish the local-in-time convergence of solutions to \eqref{INLS} to solutions to \eqref{CCM} as $h\to \infty$.
We focus on the low regularity setting $\frac 12 \leq s\leq 1$. The convergence in the remaining range $1<s\leq \frac 32$ follows from the same ideas and the estimates already appearing earlier. Fix $u_0\in H^{s}(\T)$ and let $\{u_{0,h}\}_{1\leq h<\infty}$ be a net converging to $u_0$ in $H^s(\T)$ as $h\to \infty$.

By Theorem~\ref{THM:LWP}, and in view of Remark \ref{REM:unifLh}, there exists $T_{\ast} = T_\ast(\|u_0\|_{H^s}) >0$ such that we have a unique solution $u_{\infty}$ to \eqref{CCM} with $u_{\infty}|_{t=0}=u_0$ and, for each $1\leq h<\infty$, a unique solution $u_{h}$ to \eqref{INLS} with $u_{h}\vert_{t=0}=u_{0,h}$ on $[0,T_{\ast}]$. 
We then define $F_{\infty}:=F[u_{\infty}]$ and $F_{h}:=F[u_{h}]$, where $F$ is the primitive defined in \eqref{F}, and the gauged variables $(v_{\infty}, w_{\infty}), (v_{h},w_{h})$ as in \eqref{gauge} and $\wt v_\infty, \wt v_h$ as in \eqref{gauge2}. 
For $(u_{\infty}, \wt{v}_{\infty}, w_{\infty})$ and $(u_h, \wt{v}_{h}, w_{h})$, we define $N_{T}^{s}$ as in \eqref{NTu}.
As $u_{0,h}$ converges to $u_0$, there exists $h_0\geq 1$ such that 
\begin{align*}
\sup_{h\geq h_0} \| u_{0,h}\|_{H^{s}} \leq \|u_0\|_{H^{s}}+1 =:K_s.
\end{align*}
Then, by repeating the arguments of Subsection~\ref{SEC:apriori} (noting that $T_{\ast}$ can be chosen uniformly in $h_0\leq h\leq \infty$), we have 
\begin{align}
N^{s_1}_{T}(u_{h})  \leq 2^{7}Q_{0}(K)K_{s_1} =: \wt{Q}_{0}(K) K_{s_1}\label{Nsh}
\end{align}
for $s_1\in\{\frac12, s\}$, any $0<T\leq T_{\ast}$, and uniformly in $h_0 \leq h\leq \infty$.
First we consider the Fourier restriction norm estimates for the difference $U_{h} := u_{h}-u_{\infty}$. Using \eqref{INLS2} and \eqref{INLS2}, we see that $U_{h}$ satisfies
\begin{align*}
(\dt +i\dx^2)U_{h}  =& 2\be \big[u_{h}\P_{+}\dx( |u_{h}|^2) - u_{\infty}\P_{+}\dx( |u_{\infty}|^2) ] +i\g( |u_h|^2 u_h- |u_{\infty}|^2 u_{\infty}) 
  -i\be u_{h}\mathcal{G}_{h}(|u_h|^2),
\end{align*}
noting that there is no dependence on $u_\infty$ on the third term on the right-hand side. 
By repeating the proof of \eqref{XsbU} and using \eqref{Nsh}, we then obtain
\begin{align}
 \| U_{h}\|_{\mathcal{X}^{ \frac 12}_{T}} \leq C( 1 + \wt{Q}_{0}(K)^2 K^2   ) \big( \|U_{h}\|_{L^{\infty}_{T}H^{\frac 12}_{x}}  + \|J^\frac12_x U_{h}\|_{\wt{L^4_{T,x}}}   \big) 
 + C \wt{Q}_{0}(K)^3 K^3 \| \mathcal{G}_{h}\|_{\text{op}}, 
 \notag
\end{align}
where $C>0$ is independent of $h_0\leq h\leq \infty$.
Proceeding analogously for $\wt{U}_h$ and combining both estimates, we obtain
\begin{align}
\| U_h\|_{\mathcal{X}^\frac12_T} + \| \wt U_h \|_{\mathcal{X}^\frac12_T} 
\le 
C( K )
\big\{ 
\| U_h\|_{S^\frac12(T)} + \|\wt U_h \|_{S^\frac12(T)}
\big\}
 + C (K) \| \mathcal{G}_{h}\|_{\text{op}}
 .
    \label{UhXsb}
\end{align}
Next, we estimate the differences $\wt{V}_{h}:= \wt{v}_{h} - \wt{v}_{\infty}$ and $W_{h} : = w_{h}-w_{\infty}$.
Namely, using that $\wt{V}_{h}$ satisfies \eqref{Veqn} with $\wt{u}_1=\wt{u}_h$, $\wt{u}_2 =\wt{u}_{\infty}$, and with $\mathcal{A}_5, \mathcal{A}_6$ replaced by
\begin{align*}
\mathcal{A}_{5} = -i\be \Pbhip[ e^{i\wt{F_{h}}}\wt{u_{h}} \GG_{h}(|\wt{u_h}|^2)] \quad \text{and} \quad \mathcal{A}_{6}=0,
\end{align*}
and $W_{h}$ satisfies \eqref{Weqn} with the same modifications and with $\mathcal{B}_{3}$ replaced by
\begin{align*}
\mathcal{B}_{3} = -i\be \P_{-,\text{hi}}[ u_{h}\GG_{h}(|u_h|^2)], 
\end{align*}
we repeat the analysis leading to \eqref{VWbd}, which together with \eqref{Nsh} and \eqref{UhXsb} gives
\begin{align}
\begin{split}
\|\wt{V}_h\|_{X^{s_1,b}_{T}} +\|W_h\|_{X^{s_1,b}_T}   
&
\leq C_1 [\| \wt{V}_h (0)\|_{H^{s_1}} +  \| W_h(0)\|_{H^{s_1}}]  
+ C(K)K_{s_1}T^{\ta}  \|\GG_{h}\|_{\text{op}}
\\
& 
\quad 
+C_0(K)T^{\ta}\big\{ 
\|(U_h,\wt{U}_h,\wt{V}_h,W_h)\|_{R^{s_1}(T)} \\
&\hphantom{XXXXXXXX}+(1+K_{s_1})(\|U_h\|_{S^{\frac 12}(T)} +\|\wt{U}_h\|_{S^{\frac 12}(T)}) \big\} ,
\end{split} \label{VWbdh}
\end{align}
for some $\ta>0$ and constants $C, C_1, C_0>0$ independent of $h_0 \leq h\leq \infty$, where we used that $\sup_{h_0 \leq h}L_{h}\les 1$. Note that the second term on the RHS of \eqref{VWbdh} comes from bounding $\mathcal{A}_{5}$ and $\mathcal{B}_3$.
We can similarly repeat the arguments used in the proof of \eqref{UBsbd}-\eqref{UBsbd2} to obtain bounds for $U_h, \wt U_h$ in $S^\frac12(T)$, which include an additional term  $C(K)T \|\GG_h\|_{\op}$ coming from the term $\P_{\le M} (u_h \GG_h(|u_h|^2)$ appearing when adapting \eqref{uHsbddiff} (and \eqref{uHsbddiff2}). 
Altogether, we  find
\begin{align}
\| (U, \wt{U}, \wt{V},W)\|_{R^{s_1}(T)} 
&
\leq 
 C_{1}\big( 1+T^{\frac{5}{8}}\Ld^{\frac{5}{4}}C_0(K)K_{s_1} +T^{\frac 14}C_0(K) + C_2(K) K_{s_1}\big) \|U_h(0)\|_{H^{ s_1}} 
 \notag
 \\
&   
\quad 
+(1+T^{\frac 18}C_0 (K))\big[ \|\Pi_{>\Ld}\wt{v}_\infty\|_{X^{s_1,b}_{T}} + \|\Pi_{>\Ld}w_\infty\|_{X^{s_1,b}_{T}}  \big] 
\notag
\\
& 
\quad 
+ C_{2}(K)\big[ T^{\ta}K_{s_1}+T^{\frac34} M^{3} \big] \| (U_h, \wt{U}_h, \wt{V}_h,W_h)\|_{R^{\frac 12}(T)} 
\notag
\\
& 
\quad +\|\P_{\ll M}[e^{-i\be \wt{F_h}}-e^{-i\be \wt{F_\infty}}]\|_{L^{\infty}_{T,x}} \|\P_{\ges M}\wt{v}_\infty\|_{X^{s_1,b}_{T}} 
\notag\\
& \quad 
+\|\P_{\ll M}[e^{-i\be F_h}-e^{-i\be F_\infty}]\|_{L^{\infty}_{T,x}} \|\P_{\ges M}w_\infty\|_{X^{s_1,b}_{T}}  
\notag
\\
&
\quad 
+ C (K)K_{s_1}T^\ta \|\GG_{h}\|_{\text{op}}. 
\notag
\end{align}
By repeating the arguments in the proof of Lemma~\ref{LEM:sdiff}, involving choosing $M=M(K_{s_1})$ sufficiently large, and then $T=T(K_{s_1})$ sufficiently small, we get
\begin{align*}
 \| (U_h, \wt{U}_h, \wt{V}_h,W_h)\|_{R^{\frac12}(T)} 
 &\les
 C(K, \Ld, T)  \| U_h(0)\|_{H^{\frac12}} +C(K)[ \Ld^{-\ta_3} +  \|\Pi_{>\Ld}u_{\infty}(0)\|_{H^\frac12}] +    \|\GG_{h}\|_{\text{op}},
 \\
 \| (U_h, \wt{U}_h, \wt{V}_h,W_h)\|_{R^{s}(T)} 
 &\les
 C(K, \Ld, T)  \| U_h(0)\|_{H^{s}} +C(K)[ \Ld^{-\ta_3} +  \|\Pi_{>\Ld}u_{\infty}(0)\|_{H^s}] 
 \\
 &\quad 
 + C(K,T,M,K_s) \|(U_h, \wt U_h, \wt V_h, W_h)\|_{R^\frac12(T)} +    \|\GG_{h}\|_{\text{op}},
\end{align*}
for any $0<T\leq T_0$, and where the constants are uniform in $h_0\leq h\leq \infty$. Now by taking $h\to \infty$ using that $\lim_{h\to\infty}\| \mathcal{G}_{h}\|_{L^p\to L^p}=0$ for any $1<p<\infty$ and $U_{h}(0)\to 0$ in $H^{s}(\T)$, and then taking $\Ld\to \infty$, this establishes that $u_{h}\to u_{\infty}$ in $C([0,T_0];H^{s}_x)\cap L^{4}([0,T_0];W^{s,4}(\T))$ as $h\to \infty$. 
By iterating this argument, we obtain the convergence up to time $T=T_{\ast}$.

\section{Unconditional uniqueness}\label{SEC:UU}

In this section, we complete the proof of Theorem~\ref{THM:LWP} by showing the uniqueness claims in (ii)-(iii). 
In particular, 
we show that if $u_0 \in H^{s}(\T)$ with $s>\frac 35$, then uniqueness holds in the class $L^{\infty}_{T}H^{s}_x \cap L^{4}_{T}W^{s,4}_{x}$, proving (ii).  By the Sobolev embedding $H^{s+\frac{1}{4}}(\T) \embeds W^{s,4}(\T)$, this implies (iii), namely, unconditional uniqueness of solutions in $L^{\infty}_{T}H^{s}_x$ for $s>\frac{3}{5}+\frac{1}{4}= \frac{17}{20}$.
From the uniqueness in Theorem~\ref{THM:LWP}~(i), matters are reduced to verifying that under these assumptions, the gauged variables $\wt{v}$ and $w$, defined in \eqref{gaugeclass} belong to $X^{\frac 12,\frac 12+}_{T}$. We broadly follow the approach in \cite{MP}; however, in order to obtain unconditional well-posedness in the energy space ($s=1$), we need  to refine the argument.

In the following, let $u\in L^\infty_T H^s_x \cap L^4_T W^{s,4}_x$ be a solution to \eqref{INLS}. From \eqref{Xsbu}, we get that $u\in X^{s-\ta,\ta}_T$ for any $\ta \in[0,1]$, and thus $w =\P_{-,\text{hi}}u \in X^{s-\ta,\ta}_{T}$. 
If $s>1$, this suffices to show that $w \in X^{\frac12, \frac12+}_T$. When $\frac12<s\le 1$, we rely on the following crucial smoothing property for~$w$.

\begin{lemma}\label{LEM:wreg}
Let $s>\frac12$, $u\in L^{\infty}_{T}H^{ s}_x \cap L^4_{T}W^{s,4}_x $ be a distributional solution to \eqref{INLS}, and $w := \P_{-,\textup{hi}}u$. Then,  $w\in X^{ \min(\frac{3s-1}{2}, s)-,\frac 12+}_{T}$.

\end{lemma}
\begin{proof}
First, we note that $w\in L^\infty_T H^s_x \cap  L^{4}_{T}W^{s,4}_x$ and it solves \eqref{weq}.
Proceeding as in the proof of \eqref{Xsbu}, we start by estimating $\| w\|_{X^{\s,1}_T}$, for some $\frac12<\s \le s$ to be chosen later, which reduces to estimating the terms
 \begin{align}
\| \P_{-,\text{hi}}( w\P_{+}\dx(|u|^2))\|_{L^2_{T}H^{\s}_x} \quad \text{and} \quad \| -i\be u \GG_{h}(|u|^2)+i\g |u|^2 u\|_{L^2_{T}H^{\s}_x} 
.
\label{wXsb1}
\end{align}
 The second term in \eqref{wXsb1} is estimated using the algebra property of $H^\s$, $\s>\frac 12$, and the boundedness of $\GG_h$. 
 Thus, we focus on estimating the first term in \eqref{wXsb1}.
By dyadic decomposition, we have
\begin{align}
\|& \P_{-,\text{hi}}( w\P_{+}\dx(|u|^2))\|_{L^2_{T}H^{\s}_x}^{2} 
\notag
\\
& \les \sum_{N \geq 1} N^{2\s} \bigg(  \sum_{  \substack{N_1, N_2,N_3, N_{23} \\ N_1 \ges N\vee N_{23}  }   } \big\|\P_{N} \big[  \P_{N_1}w \cdot \P_{+}\dx\P_{N_{23}}( \cj{\P_{N_2}u} \P_{N_3}u)    \big] \big\|_{L^{2}_{T,x}}\bigg)^2.
\label{uu1a}
\end{align}
Note that the extra condition $N_1 \ges N \vee N_{23}$ arises from the signs of the frequencies as we have $|\xi_1| = |\xi|+|\xi_3 -\xi_2|$.
First, if $N_{23} \les 1$, then using H\"older and Bernstein's inequalities,
\begin{align*}
\big\|\P_{N} \big[  \P_{N_1}w \cdot \P_{+}\dx\P_{N_{23}}( \cj{\P_{N_2}u} \P_{N_3}u)    \big] \big\|_{L^{2}_{T,x}}& \les N_{23} \|\P_{N_1}w\|_{L^{4}_{T,x}} \| \P_{N_2}u \cdot \P_{N_3}u\|_{L^{4}_{T,x}} \\ 
& \les T^{\frac 14} N_1^{-s} (N_2 N_3)^{-\frac18} \|J^{s}\P_{N_1}w\|_{L^{4}_{T,x}}\prod_{j=2}^3 \|\P_{N_j}u\|_{L^{\infty}_{T}H^{\frac{1}{2}}_x} .
\end{align*}
We can then perform the dyadic sums provided that $\s < s$, obtaining the bound
\begin{align*}
\text{LHS }\eqref{uu1a}
\les 
T^{\frac 14} \|w\|_{L^{4}_{T}W^{s,4}_x} \|u\|_{L^{\infty}_{T}H^{\frac 12}_{x}}^2.
\end{align*}
Now we consider the case when $N_{23}\ges 1$. Without loss of generality, we may assume that $N_2 \geq N_3$ and so $N_{23}\les N_2$. 
Then, by H\"{o}lder's inequality, we have
\begin{align*}
\big\|\P_{N} \big[  \P_{N_1}w \cdot \P_{+}\dx\P_{N_{23}}( \cj{\P_{N_2}u} \P_{N_3}u)    \big] \big\|_{L^{2}_{T,x}}& \les \frac{N_{23}}{N_{1}^{s}N_2^s} \|J^s \P_{N_1}w\|_{L^{4}_{T,x}} \| J^{s}\P_{N_2}u \|_{L^{4}_{T,x}} \|  \P_{N_3}u\|_{L^{\infty}_{T,x}} .
\end{align*}
In order to perform the dyadic summations, we need to control the dyadic factor which is 
\begin{align*}
N^{\s} N_{23} (N_1 N_2)^{-s} \les N^{\s} N_{23}^{1-s}  N_1^{-s}  \les N_{1}^{1- 2 s + \s}.
\end{align*}
This power is negative provided that $\s<2s-1$, and we obtain the bound
\begin{align*}
\text{LHS }\eqref{uu1a} 
\les 
\| w\|_{L^{4}_{T} W^{s,4}_x} \| \PbHI u\|_{L^{4}_{T}W^{s,4}_x} \| u\|_{L^{\infty}_{T} H^{s}_x}.
\end{align*}
This show that $w\in X^{\min(2s-1,s)-, 1}_{T}$. Interpolating this with $w\in X^{ s, 0}_{T}$ we get $w\in X^{\min(\frac{3s-1}{2},s)-, \frac{1}{2}+}_{T}$, as intended.
\end{proof}

At this point, Lemma~\ref{LEM:wreg} already implies that $w\in X_{T}^{\frac 12,\frac 12+}$ for $s>\frac 23$. 
The following nonlinear estimate allows us to further improve this restriction (to $s>\frac 35$).

\begin{lemma}\label{LEM:wunique}
Let $\frac 12 < s \le 1$, $ 1- s < \s \le  s $, and $0 < \dl \ll 1 $. Then,
it holds that 
\begin{align}
\begin{split}
\| \ind_{[0,T]}& \P_{-\textup{hi}}[ w \P_{+}\dx(|u|^2)] \|_{X^{ \s, -\frac 12 +2\dl}}\\
&  \les \|w\|_{X^{\s, \frac{1}{2}+\dl}_{T}}\big( \| u\|_{L^{\infty}_{T}H_x^{ s}} +\|J_x^{s}\PbHI u\|_{L^{4}_{T,x}} + \| u\|_{X^{s-1,1}_{T}})^2.
\end{split} \label{wuniq1}
\end{align}
\end{lemma}

\begin{proof}
We proceed as in the proof of \eqref{vX1}-\eqref{wX1}.  
By duality and dyadic decomposition, we bound
\begin{align}
 \int_{\R} \int_{\T} \jb{\dx}^{\s} \cj{\P_{N}g} \cdot \P_{-,\text{hi}}[ \P_{N_1}( \ind_{[0,T]}w) \, \P_{+}\P_{N_{23}}\dx( \cj{\P_{N_2} u} \P_{N_3}u)]dx dt
\label{wXsb2}
\end{align}
under the conditions \eqref{N1cond} and \eqref{N23}, where $g\in X^{0,\frac 12-2\dl}$, for some small $\dl>0$. When $N_2  \sim N_3$, we argue as in Case 1 of \eqref{vX1}, noting that since $s>\frac 12$, we can use $L^{4}_{T,x}$ for $u$ instead of the slightly stronger $\wt{L^{4}_{T,x}}$. When $N_2 \vee N_3 \gg N_2 \wedge N_3$ and $N_2 \land N_3 \ges N$, we use H\"{o}lder's inequality and  \eqref{L4}, to bound
\begin{align*}
\eqref{wXsb2} &\les N^{\s}N_{23}  \| \P_{N_1}w_{T}\|_{L^4_{t,x}} \|\P_{N_2}u\|_{L^{4}_{t,x}} \|\P_{N_3}u\|_{L^{4}_{t,x}} \\
& \les N^{\s}N_{23}  N_1^{-\s} N_{2}^{-s} N_3^{-s} \|\P_{N_1}w\|_{X^{ \s, \frac{3}{8}}} \|J_x^s \P_{N_2}u\|_{L^{4}_{t,x}} \|J_x^s \P_{N_3}u\|_{L^{4}_{t,x}}.
\end{align*}
We then control the dyadic factor as follows: 
\begin{align*}
N^{\s}N_{23}  N_1^{-\s} N_{2}^{-s} N_3^{-s}  
& \les  (N_2 \land N_3)^{\s - s } (N_2 \lor N_3)^{1-s} N_1^{-\s} 
\les N_{\max}^{1-s-\s}
\end{align*}
since $\s \le s$, 
which is a negative power provided that $\s>1-s$ .

Now we consider the case when $N_2 \vee N_3 \gg N_2 \wedge N_3$ and $N_2 \land N_3 \ll N$. Here we have both \eqref{nonres} and \eqref{N1max}. We decompose according to the modulation variables as in \eqref{X0}-\eqref{X3} (note that by the same argument, the corresponding term \eqref{Xres} will be identically zero due to impossible interactions).
As in the bound for \eqref{X0}, we have 
\begin{align*}
 & \bigg|\int_{\R} \int_{\T} \jb{\dx}^{\s}\cj{ \Q_{\ges K}\P_{N}g} \cdot \Pbhip[ \P_{N_1}w_T \, \P_{-}\dx( \cj{\P_{N_2}u} \P_{N_3}u)]dx dt \bigg|  \\
& \les N^{\s}  K^{-\frac 12+2\dl}N_{23} \|\P_{N_1}w\|_{L^{4}_{t,x}}  \| \P_{N_2 \vee N_3}u\|_{L^{4}_{t,x}} \|\P_{N_2\wedge N_3}u\|_{L^{\infty}_{t,x}} \\
& \les  N^{\s+\dl}  K^{-\frac 12+2\dl}N_{23} N_{1}^{-\s} (N_2 \vee N_3)^{-s} \|\P_{N_1}w\|_{X^{\s, \frac{3}{8}}} \|J_x^s \PbHI u\|_{L^{4}_{t,x}}\| \PbHI u\|_{L^{\infty}_{t}H^{\frac 12}_x}.
\end{align*}
The resulting dyadic multiplier is then 
\begin{align*}
\frac{ N^{\s+\dl} N_{23}  }{ K^{\frac 12-2\dl} N_{1}^{\s}  (N_2 \vee N_3)^{s}  }
&\les \frac{ N^{\s -\frac 12+3\dl} N_{23}^{\frac{1}{2}+2\dl -s }  }{  N_{1}^{\s}  } \les N_{1}^{-\frac 14} \sim N_{\max}^{-\frac 14 },
\end{align*}
given that $\dl>0$ is sufficiently small. The contribution for the corresponding term of the form \eqref{X1} when $w$ has high modulation is handled analogously, by placing $(g, w)\in L^4_{t,x}\times L^2_{t,x}$, which leads us to controlling the following dyadic multiplier
\begin{align*}
    \frac{N^\s N_{23} (N_2 \land N_3)^{\frac12-s+}}{K^\frac12 N_1^\s (N_2 \lor N_3)^s}
    \les \frac{N^{\s-\frac12} N_{23}^{\frac12-s} (N_2 \land N_3)^{\frac12-s + } }{ N_1^\s } \les N_{\max}^{-\frac12}. 
\end{align*}
When $\P_{N_2}u$ has high modulation as in \eqref{X2}, we have
\begin{align*}
N^{\s+\dl} N_{23} N_1^{-\s} \|\P_{N}g\|_{X^{0,\frac 12-2\dl}} \|\P_{N_1}w\|_{X^{\s,\frac 12}} K^{-1} N_{2}^{1-s} \| \P_{N_2}u\|_{X^{s-1, 1}} \|\P_{N_3}u\|_{L^{\infty}_{t}H^{\frac 12}}.
\end{align*}
For the dyadic factor, we have 
\begin{align*}
&
\frac{N^{\s+\dl} N_{23} N_2^{1-s}}{ K N_1^\s }
\les 
\frac{N^{\s+\dl} N_{23} (N_2 \lor N_3)^{1-s}}{N_{23} N  N_1^\s} 
\les N^{\s+\dl -1 } N_1^{1-s-\s} \les N_{\max}^{1-s-\s + \dl}
\end{align*}
where we used $1-s<\s\le 1$.
Finally, we have the term when $\P_{N_3}u$ has a high 
modulation, as in \eqref{X3}, where we consider the further case separation $N_2 \ll N_3$ or $N_2 \gg N_3$. If $N_2 \ll N_3$, we use H\"older's inequality, \eqref{L4}, \eqref{QKLp} since $K \sim N N_3 \gg N_2^2$, to control the quantity by
\begin{align*}
  N^\s N_{23} N^{\frac12+} K^{2\dl+} \| \P_N g \|_{X^{0,\frac12-2\dl}} N_1^{-\s} \| \P_{N_1} w\|_{X^{\s, \frac12+}} N_2^{-s} \| J^s \P_{N_2 } u\|_{L^4_{T,x}} 
  N_3^{1-s} K^{-1 } \| \P_{N_3} u \|_{X^{s-1,1}}
  ,
\end{align*}
where it remains to estimate the multiplier
\begin{align}
    \frac{N^{\s + \frac12+} N_{23} N_3^{1-s} }{ N_1^{\s} K^{1-2\dl - } N_2^{s} }
    \les \frac{ N^{\s + \frac12+} N_3^{1-s+2\dl+} }{N_1^\s (N\lor N_3)^{1-2\dl-}}
    \les N_{\max}^{\frac12 -s + 4\dl + }
    \notag
\end{align}
since $N\lor N_3 \sim N_1 \sim N_{\max}$, where the power is negative for $s>\frac12$ and $\dl>0$ sufficiently small. If $N_2 \gg N_3$, we instead use \eqref{L4} to place the $\P_{N_2}u$-term in $X^{s-\frac38, \frac38}$, giving the following multiplier 
\begin{align*}
    \frac{N^{\s+ \frac12+} N_{23} N_3^{1-s}}{ N_1^\s K^{1-2\dl-} N_2^{s-\frac38}} \les \frac{N^{\s+\frac12+} N_3^{1-s} }{ N_1^\s N_2^{s - \frac38 -2\dl -} N^{1-2\dl-} } \les \frac{N^{\s + \frac12 + 2\dl - s +}}{N_1^{\s}} \les N_{\max}^{\frac12 +2\dl -s +}, 
\end{align*}
which again has a negative power. 
This completes the proof. \qedhere
\end{proof}

Lastly, we combine the earlier results to prove unconditional uniqueness.

\begin{proof}[Proof of Theorem~\ref{THM:LWP}~(ii)]
    
Let $s> \frac35$ and $u \in L^\infty_T H^s_x \cap L^4_T W^{s,4}_x$ be a solution to \eqref{INLS}. 
We first show that $w = \P_{-,\hi} u$ is in $X^{\frac12 , \frac12+}_T$. 
If $s >1$, we argue as discussed at the start of Section~\ref{SEC:UU}. In the following, we take $\frac35 < s \le 1$. 
From the conditions on Lemma~\ref{LEM:wreg} and Lemma~\ref{LEM:wunique}, we want to take $1-s < \s < \min(\frac{3s-1}{2},s) = \frac{3s-1}{2}$, which requires that $s>\frac35$. Then, using Lemma~\ref{LEM:wunique} to estimate the first term on the RHS of \eqref{weq} and \eqref{L4} for the remaining ones, 
we can run a contraction mapping argument to obtain a unique solution to \eqref{weq} in $ X^{\s, \frac12+\dl}_{T'}$ where $\s = \frac{3s-1}{2}-\dl$, for fixed $u$, with initial data $\P_{-, \hi}u_0$, and for some $0 < T' \le T$.
Similarly, using Proposition~\ref{PROP:tri} and Lemma~\ref{LEM:easyterms}, we can solve equation \eqref{weq} via a contraction mapping argument in $X^{s, \frac12+\dl}_T$ to obtain a unique solution $w' \in X^{s, \frac12+}_T$. Since $\s \le s$, we have that $w' \in X^{\s, \frac12+\dl}_T$, and by the earlier uniqueness, we must have 
$w' = w = \P_{-, \hi} u \in X^{\frac12, \frac12+}_T$.

We now move onto showing that $\wt{v}:= \P_{+, \hi}[ e^{i\be F[u]}u] (t,x-2\be \P_{0}(|u_0|^2)t)\in X^{\frac 12,\frac 12+}_{T}$. In order to show that $\wt{v}$ solves \eqref{veq2b}, we need to first check that $\P_{0}(|u(t)|^2)$ is conserved.
To prove this, we fix $N\in \N$ and consider the truncated variable $u_{N}(t) : = \P_{\leq N}u(t)$. Then, $u_{N}\in C^{\infty}(\R\times \T)$ and solves
\begin{align}
\dt u_{N}+i\dx^2 u _{N}
&
= 2\be u_{N} \P_{+}\dx(|u_{N}|^2)-i\be u_{N} \GG_{h}(|u_{N}|^2)+i\g |u_{N}|^2 u _{N} +R_{N}\label{INLSN}
, 
\\
\text{where }\quad 
R_{N} & =  2\be \big\{  \P_{\leq N}[  u\P_{+}\dx(|u|^2)]-   u_{N} \P_{+}\dx(|u_{N}|^2) \big\} 
\notag 
\\
& \quad -i\be \big\{ \P_{\leq N}[ u \GG_{h}(|u|^2) ]-  u_{N} \GG_{h}(|u_{N}|^2) \big\}
+ i\g \big\{ \P_{\leq N}[ |u|^2 u]- |u_N|^2 u_N \big\}.
\notag
\end{align}
From the definition of the mass in \eqref{mass} and \eqref{INLSN}, it follows that
\begin{align*}
\frac{d}{dt} M(u_{N}(t)) = 2 \text{Re} \int_{\T} \cj{u_{N}(t)} R_{N} dx.
\end{align*}
Let $\frac 12 <s_0<s$, which implies that $1-s_0 \le s$.
Using Cauchy-Schwarz and \eqref{prodestalg}, we have 
\begin{align*}
\bigg| 2 \text{Re} \int_{\T} \cj{u_{N}(t)} R_{N} dx \bigg|  &
\les \| u_{N}\|_{H^{1-s_0}} \| R_{N}\|_{H^{s_0-1}} \leq C_{h}( \|u\|_{L^{\infty}_{T}H^{s}_x}) \big[ \| u-u_{N}\|_{H^{s_0}_{x}} +N^{-(s-s_0)}] \\
& \leq \wt{C}_{h}( \|u\|_{L^{\infty}_{T}H^{s}_x})N^{-(s-s_0)}.
\end{align*}
Therefore, 
\begin{align}\label{MuN}
M(u_{N}(t)) =M(\P_{\leq N}u_0)+O(N^{-(s-s_0)})
\end{align}
and since $u_{N}(t)\to u(t)$ in $L^2(\T)$ for each fixed $t\in [0,T]$, as $N\to \infty$, we conclude that $M(u(t))= M(u_0)$ for all $t\in [0,T]$.

It follows from Lemma~\ref{LEM:leib} that $\wt{v}\in L^{\infty}_{T}H^{s}_{x} \cap L^{4}_{T}W^{s,4}_{x}$. Moreover as $s>\frac 12$, one can check that $\wt{v}$ solves \eqref{veq2b} in $L^{\infty}_{T} H^{s-2}_{x}$. Then, by repeating the argument in Lemma~\ref{LEM:wreg}, we get $\wt{v}\in X^{\max(\frac{3s-1}{2}+\dl, s),\frac{1}{2}+\dl}_{T}$.  In particular, we are done if $s>1$. If $s\leq 1$, then we follow the same argument as we did for $w$ and note that \eqref{wuniq1} also holds for the nonlinear term $\P_{+,\hi}[ \wt{v}\,  \P_{-}\dx(|u|^2)].$
We thus get $\wt{v}\in X^{\frac{1}{2},\frac 12+}_{T}$ provided that $s>\frac 35$. This completes the proof of unconditional uniqueness for \eqref{INLS}.
\end{proof}

\section{Global well-posedness}\label{SEC:GWP}

We start by showing global well-posedness in the energy space.

\begin{proof}[Proof of Corollary~\ref{COR:gwp}]
 The global-in-time result of Corollary~\ref{COR:gwp} follows from iterating the local argument of Theorem~\ref{THM:LWP} once we derive an a priori bound on the $H^1(\T)$ norm of solutions to \eqref{INLS}. 
Let $u$ be a smooth solution to \eqref{INLS} with $u(0)=u_0$. From the energy in \eqref{Energy}, we have
\begin{align}
\int_\T |\dx u|^2 dx 
&
= E(u_0) 
+ \int_\T 
\big\{ 
\be |u|^2 \Im(u\dx \cj u) - \tfrac\be2 |u|^2 \TT_h\dx (|u|^2) - \tfrac\g2|u|^4 - \tfrac{\b^2}{3}|u|^6
\big\}
dx 
\label{E0}
\end{align}
For the first term in the integral, from H\"older's and Young's inequality, we have
\begin{align}
\bigg| \int_\T \be |u|^2 \Im(u\dx \cj u) dx \bigg|
&
\le |\be| \| u\|^3_{L^6} \| \dx u\|_{L^2} 
\le \tfrac{|\be|^2}{2\eps_0} \|u\|^6_{L^6} + \tfrac{\eps_0}{2} \|\dx u\|^2_{L^2}. 
\label{energy-bad}
\end{align}

In the defocusing case, $\be>0$, from \eqref{E0} with \eqref{energy-bad}, and the positive-definiteness of the operator $\TT_h \dx$, we obtain
\begin{align*}
\| \dx u \|^2_{L^2} 
\le 
E(u_0) 
+ 
\tfrac{\g^2}{8(2\pi)\eps_1} M(u_0) 
+ \big( \tfrac{\be^2}{2\eps_0} + \tfrac{\eps_1}{2}- \tfrac{\be^2}{3} \big) \|u\|^6_{L^6} 
+ \tfrac{\eps_0}{2} \| \dx u \|^2_{L^2}
,
\end{align*}
where we used that $\frac{|\g|}{2}|u|^4 \le \frac{\g^2}{8 \eps_1} |u|^2 + \frac{\eps_1}{2} |u|^6$ for $\eps_1>0$, and the conservation of mass in~\eqref{mass}. Thus, by choosing $\frac32<\eps_0 <2$ and $\eps_1 = \eps_1(\eps_0, \be) >0$ sufficiently small, we obtain the intended a priori bound. 

Next, we consider the focusing case, $\be<0$. 
In this case, we need to estimate the second term in the integral in \eqref{E0}. In particular, 
 from H\"older's inequality, $L^p$-boundedness of $\TT_h$, Gagliardo-Nirenberg-Sobolev inequality \cite[Lemma 4.1]{LRS}, and Young's inequality, we have
\begin{align*}
\int_\T |u|^2 \TT_h\dx(|u|^2) dx
&
\le 
\| u\|^2_{L^6} \| \TT_h \dx(|u|^2) \|_{L^{\frac32}}
\\
&
\le  C_h \| u\|^3_{L^6} \| \dx u \|_{L^2}
\\
&
\le C_h \big( (C_{\textup{GNS}} + \eps) \|u\|^4_{L^2} \| \dx u \|^2_{L^2} + C_\eps \|u\|^6_{L^2} \big)^\frac12 \| \dx u \|_{L^2}
\\
&
\le C_h (C_{\textup{GNS}}+ \eps)^\frac12  \|u\|^2_{L^2} \| \dx u\|^2_{L^2} + C_h C_\eps^\frac12 \|u\|^3_{L^2} \|\dx u \|_{L^2}
\\
&
\le C_h \big\{ (C_{\textup{GNS} } + \eps)^\frac12 M(u_0) + \eps' \big\} \| \dx u \|^2_{L^2}
+ 
C_h^2 C_\eps C_{\eps'} M(u_0)^3
\, 
\end{align*} 
where $C_{\textup{GNS}}>0$ denotes the optimal constant in the Gagliardo-Nirenberg-Sobolev inequality on $\R$, $C_h>0$ depends on $\|\TT_h\|_{L^p\to L^p}$, $0<\eps, \eps'\le 1$, and $C_\eps, C_{\eps'} >0$ depend only on $\eps, \eps'$, respectively. 
Thus, proceeding as in the defocusing case, combining \eqref{E0}, \eqref{energy-bad}, and the estimate above, for $\eps_0$ as before,  $\eps,\eps', \eps_1>0$ and $M(u_0)>0$ sufficiently small, we obtain an a priori bound for $\int_{\T} |\dx u|^2 dx$, completing the proof. 
\end{proof}

Next, we turn to the $\gamma=0$ case of \eqref{INLS}:
\begin{align}
\dt u + i\dx^2 u = 2\beta u \Pih \dx \big( |u|^2 \big),\label{INLSG}
\end{align}
where
\begin{equation*}
\Pih = \tfrac{1}{2}(1+i\mathcal{T}_{h}) .
\end{equation*}
Our goal is to establish a priori estimates in $H^s (\T)$ for $\frac12 \leq s < 1$ (see Theorem~\ref{THM:apriori} below), and so finish the proof of Theorem~\ref{THM:GWP}.

\begin{proposition}\label{PROP:Laxpair}
For any $0<h\leq \infty$ and $c\in\R$, $u(t)$ solves~\eqref{INLSG} on the circle if and only if the operators
\begin{equation*}
\lax_{u;h} = -i\partial_x +\beta u (\Pih + c\P_0) \cj{u} 
\quad\text{and}\quad
\peter_{u;h} = -i\partial_x^2 + 2\beta u\partial_x\Pih \cj{u}
\end{equation*}
on $L^2(\T)$ satisfy
\begin{equation}
\tfrac{d}{dt}\lax_{u;h}f = [\peter_{u;h},\lax_{u;h}]f + \tfrac{\be^2}{2} u \P_0 \big[ (1+\TT_h^2)\partial_x\big( |u|^2 \big) \cdot \cj{u}f \big] + \tfrac{\be^2}{2} u (1+\TT_h^2)\partial_x\big( |u|^2 \big) \cdot \P_0(\cj{u}f)  
\label{lax T}
\end{equation}
for any $f\in H^\infty(\T)$.
\end{proposition}

The proof of Proposition~\ref{PROP:Laxpair} follows an argument parallel to the line case in \cite[Proposition 6.1]{CFL}, thus we omit details. We note that, on $\T$, the operators $\L_{u;h}, \peter_{u;h}$ only constitute an \emph{almost} Lax pair, due to the second and third terms on the right-hand side of \eqref{lax T}. These arise from the application of the following Cotlar-type identity on the circle 
\begin{align}
 \TT_{h}[ (\P_{\neq 0}g) \TT_{h} k + (\P_{\neq 0}k) \TT_{h}g] = \P_{\neq 0}\big[ \TT_{h} g \cdot \TT_{h}k - (\P_{\neq 0}g)( \P_{\neq 0}k)\big] , 
\label{Tilb22}
\end{align}
as mean zero corrections needed for \eqref{Tilb22} to apply. 
However, the Lax relation is recovered in the limit $h=\infty$, as we see that \eqref{lax T} becomes 
\begin{equation}
\tfrac{d}{dt}\lax_{u;\infty} = [\peter_{u;\infty},\lax_{u;\infty}] ,
\label{lax T 2}
\end{equation}
for a solution $u$ to \eqref{CCM},
since $\TT_{\infty}^2 = \H^2 = -\P_{\neq 0}$ and $\P_{0}\partial_x(|u|^2) = 0$.
In this case, we take the Lax pair as
\begin{align}
\lax_{u;\infty} & 
= -i\partial_x + \beta u \Pi_+ \cj u , 
\qquad 
\peter_{u;\infty}
= -i \dx^2 + 2 \be u \dx
\Pi_+ \cj u, 
\label{LPinfty}
\end{align}
where the projector $\Pi_+$ is given by
\begin{align}
    \Pi_+ : = \P_+ + \P_0  = \tfrac12(\Id + i \H + \P_0 ). 
    \label{Pi+}
\end{align}
Note that \eqref{LPinfty} comes from Proposition~\ref{PROP:Laxpair}, by setting $h=\infty$, $c=\frac12$, and writing $\dx \Pi_{+, \infty} = \dx \Pi_+$. We make these choices to have agreement with the operators that have appeared previously in the  literature for CCM \eqref{CCM}, while having no effect in either~\eqref{lax T} or~\eqref{lax T 2}.

The following result summarizes what we will need about the Lax operator.  Note that these results are new, as the potential $u$ is not required to lie in a Hardy space.  We omit the details however, since the proof is identical to that in the line case from our previous work. See \cite[Proposition 6.2]{CFL}.
\begin{proposition}[Lax operator]
Fix $0<h\leq\infty$.  Given $u\in H^{\frac14}(\T)$, the operator
\begin{equation*}
\lax_{u;h} f = -i\partial_x f + \beta u(\Pih + \tfrac12 \P_0) (\cj{u}f)
\end{equation*}
with domain $H^1 (\T)$ is self-adjoint.  Moreover, there is a constant $C\geq 1$ so that whenever
\begin{equation}
\kappa \geq C \max\big( \|u\|_{L^2}^2 , \|u\|_{H^\frac14}^4 \big), 
\label{k0 T}
\end{equation}
we have
\begin{equation}
\tfrac12 ( \lax_0^2 + \kappa^2 ) \leq \lax_{u; h}^2 + \kappa^2 \leq \tfrac32 (\lax_0^2 + \kappa^2)
\label{mono T}
\end{equation}
as quadratic forms, where $\lax_0 := \lax_{0;h} = -i\partial_x$.
\end{proposition}

Since for $0< h <\infty$, the pair of operators $\L_{u;h},\peter_{u;h}$ on $\T$ do not form an exact Lax pair, in particular,
the relation \eqref{lax T} presents a rank-two perturbation of the usual Lax relation.
Instead, we treat \eqref{INLSG} as a perturbation of the limiting case $h=\infty$, where the relation \eqref{lax T 2} holds exactly. 
Specifically, \eqref{INLSG} is equivalent to
\begin{equation}
\dt u + i\dx^2 u = 2\beta u \Pi_+ \dx \big( |u|^2 \big) -i\beta u \GG_h\big(|u|^2\big) ,
\label{INLSG 2}
\end{equation}
where $\GG_h$ is as in \eqref{Lh}. 
Note that this operator satisfies $\cj{ \GG_h(f) } = \GG_h (\cj{f})$, and so the function $\GG_h\big(|u|^2\big)$ appearing in \eqref{INLSG 2} is real-valued.
Comparing \eqref{INLSG 2} with \eqref{CCM}, we see that the equations agree up to the second term on the right-hand side of \eqref{INLSG 2}.  
Moreover, given a smooth solution $u$ of \eqref{INLSG 2}, 
we can write 
\begin{align}
\partial_t u 
& 
= \peter_{u;\infty} u - i\beta u \GG_h\big( |u|^2 \big),  
\label{INLSP}
\\
\tfrac{d}{dt} \lax_{u, \infty} f 
&= [ \peter_{u;\infty}, \lax_{u,\infty} ] f + iu\big[ \Pi_+, \GG_h \big( |u|^2 \big)\big] \cj{u}f  .
\label{lax T 3}
\end{align}

\noi
Algebraically, this represents a further departure from the Lax pair relation~\eqref{lax T 2}.  Analytically though, \eqref{lax T 3} is preferable to \eqref{lax T} for proving a priori estimates, since the operator $\GG_h$ and the commutator $[\Pi_+, \GG_h(|u|^2)]$ are smoothing, as shown in the following lemma.

\begin{lemma}
For any $0<h\leq \infty$, we have
\begin{align}
\big\| \GG_h\big(|u|^2\big) \big\|_{H^3}
&\lesssim \|u\|_{L^2}^2 ,
\label{G smooth}\\
\big\| \big[ \Pi_+, \GG_h\big(|u|^2\big) \big] f \big\|_{H^1} 
&\lesssim \|u\|_{L^2}^2 \|f\|_{H^{-1}} .
\label{G smooth 2}
\end{align}
Moreover, the implicit constants can be chosen uniformly for $1\leq h\leq \infty$.
\end{lemma}

\begin{proof}

The estimate~\eqref{G smooth} follows from the smoothing property of the $\GG_h$ operator which maps $H^{s_1}(\T)$ to $H^{s_2}(\T)$ for any $s_1 \le s_2$ (see \cite[Lemma 2.1]{CFLOP-2}, for example) together with the embedding $L^1(\T)\hookrightarrow H^{-1}(\T)$:
\begin{equation*}
\big\| \GG_h\big(|u|^2\big) \big\|_{H^3}
\lesssim \big\| |u|^2 \big\|_{H^{-1}}
\lesssim \big\| |u|^2 \big\|_{L^1}
\lesssim \|u\|_{L^2}^2 .
\end{equation*}

Next, we turn our attention to~\eqref{G smooth 2}.  As a first step, we claim that for any $\sigma\geq 0$ we have
\begin{equation}
\big\| [ \H, g ] f \big\|_{L^2} \lesssim \|g\|_{H^{\sigma+1}} \|f\|_{H^{-\sigma}} .
\label{calderon 2}
\end{equation}
In Fourier variables, this is equivalent to the estimate
\begin{equation}
\bigg\| \sum_{\xi = \xi_1+\xi_2} \frac{\langle\xi_2\rangle^\sigma}{\langle\xi_1\rangle^{\sigma +1}} (\sgn(\xi) - \sgn(\xi_2)) u(\xi_1)v(\xi_2) \bigg\|_{\l^2_\xi} \lesssim \|u\|_{\l^2} \|v\|_{\l^2} .
\label{calderon}
\end{equation}
If $\xi_2 \ne 0$, the factor $\sgn(\xi) - \sgn(\xi_2)$ is only nonzero when $|\xi_1|\geq |\xi_2|$, and thus by Young's inequality, we have
\begin{align*}
\text{LHS \eqref{calderon}}
&\lesssim \big\| (\langle\xi\rangle^{-1} u)\ast v \big\|_{\l^2} 
\leq \| \langle\xi\rangle^{-1} u \|_{\l^1} \|v\|_{\l^2}
\lesssim \|u\|_{\l^2} \|v\|_{\l^2} , 
\end{align*}
while if $\xi_2=0$, it reduces to
\begin{equation*}
\text{LHS \eqref{calderon}}
= \big\|  \langle\xi\rangle^{-\sigma-1} u(\xi)v(0) \big\|_{\l^2_\xi}
\leq \|u\|_{\l^2} \|v\|_{\l^2} .
\end{equation*}
Thus, \eqref{calderon} follows, and so too does \eqref{calderon 2}.

We can now consider $[\Pi_+, \GG_h(|u|^2)]$, with $\Pi_+$ as in \eqref{Pi+}. Since the contributions from $\frac12[1 + \P_0, \GG_h(|u|^2)]$ are harmless, we omit details, and focus on those from $\frac12[\H, \GG_h(|u|^2)]$.
 Combining \eqref{G smooth} and the $\sigma=1$ case of \eqref{calderon 2}, we deduce
\begin{align*}
\big\| \big[ \Pi_+, \GG_h\big(|u|^2\big) \big] f \big\|_{L^2} 
+
\big\| \big[ \Pi_+, \dx \GG_h\big(|u|^2\big) \big] f \big\|_{L^2} 
&\lesssim \|u\|_{L^2}^2 \|f\|_{H^{-1}}
.
\end{align*}
On the other hand, the $\sigma=2$ case of \eqref{calderon 2} implies
\begin{equation*}
\big\| \big[ \Pi_+, \GG_h\big(|u|^2\big) \big] \dx f \big\|_{L^2} 
\lesssim \|u\|_{L^2}^2 \|\dx f\|_{H^{-2}} 
\lesssim \|u\|_{L^2}^2 \|f\|_{H^{-1}} .
\end{equation*}
Collecting the previous three inequalities, we obtain \eqref{G smooth 2}.
\end{proof}

We are now equipped to prove our a priori estimates.
\begin{theorem}[A priori estimates] \label{THM:apriori}
Let $\frac12\leq s < 1$ and $0 < h \leq \infty$. Then, there exists $r>0$ such that for any $A,T>0$, there exists $B>0$ so that 
\begin{equation}
\|u(0)\|_{L^2} \leq r \quad\text{and}\quad \| u(0) \|_{H^s} \leq A \quad\implies\quad \sup_{|t|\leq T}\, \| u(t) \|_{H^s}\leq B , 
\label{ap T}
\end{equation}
for all (global) $H^\infty(\T)$ solutions to~\eqref{INLSG}. 
Moreover, the constants $r$ and $B$ can be chosen uniformly for $1\leq h \leq \infty$.
\end{theorem}
\begin{proof}

Fix $\frac12 \leq s < 1$, recall that $\L_0 = -i\dx$, and set $\kk$ as in \eqref{k0 T}. Using the short-hand notation $\lax_u = \lax_{u;\infty}$,  we consider
\begin{equation*}
F(\lax_u^2+\kappa^2) := \big\langle u, (\lax_u^2+\kappa^2)^{s} u \big\rangle .
\end{equation*}
By Loewner's Theorem, the function $x\mapsto x^{s}$ on $(0,\infty)$ is operator monotone (see \cite{Simon} for details).
In particular, the relation~\eqref{mono T} implies
\begin{equation*}
F\big(\tfrac12 (\lax_0^2 + \kappa^2)\big) \leq F( \lax_{u; \infty}^2 + \kappa^2 ) \leq F\big(\tfrac32(\lax_0^2 + \kappa^2)\big) ,
\end{equation*}
and so
\begin{equation}
\tfrac{1}{C} F (\lax_0^2 + \kappa^2) \leq F( \lax_{u; \infty}^2 + \kappa^2 ) \leq C F(\lax_0^2 + \kappa^2)
\label{equiv}
\end{equation}
for some constant $C>0$.

Next, we will compute the time evolution of $F( \lax_{u}^2 + \kappa^2 )$.
Let $u$ be a global $H^\infty(\T)$-solution to~\eqref{INLSG}.    For any $E\geq 0$, a simple residue computation demonstrates 
\begin{equation*}
\frac{\sin(\pi s)}{\pi} \int_0^\infty \frac{E\lambda^{s-1}}{E+\lambda}\, d\lambda = E^{s} .
\end{equation*}
Consequently, if $A$ is a positive self-adjoint operator and $f$ is in the domain of $A$, then,  by the functional calculus, we have
\begin{equation}
A^{s}f = \frac{\sin(\pi s)}{\pi} \int_0^\infty \lambda^{s-1} (A+\lambda)^{-1}Af\, d\lambda .
\notag
\end{equation}
We apply this to $A = \lax_u^2 + \kappa^2$, which is positive by~\eqref{mono T}.  Combining this with the resolvent identity and~\eqref{lax T 3}, we find
\begin{align*}
\tfrac{d}{dt} \big( \lax_u^2 + \kappa^2 \big)^{s}f
={} &\frac{\sin(\pi s)}{\pi} \int_0^\infty \lambda^{s-1} (\lax_u^2 + \kappa^2+\lambda)^{-1} \big(\tfrac{d}{dt} \lax_u^2 \big)f\, d\lambda   \\
& - \frac{\sin(\pi s)}{\pi} \int_0^\infty \lambda^{s-1} (\lax_u^2 + \kappa^2+\lambda)^{-1} \big( \tfrac{d}{dt} \lax_u^2 \big) (\lax_u^2 + \kappa^2+\lambda)^{-1} (\lax_u^2 + \kappa^2)f\, d\lambda  \\
={} & \big[ \peter_{u;\infty}, \big( \lax_u^2 + \kappa^2 \big)^{s} \big] f  \\
& + i (\lax_u^2 + \kappa^2)^{s-1}  \lax_u \big\{ u \big[ \Pi_+ , \big( \GG_h \big( |u|^2 \big) \big] \cj{u}f \big\} \\
& + i (\lax_u^2 + \kappa^2)^{s-1} \big\{ u \big[ \Pi_+ , \GG_h \big( |u|^2 \big) \big] \cj{u}\lax_uf \big\} \\
&- \frac{i\sin(\pi s)}{\pi} \int_0^\infty \lambda^{s-1} (\lax_u^2 + \kappa^2+\lambda)^{-1} \lax_u \Big\{ u \big[ \Pi_+ , \GG_h  \big( |u|^2 \big) \big] \cj{u} \tfrac{\lax_u^2+\kappa^2}{\lax^2_u+\kappa^2+\lambda} f \big] \Big\}\,d\lambda \\
&- \frac{i\sin(\pi s)}{\pi} \int_0^\infty \lambda^{s-1} (\lax_u^2 + \kappa^2+\lambda)^{-1} \Big\{ u \big[ \Pi_+ , \GG_h  \big( |u|^2 \big) \big] \cj{u} \lax_u \tfrac{\lax_u^2+\kappa^2}{\lax^2_u+\kappa^2+\lambda} f \big] \Big\}\,d\lambda .
\end{align*}

Combing this with~\eqref{INLSP}, we arrive at
\begin{align}
\tfrac{d}{dt} F\big( &\lax_u^2 + \kappa^2 \big) 
= \tfrac{d}{dt} \big\langle u , \big( \lax_u^2 + \kappa^2 \big)^{s} u\big\rangle 
\nonumber \\
&= {-2} \beta\Im\big\langle u \GG_h\big( |u|^2 \big) , \big( \lax_u^2 + \kappa^2 \big)^{s} u\big\rangle 
\label{ap1}\\
&\phantom{=} + i \big\langle u, (\lax_u^2 + \kappa^2)^{s-1} \lax_u \big\{ u \big[ \Pi_+, \GG_h \big( |u|^2 \big) \big] |u|^2 \big\} \big\rangle 
\label{ap5} \\
&\phantom{=} + i \big\langle u, (\lax_u^2 + \kappa^2)^{s-1} \big\{ u\big[ \Pi_+ ,\GG_h \big( |u|^2 \big)\big] \cj{u}\lax_uu \big\} \big\rangle
\label{ap4} \\
&\phantom{=}- \frac{i\sin(\pi s)}{\pi} \int_0^\infty \lambda^{s-1} \big\langle u , (\lax_u^2 + \kappa^2+\lambda)^{-1} \lax_u \big\{ u \big[ \Pi_+ , \GG_h  \big( |u|^2 \big) \big] \cj{u} \tfrac{\lax_u^2+\kappa^2}{\lax^2_u+\kappa^2+\lambda} u \big\} \big\rangle\,d\lambda 
\label{ap3} \\
&\phantom{=}- \frac{i\sin(\pi s)}{\pi} \int_0^\infty \lambda^{s-1} \big\langle u , (\lax_u^2 + \kappa^2+\lambda)^{-1} \big\{ u \big[ \Pi_+ , \GG_h  \big( |u|^2 \big) \big] \cj{u} \lax_u \tfrac{\lax_u^2+\kappa^2}{\lax^2_u+\kappa^2+\lambda} u \big\} \big\rangle\,d\lambda .
\label{ap2} 
\end{align}

We will estimate the contributions from \eqref{ap1}--\eqref{ap2} individually.
First,  for $\sigma\in\R$, we define the norms
\begin{equation*}
\| f\|_{H^\sigma_\kappa}^2 := \| (\lax_0^2+\kappa^2)^{\frac{\sigma}{2}} f \|_{L^2}^2 = \sum_{\xi} |\ft{f}(\xi)|^2 (\xi^2+\kappa^2)^{\sigma} .
\end{equation*}
The $H^{\s}_{\kk}$-norms are particularly useful because~\eqref{mono T} and Loewner's theorem imply that
\begin{equation*}
\big\| (\lax^2_u + \kappa^2)^{\frac{\sigma}{2}} f \big\|_{L^2}^2
= \big\langle f, (\lax^2_u + \kappa^2)^{\sigma} f \big\rangle 
\sim
\| f\|_{H^{\sigma}_\kappa}^2
\end{equation*}
uniformly for $-1\leq \sigma\leq 1$ and $\kappa$ satisfying~\eqref{k0 T}.  These norms are equivalent to the usual $H^\sigma$-norm for fixed $\kappa\geq 1$ 
and they enjoy similar estimates, such as 
\begin{equation}
\| fg \|_{H^{s}_\kappa} \lesssim \|f\|_{H^{\s
}} \|g\|_{H^s_\kappa} 
\qquad
\text{for}
\qquad 
\s \ge \max(s ,  \tfrac12+), 
\label{Hsk 1}
\end{equation}
where the constants are uniform in $\kk \ge 1$.

We begin with \eqref{ap1}.  Using \eqref{Hsk 1} and \eqref{G smooth}, we bound
\begin{align*}
|\eqref{ap1}| 
\lesssim \big\| u \GG_h\big( |u|^2 \big) \big\|_{H^{s}_\kappa}  \| u \|_{H^{s}_\kappa} 
\lesssim \big\| \GG_h\big( |u|^2 \big) \big\|_{H^{s+1}}  \| u \|_{H^{s}_\kappa}^2
\lesssim 
\| u \|_{L^2}^2 \| u \|_{H^{s}_\kappa}^2 .
\end{align*}

Next, we turn to~\eqref{ap5}.  Using \eqref{G smooth 2}, we find 
\begin{align}
\big\|  u \big[ \Pi_+, \GG_h \big( |u|^2 \big) \big] \cj{u} f \big\|_{L^2}
\lesssim \| u\|_{L^2} \big\| \big[ \Pi_+, \GG_h \big( |u|^2 \big) \big] \cj{u} f \big\|_{L^\infty} 
&\lesssim \|u\|_{L^2}^3 \| \cj{u} f \|_{H^{-1}} 
\nonumber 
\\
&\lesssim \|u\|_{L^2}^4 \|f\|_{L^2} .
\label{G smooth 3}
\end{align}

\noi
Employing this with $f=u$, we obtain
\begin{align*}
|\eqref{ap5}|
&\lesssim \|u\|_{H^{s}_\kappa}  \big\| (\lax_u^2+\kappa^2)^{\frac{s}{2}-1} \lax_u \big\{ u \big[ \Pi_+, \GG_h \big( |u|^2 \big) \big] |u|^2 \big\} \big\|_{L^2} \\
&\lesssim \|u\|_{H^{s}_\kappa}  \big\| (\lax_u^2+\kappa^2)^{\frac{s}{2}-\frac12} \big\{  u \big[ \Pi_+, \GG_h \big( |u|^2 \big) \big] |u|^2 \big\} \big\|_{L^2} \\
&\lesssim \|u\|_{H^{s}_\kappa}  \big\| u \big[ \Pi_+, \GG_h \big( |u|^2 \big) \big] |u|^2 \big\|_{L^2} \\
&\lesssim \|u\|_{L^2}^5 \|u\|_{H^s_\kappa} \\
&\lesssim \|u\|_{L^2}^4 \|u\|_{H^s_\kappa}^2 .
\end{align*}

Let us skip now to~\eqref{ap3}, as it also uses~\eqref{G smooth 3}.  Applying \eqref{G smooth 3} to $f=\tfrac{\lax_u^2+\kappa^2}{\lax^2_u+\kappa^2+\lambda} u$ and following an argument parallel to that for \eqref{ap5}, we obtain
\begin{align*}
|\eqref{ap3}|
&\lesssim \|u\|_{H^s_\kappa} \int_0^\infty \lambda^{s-1} \langle \lambda \rangle^{-\frac12-\frac{s}{2}} \big\| (\lax_u^2 + \kappa^2)^{-\frac{1}{2}} \lax_u \big\{ u \big[ \Pi_+ , \GG_h  \big( |u|^2 \big) \big] \cj{u} \tfrac{\lax_u^2+\kappa^2}{\lax^2_u+\kappa^2+\lambda} u \big\} \big\|_{L^2} \,d\lambda \\
&\lesssim \|u\|_{H^s_\kappa} \int_0^\infty \lambda^{s-1} \langle \lambda \rangle^{-\frac12-\frac{s}{2}} \big\| u \big[ \Pi_+ , \GG_h  \big( |u|^2 \big) \big] \cj{u} \tfrac{\lax_u^2+\kappa^2}{\lax^2_u+\kappa^2+\lambda} u \big\|_{L^2} \,d\lambda \\
&\lesssim \|u\|_{L^2}^4 \|u\|_{H^s_\kappa} \int_0^\infty \lambda^{s-1} \langle \lambda \rangle^{-\frac12-\frac{s}{2}} \big\| \tfrac{\lax_u^2+\kappa^2}{\lax^2_u+\kappa^2+\lambda} u \big\|_{L^2} \,d\lambda \\
&\lesssim \|u\|_{L^2}^5 \|u\|_{H^s_\kappa} \bigg( \int_0^1 \lambda^{s-1}\,d\lambda + \int_1^\infty \lambda^{\frac{s}{2}-\frac32} \,d\lambda \bigg) \\
&\lesssim \|u\|_{L^2}^4 \|u\|_{H^{s}_\kappa}^2 .
\end{align*}

It remains to estimate~\eqref{ap4} and~\eqref{ap2}.  Using \eqref{uniqueL2} and Sobolev embedding, we bound
\begin{align}
\| \cj{u}\lax_uf \|_{H^{-1}}
&\leq \| \cj{u}\dx f \|_{H^{-1}} + \big\| |u|^2 \Pi_+ \big( \cj{u}f \big) \big\|_{H^{-1}} 
\nonumber \\
&\lesssim \| u\|_{H^s} \|f\|_{H^s} + \| u\|_{L^4}^3 \|f\|_{L^4}
\nonumber \\
&\lesssim \| u\|_{H^s} \|f\|_{H^{s}} + \| u\|_{H^{\frac14}}^3 \|f\|_{H^{\frac14}}
\nonumber \\
&\lesssim \| u\|_{H^s} \|f\|_{H^{s}} + \| u\|_{L^2}^{3-\frac{3}{4s}} \|u\|_{H^s}^{\frac{3}{4s}} \|f\|_{L^2}^{1-\frac{1}{4s}} \|f\|_{H^s}^{\frac{1}{4s}} .
\label{ap42}
\end{align}

To estimate \eqref{ap4}, we apply \eqref{ap42} with $f=u$, to obtain
\begin{equation*}
\| \cj{u}\lax_uu \|_{H^{-1}}
\lesssim  \big( 1 + \| u\|_{L^2}^2 \big) \| u\|_{H^s_\kappa}^2 
\end{equation*}
uniformly in $\kappa\geq 1$.  Combining this with \eqref{G smooth 2}, we obtain
\begin{align*}
|\eqref{ap4}|
&\lesssim \|u\|_{L^2} \big\| (\lax_u^2 + \kappa^2)^{s-1} \big\{ u\big[ \Pi_+ ,\GG_h \big( |u|^2 \big)\big] \cj{u}\lax_uu \big\} \big\|_{L^2} \\
&\lesssim \|u\|_{L^2} \big\| u\big[ \Pi_+ ,\GG_h \big( |u|^2 \big)\big] \cj{u}\lax_uu \big\|_{L^2} \\
&\lesssim \|u\|_{L^2}^2 \big\| \big[ \Pi_+ ,\GG_h \big( |u|^2 \big)\big] \cj{u}\lax_uu \big\|_{L^\infty} \\
&\lesssim \|u\|_{L^2}^4 \| \cj{u}\lax_uu \|_{H^{-1}} \\
&\lesssim \big( 1 + \| u\|_{L^2}^2 \big) \|u\|_{L^2}^4 \| u\|_{H^s_\kappa}^2 .
\end{align*}

For \eqref{ap2}, we will take $f=\tfrac{\lax_u^2+\kappa^2}{\lax^2_u+\kappa^2+\lambda} u$ in \eqref{ap42}.  Note that
\begin{equation}
\big\| \tfrac{\lax_u^2+\kappa^2}{\lax^2_u+\kappa^2+\lambda} u \big\|_{H^\sigma_\kappa}
\lesssim \big\| \tfrac{\lax_u^2+\kappa^2}{\lax^2_u+\kappa^2+\lambda} \big( \lax_u^2 + \kappa^2 \big)^{\frac{\sigma}{2}} u \big\|_{L^2}
\lesssim \big\| \big( \lax_u^2 + \kappa^2 \big)^{\frac{\sigma}{2}} u \big\|_{L^2}
\lesssim \|u\|_{H^\sigma_\kappa} 
\notag
\end{equation}
for any $-1\leq \sigma\leq 1$, and so \eqref{ap42} yields
\begin{equation*}
\| \cj{u}\lax_u\tfrac{\lax_u^2+\kappa^2}{\lax^2_u+\kappa^2+\lambda} u \|_{H^{-1}}
\lesssim  \big( 1 + \| u\|_{L^2}^2 \big) \| u\|_{H^s_\kappa}^2 
\end{equation*}
uniformly for $\kappa\geq 1$.  Following an argument parallel to that in the previous paragraph for~\eqref{ap3}, we see that
\begin{align*}
|\eqref{ap2}|
&\lesssim \|u\|_{L^2} \int_0^\infty \lambda^{s-1} \langle \lambda\rangle^{-1} \big\| u \big[ \Pi_+ , \GG_h  \big( |u|^2 \big) \big] \cj{u} \lax_u \tfrac{\lax_u^2+\kappa^2}{\lax^2_u+\kappa^2+\lambda} u \big\|_{L^2} \,d\lambda \\
&\lesssim \|u\|_{L^2}^2 \int_0^\infty \lambda^{s-1} \langle \lambda\rangle^{-1} \big\| \big[ \Pi_+ , \GG_h  \big( |u|^2 \big) \big] \cj{u} \lax_u \tfrac{\lax_u^2+\kappa^2}{\lax^2_u+\kappa^2+\lambda} u \big\|_{L^\infty} \,d\lambda \\
&\lesssim \|u\|_{L^2}^4 \int_0^\infty \lambda^{s-1} \langle \lambda\rangle^{-1} \| \cj{u}\lax_u\tfrac{\lax_u^2+\kappa^2}{\lax^2_u+\kappa^2+\lambda} u \|_{H^{-1}} \,d\lambda \\
&\lesssim \big( 1+ \| u\|_{L^2}^2 \big) \|u\|_{L^2}^4  \| u\|_{H^s_\kappa}^2 \bigg( \int_0^1 \lambda^{s-1}\, d\lambda + \int_1^\infty \lambda^{s-2}\,d\lambda \bigg) \\
&\lesssim \big( 1 + \| u\|_{L^2}^2 \big) \|u\|_{L^2}^4  \| u\|_{H^s_\kappa}^2 , 
\end{align*}
since $\frac12 \le s <1$. 

Collecting the previous estimates, we obtain
\begin{equation*}
\tfrac{d}{dt} F(\lax_{u(t)}^2+\kappa^2) 
\le 
K\big( \| u(t)\|_{L^2} \big) 
F(\lax_{u(t)}^2+\kappa^2) 
\end{equation*}
for some polynomial $K$ with no constant term. Note that $K(\|u(t)\|_{L^2})=K(\|u(0)\|_{L^2})=K$ is constant in time, by conservation of the $L^2$-norm.  Combining the Gronwall inequality with \eqref{equiv}, this yields
\begin{align}
\|u(t)\|_{H^s}^2
&\leq \big\langle u(t), \big( \lax_0^2 + \kappa^2 \big)^s u(t) \big\rangle
\nonumber\\
&\leq C F(\lax_{u(t)}^2+\kappa^2) 
\nonumber\\
&\leq Ce^{K|t|} F(\lax_{u(0)}^2+\kappa^2) 
\nonumber\\
&\leq C^2e^{K|t|} \big\langle u(0), \big( \lax_0^2 + \kappa^2 \big)^s u(0) \big\rangle
\nonumber\\
&\leq C^2e^{K|t|} \big( \|u(0)\|_{H^s}^2 + \kappa^{2s} \|u(0)\|_{L^2}^2 \big) .
\label{gronwall}
\end{align}

 Now we will prove \eqref{ap T}
in the case $T=1$ using a bootstrap argument. Later in the argument we will iterate this to cover the full interval $[0,T]$, for $T\geq 1$.
Given $u(0)\in H^s(\T)$ with $u(0)\neq 0$, set
\begin{equation*}
\kappa = C \max\big( \|u(0)\|_{L^2} ,  4C^2 e^{K} \|u(0)\|_{H^s}^2 \big)^{\frac{1}{2s}} .
\end{equation*}
Here, $C \geq 1$ is larger than the constants from \eqref{k0 T} and \eqref{gronwall}.
Also, let $0< r \le 1$, so that 
\begin{equation}
\|u(0)\|_{H^{\frac14}}^4 \leq \|u(0)\|_{L^2}^{4-\frac1s} \|u(0)\|_{H^s}^{\frac1s} \leq \|u(0)\|_{H^s}^{\frac{1}{s}} .
\label{k0 T 2}
\end{equation}
For this $\kappa$ and any time interval $[0,t_0]\subset [0,1]$ on which
\begin{equation}
\| u(t) \|_{H^s} \leq 2C e^{\frac{K}{2}} \| u(0) \|_{H^s} ,
\label{boot 3}
\end{equation}
the condition \eqref{k0 T 2} implies that \eqref{k0 T} holds, and thus so too does the monotonicity relation~\eqref{mono T}.  Then, \eqref{gronwall} yields
\begin{align*}
\|u(t)\|_{H^s}^2 
&\leq C^2e^{K} \|u(0)\|_{H^s}^2 + C^{2+2s} e^{K} \max\big( \|u(0)\|_{L^2},4C^2 e^{K} \|u(0)\|_{H^s}^2 \big) \|u(0)\|_{L^2}^2 \\
&\leq C^2e^{K} \big\{ 1 + C^{2s}\|u(0)\|_{L^2} \max\big( 1,4C^2 e^{K} \|u(0)\|_{L^2} \big) \big\} \|u(0)\|_{H^s}^2 .
\end{align*}
Taking $r$ possibly smaller so that $ r^2 \ll C^{-2-2s} e^{-K}$, we deduce
\begin{equation}
\|u(t)\|_{H^s} \leq \tfrac{3}{2} C e^{\frac{K}{2}} \| u(0) \|_{H^s} .
\label{boot 4}
\end{equation}
Comparing this with~\eqref{boot 3}, we conclude that~\eqref{boot 4} holds for all $t\in [0,1]$.

Finally, we iterate this argument to cover an arbitrary time interval $[0,T]$.  We choose $r^2 \ll C^{-2-2s} e^{-K}$ for the time interval $[0,1]$ as before, depending only on $\|u(0)\|_{L^2}$.  In turn, this constant $r$ guarantees that \eqref{boot 4} holds, and so the $H^s$-norm can grow to at most 
\begin{equation*}
\|u(1)\|_{H^s} \leq R \| u(0)\|_{H^s} \quad\text{with}\quad R = \tfrac{3}{2} C e^{\frac{K}{2}}.
\end{equation*}
Turning to the time interval $[1,2]$, we now take
\begin{equation*}
\kappa = C \max\big( \|u(1)\|_{L^2} ,  4C^2 e^{K} \|u(1)\|_{H^s}^2 \big)^{\frac{1}{2s}} ,
\end{equation*}
and deduce 
\begin{align*}
\|u(t)\|_{H^s}^2 \leq C^2e^{K} \big\{ 1 + C^{2s}\|u(1)\|_{L^2} \max\big( 1,4C^2 e^{K} \|u(1)\|_{L^2} \big) \big\} \|u(1)\|_{H^s}^2 .
\end{align*}
By conservation of the $L^2$-norm, the value of $\|u(1)\|_{L^2}$ is still bounded by $r$, and so precisely the same choice of $r^2 \ll C^{-2-2s} e^{-K}$ closes the bootstrap argument on $[1,2]$.  In this way, the argument may be iterated indefinitely.  After $N=\lceil T \rceil\geq 1$ iterations, we conclude that the $H^s$-norm can grow to at most $B= R^NA$, completing the proof. 
\end{proof}

\begin{ackno}\rm 
J.F.~was partially supported by the ARC project FT230100588.
A.C.~was partially supported by CNRS-INSMI through a
grant ``PEPS Jeunes chercheurs et jeunes chercheuses 2025".
A.C.~is thankful to Monash University for their hospitality during her visit in August 2025, during which this project was started. 
T.L.~was supported by an AMS-Simons Travel Grant.

\end{ackno}

\end{document}